\documentclass{article}

\usepackage[section]{placeins}
\usepackage{xcolor}
\usepackage{lipsum}
\usepackage{amsfonts}
\usepackage{graphicx,graphics}
\usepackage{epstopdf}
\usepackage{algorithmic}
 \usepackage{ulem}
 \usepackage{dsfont}
  \usepackage{amsmath, amsthm}
  \usepackage{url}

\ifpdf
  \DeclareGraphicsExtensions{.eps,.pdf,.png,.jpg}
\else
  \DeclareGraphicsExtensions{.eps}
\fi
\usepackage{caption}
\usepackage{placeins}
\newtheorem{theorem}{Theorem}[subsection]
\newtheorem{lemma}[theorem]{Lemma}

\newtheorem{assumption}{Assumptions}[section]

\newcommand{\TheTitle}{Kinetic/Fluid micro-macro numerical scheme for a two component gas mixtures} 
\newcommand{\TheAuthors}{A. Crestetto, C. Klingenberg, and M. Pirner}

\title{{\TheTitle}}

\author{{\TheAuthors}}

\usepackage{amsopn}

\ifpdf
\fi

\begin{document}
\date{}
\maketitle
\footnote{We certify that the general content of the manuscript, in whole or in part, is not submitted, accepted or published elsewhere, including conference proceedings.}
\begin{abstract}
This work is devoted to the numerical simulation of the \textcolor{black}{BGK} equation for two species in the fluid limit using a particle method. Thus, we are interested in a \textcolor{black}{gas mixture consisting of two species without chemical reactions} assuming that the number of particles of each species remains constant.
We consider the kinetic two species model proposed by Klingenberg, Pirner and Puppo in \cite{Pirner}, which separates the intra and interspecies collisions. \textcolor{black}{We want to study numerically the influence of the two relaxation term, one corresponding to intra, the other to interspecies collisions. For this, we use the method of micro-macro decomposition. First, we derive an equivalent model} based on the micro-macro decomposition (see Bennoune, Lemou and Mieussens \cite{MicroMacro2007} and Crestetto, Crouseilles and Lemou \cite{MicroMacro2013}). The kinetic micro part is solved by a particle method, whereas the fluid macro part is discretized by a standard finite volume scheme. Main advantages of this approach are: (i) the noise inherent to the particle method is reduced compared to a standard (without micro-macro decomposition) particle method, (ii) the computational cost of the method is reduced in the fluid limit since a small number of particles is then sufficient.
\end{abstract}

\\

\textbf{AMS subject classsification:} 
65M75, 82C40, 82D10, 35B40.

\section{Introduction}
We want to model a \textcolor{black}{gas mixture} consisting of two species. The kinetic description of a plasma is based on the \textcolor{black}{BGK} equation. In \cite{MicroMacro2013}, Crestetto, Crouseilles and Lemou developed a numerical simulation of the Vlasov-BGK equation in the fluid limit using particles. They consider a Vlasov-BGK equation for the electrons and treat the ions as a background charge. In \cite{MicroMacro2013} a micro-macro decomposition is used as in \cite{MicroMacro2007} where asymptotic preserving schemes have been derived in the fluid limit. In \cite{MicroMacro2013}, the approach in \cite{MicroMacro2007} is modified by using a particle approximation for the kinetic part, the fluid part being always discretized by standard finite volume schemes. Other approaches where kinetic description of one species is written in a micro-macro decomposition can be seen in \cite{CL,dececco}.\\
In this paper, we want to model \textcolor{black}{two species by a system of two BGK} equations. Such a multi component kinetic description of the gas mixture has for example  importance \textcolor{black}{in modelling applications in air, since air is a gas mixture. We want to consider applications where the gas mixture is close to a fluid in some regions, but the kinetic description is mandatory in some other regions}. For this, we want to use the approach in \cite{MicroMacro2013}, since it has the following advantages: the presented scheme has a much less level of noise compared to the standard particle method and the computational cost of the micro-macro model is reduced in the fluid regime since a small number of particles is needed for the micro part. \\
From the modelling point of view, we want to describe this gas mixture using two distribution functions via the BGK equation with interaction terms on the right-hand side. 
BGK models give rise to efficient numerical computations, see for example \cite{Puppo_2007, Jin_2010,Dimarco_2014, MicroMacro2007, Dimarco, Bernard_2015, MicroMacro2013}. In the literature one can find two types of models for gas mixtures. \textcolor{black}{The Boltzmann equation for gas mixtures contains a sum of collision terms on the right-hand side. One type of BGK model for gas mixtures also has a sum of collision terms in the relaxation operator.} One example is the model of Klingenberg, Pirner and Puppo \cite{Pirner} which we will consider in this paper. It contains the often used models of Gross and Krook   \cite{gross_krook1956} and Hamel \cite{hamel1965} as special cases. 
The other type of model contains only one collision term on the right-hand side. Example of this is the well-known model of Andries, Aoki and Perthame in \cite{AndriesAokiPerthame2002}. \\
In this paper we are interested in the first type of models, and use the model developed in \cite{Pirner}. In this type of model the two different types of interactions, interactions of a species with itself and interactions of a species with the other one, are kept separated. Therefore, we can see how these different types of interactions influence the trend to equilibrium. From the physical point of view, we expect two different types of trends to equilibrium. For example, if the collision frequencies of the particles of each species with itself are larger compared to the collision frequencies related to interspecies collisions, we expect that we first observe that the relaxation of the two distribution functions to its own equilibrium distribution is faster compared to the relaxation towards a common velocity and a common temperature. This effect is clearly seen in the model presented in \cite{Pirner} since the two types of interactions are separated.

The outline of the paper is as follows: In section \ref{sec:model} we present the model for a \textcolor{black}{gas mixture} consisting of two species  and write it in dimensionless form. In section \ref{sec:micromacro} we derive the micro-macro decomposition of the model presented in section \ref{sec:model}. In section \ref{sec:spacehomcase} we prove some convergence rates in the space-homogeneous case of the distribution function to a Maxwellian distribution and of the two velocities and temperatures to a common value which we will verify numerically later on.
In section \ref{sec:numapp}, we briefly present the numerical approximation, based on a particle method for the micro equation and a finite volume scheme for the macro one. In section \ref{sec:numresults}, we present some numerical examples. First, we verify numerically the convergence rates obtained in section \ref{sec:spacehomcase}. Then, in the general case, we are interested in the evolution in time of the system. We consider different possibilities for the values of the collision frequencies. 
 When the collision frequencies are very large we observe relaxations towards Maxwellian distributions. Finally, if we vary the relationships between the different collision frequencies, we observe a corresponding variation in the speed of relaxation towards Maxwellians and the relaxation towards a common value of the mean velocities and temperatures. Finally, section \ref{sec:conclusion} presents a brief conclusion.

\section{The two-species model}\label{sec:model}
In this section we present in 1D the BGK model for a mixture of two species developed in \cite{Pirner} and mention its fundamental properties like the conservation properties. Then, we present its dimensionless form.

\subsection{1D BGK model for a mixture of two species}\label{ssec:model}
We consider a \textcolor{black}{gas mixture consisting of two species} denoted by the index $1$ and 2. Thus, our kinetic model has two distribution functions $f_1(x,v,t)> 0$ and $f_2(x,v,t) > 0$ where $x\in [0,L_x], L_x>0$,  
$v\in \mathbb{R}$ are the phase space variables and $t\geq 0$ the time.   

 Furthermore, for any $f_1,f_2: [0,L_x] \times \mathbb{R} \times \mathbb{R}^+_0 \rightarrow \mathbb{R}^+$ with $(1+|v|^2)f_1,$ \\$(1+|v|^2)f_2 \in L^1(\mathbb{R})$, we relate the distribution functions to  macroscopic quantities by mean-values of $f_k$, $k=1,2$
\begin{align}
\int f_k(v) \begin{pmatrix}
1 \\ v  \\ m_k |v-u_k|^2 
 \end{pmatrix} 
dv =: \begin{pmatrix}
n_k \\ n_k u_k \\   n_k T_k 
\end{pmatrix} , \quad k=1,2,
\label{moments}
\end{align} 
where $m_k$ is the mass,  $n_k$ the number density, $u_k$ the mean velocity and $T_k$ the mean temperature of species $k$, $k=1,2$. Note that in this paper we shall write $T_k$ instead of $k_B T_k$, where $k_B$ is Boltzmann's constant.

We want to model the time evolution of the distribution functions by BGK equations. Each distribution function is determined by one BGK equation to describe its time evolution. The two equations are coupled through a term which describes the interaction of the two species.  We consider binary interactions. So the particles of one species can interact with either themselves or with particles of the other species. In the model this is accounted for introducing two interaction terms in both equations. 
Here, we choose the collision terms as BGK operators, so that the model writes 
\begin{align} \begin{split} \label{BGK}
\partial_t f_1 + v\partial_x  f_1 + \frac{F_1}{m_1} ~\partial_v f_1 &=\nu_{11} n_1 (M_1 - f_1) + \nu_{12} n_2 (M_{12}- f_1), 
\\ 
\partial_t f_2 + v\partial_x   f_2  + \frac{F_2}{m_2} ~\partial_v f_2 &= \nu_{22} n_2 (M_2 - f_2) + \nu_{21} n_1 (M_{21}- f_2),
\end{split}
\end{align}
with the mean-field or external forces $F_1=F_1(x,t)$ and $F_2=F_2(x,t)$ and the Maxwell distributions
\begin{align} 
\begin{split}
M_k(x,v,t) &= \frac{n_k}{\sqrt{2 \pi \frac{T_k}{m_k}} }  \exp({- \frac{|v-u_k|^2}{2 \frac{T_k}{m_k}}}),
\quad k=1,2,
\\
M_{kj}(x,v,t) &= \frac{n_{kj}}{\sqrt{2 \pi \frac{T_{kj}}{m_k}} }  \exp({- \frac{|v-u_{kj}|^2}{2 \frac{T_{kj}}{m_k}}}), \quad k,j=1,2, k \neq j,
\end{split}
\label{BGKmix}
\end{align}
where $\nu_{11} n_1$ and $\nu_{22} n_2$ are the collision frequencies of the particles of each species with itself, while $\nu_{12} n_2$ and $\nu_{21} n_1$ 
are related to interspecies collisions. 
To be flexible in choosing the relationship between the collision frequencies, we now assume the relationship
\begin{equation} 
\nu_{12}=\varepsilon \nu_{21}, \quad 
\nu_{22} = \beta_2 \nu_{21} = \frac{\beta_2}{\varepsilon} \nu_{12}, \quad 0 < \varepsilon \leq 1, ~\beta_1, \beta_2 >0.
\label{coll}
\end{equation}
The restriction on $\varepsilon$ is without loss of generality. If $\varepsilon >1$, exchange the notation $1$ and $2$ and choose $\frac{1}{\varepsilon}.$
 In addition, we take into account an acceleration due to interactions using a  mean-field or a given external forces $F_1, F_2$. \textcolor{black}{In the following we will omit the forces $F_1$ and $F_2$ for simplicity, but the following work can be extended to the equations with forces in a straightforward way.}

The functions $f_k$  are submitted to the following periodic condition
\begin{align*}
f_k (0,v,t) &= f_k (L_x,v,t),  \quad \text{ for every} \quad  v \in \mathbb{R}, t \geq 0, 
\end{align*}
together with an initial condition
$$ f_k (x,v,0) = f_k^0 (x, v), \quad \text{for every} \quad x \in [0, L_x], v \in \mathbb{R}.$$

The Maxwell distributions $M_1$ and $M_2$ in \eqref{BGKmix} have the same moments as $f_1$ and $f_2$ respectively. With this choice, we guarantee the conservation of mass, momentum and energy in interactions of one species with itself (see section 2.2 in \cite{Pirner}).
The remaining parameters $n_{12}, n_{21}, u_{12}, u_{21}, T_{12}$ and $T_{21}$ will be determined using conservation of total momentum and energy, together with some symmetry considerations.

If we assume that 
\begin{align} 
&n_{12}=n_1 \quad \text{and} \quad n_{21}=n_2,  \label{density} \\
&u_{12}= \delta u_1 + (1- \delta) u_2, \quad \delta \in \mathbb{R}, \label{convexvel}\\
&T_{12} =  \alpha T_1 + ( 1 - \alpha) T_2 + \gamma |u_1 - u_2 | ^2,  \quad 0 \leq \alpha \leq 1, \gamma \geq 0, \label{contemp}
\end{align}
we have conservation of the number of particles, of total momentum and total energy provided that
\begin{align}
u_{21}&=u_2 - \frac{m_1}{m_2} \varepsilon (1- \delta ) (u_2 - u_1), \quad \text{and}
\label{veloc} \\
\begin{split}
T_{21} &=\left[  \varepsilon m_1 (1- \delta) \left( \frac{m_1}{m_2} \varepsilon ( \delta - 1) + \delta +1 \right) - \varepsilon \gamma \right] |u_1 - u_2|^2 \\&+ \varepsilon ( 1 - \alpha ) T_1 + ( 1- \varepsilon ( 1 - \alpha)) T_2,
\label{temp}
\end{split}
\end{align}
 see theorem 2.1, theorem 2.2 and theorem 2.3 in \cite{Pirner}. 
 
In order to ensure the positivity of all temperatures, we need to impose restrictions on $\delta$ and $\gamma$ given by
 \begin{align}
&0 \leq \gamma  \leq m_1 (1-\delta) \left[(1 + \frac{m_1}{m_2} \varepsilon ) \delta + 1 - \frac{m_1}{m_2} \varepsilon \right], \quad \text{and}
 \label{gamma} \\
 &\frac{ \frac{m_1}{m_2}\varepsilon - 1}{1+\frac{m_1}{m_2}\varepsilon} \leq  \delta \leq 1,
\label{gammapos}
\end{align}
see theorem 2.5 in \cite{Pirner}.

\subsection{Dimensionless form}

We want to write the BGK model presented in subsection \ref{ssec:model} in dimensionless form \textcolor{black}{in order to do the numerical experiments with dimensionless quantities}. The principle of non-dimensionalization can also be found in chapter 2.2.1 in \cite{Laure} for the Boltzmann equation and in \cite{MHDequations} for macroscopic equations. First, we define dimensionless variables of the time $t \in \mathbb{R}^+_0$, the length $x \in [0,L_x]$, the velocity $v \in \mathbb{R}$, the distribution functions $f_1, f_2$, the number densities $n_1, n_2$, the mean velocities $u_1, u_2$,
 the temperatures $T_1,T_2$ 
 and of the collision frequency \textcolor{black}{per density } $\nu_{12}$. Then, dimensionless variables of the other  collision frequencies $\nu_{11}, \nu_{22},  \nu_{21}$ can be derived by using the relationships \eqref{coll}. We start with choosing typical scales denoted by a bar. 
 $$t'= t/\bar{t},~~~x'=x/\bar{x},~~~v'= v/\bar{v},$$
 $$f_1'(x', v', t')= \frac{ \bar{v}}{\textcolor{black}{\bar{n}_1}} f_1(x,v,t),~~~f_2'(x', v', t')= \frac{ \bar{v}}{\textcolor{black}{\bar{n}_2}} f_2(x,v,t),$$
 where \textcolor{black}{$\bar{n}_1$ is the typical order of magnitude of the density} of species 1 and \textcolor{black}{$\bar{n}_2$ the typical order of magnitude of the density} of the species 2 in the volume $\bar{x}$. 
 Further, we choose
$$n_1'=  n_1/\bar{n}_1,~~~ n_2'= n_2 /\bar{n}_2, $$ 
$$u_1' = u_1/\bar{u}_1, ~~~ u_2'= u_2/\bar{u}_2, ~~~ \bar{u}_2 = \bar{u}_1 = \bar{v},$$
$$T_1' = T_1/\bar{T}_1, ~~~T_2'= T_2/\bar{T}_2, ~~~ \bar{T}_2 = \bar{T}_1 = m_1 \bar{v}^2,$$
$$\nu_{ie}' = \nu_{ie}/\bar{\nu}_{ie}.$$
\textcolor{black}{We want to make the following assumptions on the gas mixture regime. 
\begin{assumption}
We assume
$$ \bar{n}_1=\bar{n}_2=: \bar{n}, \quad \bar{u}_1 = \bar{u}_2 = \bar{v}, \quad \bar{T}:= \bar{T}_2 = \bar{T}_1 = m_1 \bar{v}^2, 
$$
and the assumptions on the collision frequencies \eqref{coll}.
\label{ass}
\end{assumption}}
Now, we want to write equations \eqref{BGK} in dimensionless variables. We start with the Maxwellians \eqref{BGKmix} and with \eqref{convexvel}-\eqref{temp}.
We replace the macroscopic quantities $n_1, u_1$ and $T_1$ in $M_1$  by their dimensionless expressions and obtain 
\begin{align}
M_1= \frac{n'_1 \bar{n}}{ \sqrt{2 \pi \frac{\bar{T}_1 T'_1}{m_1}}} \exp ( - \frac{|v' \bar{v} - u'_1 \bar{u}_1 |^2 m_1}{2 T'_1 \bar{T}_1})
\end{align}
\textcolor{black}{by using the first assumption of assumptions \ref{ass}. By the third assumption of assumptions \ref{ass}}, we obtain
\begin{align}
M_1=\frac{\bar{n}}{\bar{v}} \frac{n'_1}{\sqrt{2\pi T'_1}} \exp( - \frac{|v' - u'_1|^2}{2 T'_1})=: \frac{\bar{n}}{\bar{v}} M'_1.
\label{Mi_dim}
\end{align}
In the Maxwellian $M_2$ \textcolor{black}{we again assume the first and third assumption in assumptions \ref{ass} } and obtain in the same way as for $M_1$
\begin{align}
M_e= \frac{\bar{n}}{\bar{v}} \left( \frac{m_e}{m_i} \right)^{\frac{1}{2}} \frac{n'_e}{\sqrt{2\pi T'_e}} \exp(- \frac{|v'- u'_e|^2}{2T'_e} \frac{m_e}{m_i}) =: \frac{\bar{n}}{\bar{v}} M'_e.
\end{align}
Now, we consider the Maxwellian $M_{12}$ in \eqref{BGKmix}, its velocity $u_{12}$ in  \eqref{convexvel} and its temperature $T_{12}$ in \eqref{contemp}. 
\textcolor{black}{Now, we use the first, second and third assumption of assumptions \ref{ass}} and obtain
\begin{align}
\begin{split}
u_{12}&= \delta u'_1 \bar{u}_1 +(1- \delta) u'_2 \bar{u}_2 = (\delta u'_1 +(1- \delta) u'_2 ) \bar{v} =: \bar{v} u'_{12}, \\
T_{12} &= \alpha T'_1 \bar{T}_1 + (1- \alpha) T'_2 \bar{T}_2 + \gamma |\bar{v}|^2 |u'_1 - u'_2|^2 \\&=  m_1 |\bar{v}|^2 [\alpha T'_1 +(1-\alpha) T'_2+  \frac{\gamma}{m_1} |u'_1 - u'_2|^2 ]=:  |\bar{v}|^2 m_1 T'_{12},\\
M_{12} &= \frac{n'_1 \bar{n}}{\sqrt{2 \pi \bar{v}^2 T'_{12}}} \exp(-\frac{|v'-u'_{12}|^2}{2 T'_{12}}) =: \frac{\bar{n}}{\bar{v}} M'_{12}.
\end{split}
\label{Mie_dim}
\end{align}
With the same assumptions we obtain for $u_{21}$, $T_{21}$ and $M_{21}$ in a similar way the expressions
\begin{align*}
u_{21}&= [( 1- \frac{m_1}{m_2}\varepsilon(1-\delta)) u'_2 + \frac{m_1}{m_2} \varepsilon (1- \delta) u'_1] \bar{v} =: u'_{21} \bar{v}, \\
T_{21} &=[(1- \varepsilon(1- \alpha)) T'_2 + \varepsilon (1- \alpha) T'_1 ] \bar{T} \\&+ ( \varepsilon m_1 (1- \delta) ( \frac{m_1}{m_2} \varepsilon (\delta -1) + \delta +1) - \varepsilon \gamma)|u'_1 -u'_2|^2 |\bar{v}|^2\\ &= [(1- \varepsilon(1- \alpha)) T'_2 + \varepsilon (1- \alpha) T'_1 ]  |\bar{v}|^2 m_2 \frac{m_1}{m_2}\\& + ( \varepsilon m_1 (1- \delta) ( \frac{m_1}{m_2} \varepsilon (\delta -1) + \delta +1) - \varepsilon \gamma)|u'_1 -u'_2|^2 |\bar{v}|^2 =:  |\bar{v}|^2 m_2 \frac{m_1}{m_2} T'_{21}, \\
M_{21}&= \frac{\bar{n}}{\bar{v}} \frac{m_2}{m_1} \frac{n'_2}{\sqrt{2\pi T'_{21}}} \exp(- \frac{|v' - u'_{21}|^2}{2T'_{21}} \frac{m_2}{m_1})=: \frac{\bar{n}}{\bar{v}} M'_{21}.
\end{align*}
Now we replace all quantities in \eqref{BGK} by their non-dimensionalized expressions. For the left-hand side of the equation for the species 1 we obtain
\begin{align}
\begin{split}
\partial_t f_1 + v \partial_x f_1 
= \frac{1}{\bar{t}}\frac{\textcolor{black}{\bar{n}}}{ \bar{v}}\partial_{t'} f'_1 + \frac{1}{\bar{x}}\frac{\textcolor{black}{\bar{n}}}{ \bar{v}} \bar{v} v' \partial_{x'}  f'_1
\end{split}
\label{left}
\end{align}
and for the right-hand side using  \eqref{coll}, \eqref{Mi_dim} and \eqref{Mie_dim}, we get
\begin{align}
\begin{split}
\nu_{11} n_1 &(M_1 - f_1) + \nu_{12}  n_2 (M_{12}- f_1) = \nu_{12} \beta_1 n_1 (M_1 - f_1) + \nu_{12} n_2 (M_{12} - f_1)\\&=\beta_1 \bar{\nu}_{12}\frac{\textcolor{black}{\bar{n}^2}}{ \bar{v}}  \nu'_{12} n'_1 (M'_1 - f'_1) + \bar{\nu}_{12}\frac{\textcolor{black}{\bar{n}^2}}{ \bar{v}}\nu'_{12}  n'_2 (M'_{12}- f'_1) .
\end{split}
\label{right}
\end{align}
Multiplying by $\frac{\bar{t}\bar{v}}{\textcolor{black}{\bar{n}}}$ and dropping the primes in the variables leads to
\begin{align*}
\partial_{t} f_1 +\frac{\bar{t}\bar{v}}{\bar{x}}  v \partial_{x}  f_1
= \beta_1 \bar{\nu}_{12}\bar{t}~\textcolor{black}{\bar{n}} ~\nu_{12} n_1 (M_1 - f_1) + \bar{\nu}_{12}\bar{t}~\textcolor{black}{\bar{n}}~\nu_{12}  n_2 (M_{12}- f_1).
\end{align*}
In a similar way we obtain for the second species
\begin{align*}
\partial_{t} f_2 &+ \frac{\bar{t}\bar{v}}{\bar{x}} v \partial_{x}  f_2 
= \textcolor{black}{\frac{\beta_2}{\varepsilon}}\bar{\nu}_{12}\bar{t}~\textcolor{black}{\bar{n}}~ \nu_{12} n_2 \left( M_2 - f_2\right) + \frac{1}{\varepsilon}\bar{\nu}_{12}\bar{t}~\textcolor{black}{\bar{n}}~\nu_{12}  n_1 \left(  M_{21}- f_2 \right),
\end{align*}
and the non-dimensionalized Maxwellians given by 
\begin{align} 
\begin{split}
M_1(x,v,t) &= \frac{n_1}{\sqrt{2 \pi T_1}} \exp({- \frac{|v-u_1|^2}{2 T_1}}),
\\
M_2(x,v,t) &= \frac{n_2}{\sqrt{2 \pi T_2}} \left( \frac{m_2}{m_1} \right) ^{\frac{1}{2}} \exp({- \frac{|v-u_2|^2}{2 T_2}} \frac{m_2}{m_1}),
\\
M_{12}(x,v,t) &= \frac{n_{1}}{\sqrt{2 \pi T_{12}}} \exp({- \frac{|v-u_{12}|^2}{2 T_{12}}}),
\\
M_{21}(x,v,t) &= \frac{n_{2}}{\sqrt{2 \pi T_{21}}} \left( \frac{m_2}{m_1} \right) ^{\frac{1}{2}} \exp({- \frac{|v-u_{21}|^2}{2 T_{21}}} \frac{m_2}{m_1}),
\end{split}
\label{Max_dim}
\end{align}
with the non-dimensionalized macroscopic quantities 
\begin{align}
u_{12}&=  \delta u_1 +(1- \delta) u_2, \label{convexveldim}   \\
T_{12} &= \alpha T_1 +(1-\alpha) T_2+  \frac{\gamma}{m_1} |u_1 - u_2|^2, \label{contempdim}\\
u_{21}&= ( 1- \frac{m_1}{m_2}\varepsilon(1-\delta)) u_2 + \frac{m_1}{m_2} \varepsilon (1- \delta) u_1, \label{velocdim} \\
\begin{split}
T_{21} &= [(1- \varepsilon(1- \alpha)) T_2 + \varepsilon (1- \alpha) T_1 ]\\& + ( \varepsilon  (1- \delta) ( \frac{m_1}{m_2} \varepsilon (\delta -1) + \delta +1) - \varepsilon \frac{\gamma}{m_1})|u_1 -u_2|^2. \label{tempdim}
\end{split}
\end{align}
\\
Defining dimensionless parameters 
 \begin{align}
 \begin{split}
  A = \frac{\bar{t}\bar{v}}{\bar{x}}, \quad
 &\frac{1}{\varepsilon_1}=\beta_1 \bar{\nu}_{12}\bar{t}~\textcolor{black}{\bar{n}}, \quad \frac{1}{\tilde{\varepsilon}_1}= \bar{\nu}_{12}\bar{t}~\textcolor{black}{\bar{n}}, \quad  \frac{1}{\varepsilon_2}= \frac{\beta_2}{\varepsilon}\bar{\nu}_{12}\bar{t}~\textcolor{black}{\bar{n}}, \quad \frac{1}{\tilde{\varepsilon}_2}= \frac{1}{\varepsilon}\bar{\nu}_{12}\bar{t}~\textcolor{black}{\bar{n}},
 \end{split}
 \label{eps}
 \end{align}
 we get
 \begin{align} 
\begin{split}
\partial_t f_1 + A~ v\partial_x  f_1
= \frac{1}{\varepsilon_1}  \nu_{12} n_1 (M_1 - f_1) + \frac{1}{\tilde{\varepsilon}_1} \nu_{12} n_2 (  M_{12}- f_1),
\\
\partial_t f_2 + A~ v\partial_x  f_2
= \frac{1}{\varepsilon_2} \nu_{12} n_2 (M_2 - f_2) + \frac{1}{\tilde{\varepsilon}_2} \nu_{12} n_1 ( M_{21}- f_2).
\end{split}
\label{nondimBGK}
\end{align}
\textcolor{black}{In the sequel, parameters $\varepsilon_1$, $\varepsilon_2$, $\tilde{\varepsilon}_1$ and $\tilde{\varepsilon}_2$ are referred to as Knudsen numbers.} In addition, we want to write the moments \eqref{moments} in non-dimensionalized form. We can compute this in a similar way as for \eqref{BGK} and obtain after dropping the primes 
\begin{align}
\begin{split}
&\int f_k dv = n_k, \quad  \int v f_k dv = n_k u_k, \quad k=1,2, \\ \frac{1}{n_1} &\int |v-u_1|^2 f_1 dv = T_1, \quad \frac{m_2}{m_1} \frac{1}{n_2} \int |v-u_2|^2 f_2 dv = T_2.
\end{split}
\label{momentsdim}
\end{align}


\section{Micro-Macro decomposition}\label{sec:micromacro}

In this section, we derive the micro-macro model equivalent to \eqref{nondimBGK}.

First, we take the dimensionless equations \eqref{nondimBGK} and choose $A=1$.
 The choice $A=1$ means $\bar{v}= \frac{\bar{x}}{\bar{t}}$. 

Now, we propose to adapt the micro-macro decomposition presented in \cite{MicroMacro2007} and \cite{MicroMacro2013}. It is used for numerical methods to solve Boltzmann-like equations for mixtures to capture the right compressible Navier-Stokes dynamics at small Knudsen numbers. The idea is to write each distribution function as the sum of its own equilibrium part (verifying a fluid equation) and a rest (of kinetic-type). 
So, we decompose $f_1$ and $f_2$ as 
\begin{align}
f_1 = M_1 + g_{11}, \quad f_2 = M_2 + g_{22}.
\label{eq:mi_ma_decomp}
\end{align}

Let us introduce $m(v):=\begin{pmatrix} 1 \\ v \\ |v|^2\end{pmatrix}$ and the notation $\langle\cdot\rangle:=\int \cdot ~dv$.
Since $f_1$ and $M_1$ (resp. $f_2$ and $M_2$) have the same moments: $\langle m(v)f_1\rangle=\langle m(v)M_1\rangle$ (resp. $\langle m(v)f_2\rangle=\langle m(v)M_2\rangle$), then the moments of $g_{11}$ (resp. $g_{22}$) are zero:
\begin{align}
\int m(v) g_{11} dv = \int m(v) g_{22} dv = 0.
\label{eq:moments_of_g}
\end{align}

With this decomposition we get from equation \eqref{nondimBGK} of species 1 in dimensionless form
\begin{align}
\begin{split}
\partial_t M_1 + \partial_t g_{11} + v\partial_x   M_1 + v\partial_x   g_{11} 
=- \frac{1}{\varepsilon_1} \nu_{12} n_1  g_{11} + \frac{1}{\tilde{\varepsilon_1}} \nu_{12} n_2( M_{12} - M_1 -g_{11}),
\label{decomp}
\end{split}
\end{align}
and a similar equation for species 2.

Now we consider the Hilbert spaces $L^2_{M_k}= \lbrace \phi$ such that $\phi M_k^{-\frac{1}{2}}\in L^2(\mathbb{R})\rbrace$, $k=1,2$,  with the weighted inner product $\langle \phi \psi M_k^{-1} \rangle$. We consider the subspace $\mathcal{N}_k=$span $\lbrace  M_k , v M_k , |v|^2 M_k \rbrace$, $k=1,2$.  
Let $\Pi_{M_k}$ the orthogonal projection in $L^2_{M_k}$ on this subspace $\mathcal{N}_k$. This subspace has the orthonormal basis 
\begin{align*}
\tilde{B}_k = \lbrace \frac{1}{\sqrt{n_k}} M_k , \frac{(v-u_k)}{\sqrt{T_km_1/m_k}} \frac{1}{\sqrt{n_k}} M_k , ( \frac{|v-u_k|^2}{2T_km_1/m_k} - \frac{1}{2}) \frac{1}{\sqrt{n_k}} M_k \rbrace =: \lbrace b_1^k, b_2^k, b_3^k \rbrace.
\end{align*}
Using this orthonormal basis of $\mathcal{N}_k$, one finds for any function $\phi \in L^2_{M_k}$ the following expression of $\Pi_{M_k}(\phi)$
\begin{align}
\Pi_{M_k} ( \phi ) = \sum_{n=1}^3 (\phi , b_n^k ) b_n^k &= \frac{1}{n_k} [ \langle \phi \rangle + \frac{(v-u_k)\cdot \langle (v-u_k) \phi\rangle }{T_km_1/m_k}\nonumber\\ &+( \frac{|v-u_k|^2}{2T_km_1/m_k} - \frac{1}{2}) 2\langle ( \frac{|v-u_k|^2}{2T_km_1/m_k} - \frac{1}{2}) \phi \rangle ] M_k.
\label{explicit}
\end{align}
This orthogonal projection $\Pi_{M_k}(\phi)$ has some elementary properties.
\begin{lemma}[Properties of $\Pi_{M_k}$] We have, for $k=1,2$,
 \begin{align*} 
 &(\mathds{1}- \Pi_{M_k})(M_k)=(\mathds{1}- \Pi_{M_k})(\partial_t M_k) =0,\\& \Pi_{M_k}(g_{kk})= \Pi_{M_k}(\partial_t g_{kk})
 =0,
 \end{align*}
and 
\begin{align}
\Pi_{M_1} ( M_{12} ) = (1 &+ \frac{(v-u_1) (u_{12} -u_1)}{T_1} \nonumber\\&+ (\frac{|v-u_1|^2}{2 T_1} - \frac{1}{2})( \frac{T_{12}}{ T_1} + \frac{|u_{12} - u_1 |^2}{ T_1}  -1))M_1,
\label{formPi}\\
\Pi_{M_2} ( M_{21} ) = (1 &+ \frac{(v-u_2) (u_{21} -u_2)}{T_2 m_1/m_2} \nonumber\\&+ (\frac{|v-u_2|^2}{2 T_2 m_1/m_2} - \frac{1}{2})( \frac{T_{21}}{ T_2} + \frac{|u_{21} - u_2 |^2}{ T_2 m_1/m_2}  -1))M_2.
\label{formPe}
\end{align}
\label{lem:formPi}
\end{lemma}

\begin{proof}
The proof of the first five equalities is analogue to the one species case and is given in \cite{MicroMacro2007}. 
Besides, using the explicit expression of $\Pi_{M_k}$, $k=1,2$, given by \eqref{explicit} we obtain \eqref{formPi}-\eqref{formPe} by direct computations. 
\end{proof}

Now we apply the orthogonal projection $\mathds{1}- \Pi_{M_1}$ to (\ref{decomp}), use lemma \ref{lem:formPi}  
and obtain 
\begin{align*}
 \partial_t g_{11} &+ (\mathds{1}- \Pi_{M_1})(v\partial_x  M_1) + (\mathds{1}- \Pi_{M_1})(v\partial_x  g_{11} ) 
 \\&=\frac{1}{\tilde{\varepsilon}_i} \nu_{12} n_2 ( M_{12} - \Pi_{M_1}(M_{12}) ) - ( \frac{1}{\varepsilon_1} \nu_{12} n_1  + \frac{1}{\tilde{\varepsilon}_1} \nu_{12} n_2 ) g_{11}.
\end{align*}
Again with lemma \ref{lem:formPi}  we replace $\Pi_{M_1} (M_{12})$ by its explicit expression 
\begin{align}
\begin{split}
\partial_t g_{11} &+ (\mathds{1}- \Pi_{M_1})(v\partial_x  M_1) + (\mathds{1}- \Pi_{M_1})(v\partial_x  g_{11} )
\\&= \frac{1}{\tilde{\varepsilon}_1} \nu_{12} n_2 ( M_{12} - (1 + \frac{(v-u_1) (u_{12} -u_1)}{T_1} + (\frac{|v-u_1|^2}{2 T_1 } - \frac{1}{2})( \frac{T_{12}}{ T_1} \\&+ \frac{1}{ T_1} |u_{12} - u_1 |^2 -1))M_1 ) - ( \frac{1}{\varepsilon_1} \nu_{12} n_1 + \frac{1}{\tilde{\varepsilon}_1} \nu_{12} n_2 ) g_{11}.
 \end{split}
 \label{micro1}
\end{align}
 We take the moments of equation (\ref{decomp}), 
use \eqref{eq:moments_of_g}, and we get 
\begin{align}
\begin{split}
\partial_t \langle m(v) M_1 \rangle +  \partial_x  \langle m(v) v M_1 \rangle + \partial_x  \langle m(v) v g_{11} \rangle 
= \frac{1}{\tilde{\varepsilon}_1} \nu_{12} n_2 (  \langle m(v) (M_{12}- M_1)\rangle).
\end{split}
\label{macro1}
\end{align}
%

In a similar way, we get an analogous coupled system for species 2 which is coupled with the system of the ions 
\begin{align}
\begin{split}
 \partial_t  g_{22} & + (\mathds{1}- \Pi_{M_2})(v\partial_x  M_2) + (\mathds{1}- \Pi_{M_2})(v\partial_x   g_{22} ) 
 \\&= \frac{1}{\tilde{\varepsilon}_2} \nu_{12} n_1   ( M_{21} - (1 + \frac{(v-u_2) (u_{21} -u_2)}{T_2} \frac{m_2}{m_1} \\&+ (\frac{|v-u_2|^2}{2 T_2 }\frac{m_2}{m_1} - \frac{1}{2})( \frac{T_{21}}{ T_2} + \frac{m_2}{m_1 T_2} |u_{21} - u_2 |^2 -1))M_2 ) \\&- ( \frac{1}{\varepsilon_2} \nu_{12} n_2  + \frac{1}{\tilde{\varepsilon}_2} \nu_{12} n_1) g_{22},
 \end{split}
 \label{micro2}\\
\begin{split}
\partial_t \langle m M_2 \rangle &+  \partial_x  \langle m (v M_2) \rangle + \partial_x  \langle m (v g_{22} )\rangle 
= \frac{1}{\tilde{\varepsilon}_2} \nu_{12} n_1 ( \langle m (M_{21}- M_2)\rangle).
\end{split}
\label{macro2}
\end{align}
Now we have obtained a system of two microscopic equations \eqref{micro1}, \eqref{micro2} and two macroscopic equations \eqref{macro1}, \eqref{macro2}. One can show that this system is an equivalent formulation of the BGK equations for species 1 and species 2. This is analogous to what is done in \cite{MicroMacro2013}. 

  \section{\textcolor{black}{Space-homogeneous case}}\label{sec:spacehomcase}
  In this section, we consider our model \textcolor{green}{(\ref{nondimBGK})} in the space-homogeneous case, where we can prove an estimation of the decay rate of $|| f_k(t) - M_k(t)||_{L^1(dv)}$, $|u_1(t) - u_2(t)|^2$ and $|T_1(t)-T_2(t)|^2.$
  
In the space-homogeneous case, the BGK model for mixtures \eqref{BGK} simplifies to
\begin{align}
\begin{split}
\partial_t f_1  &= \frac{1}{\varepsilon_1} \nu_{12} n_1 (M_1 - f_1) + \frac{1}{\tilde{\varepsilon}_1}\nu_{12} n_2 (M_{12} - f_1), \\
\partial_t f_2  &= \frac{1}{\varepsilon_2} \nu_{12} n_2 (M_2 - f_2) + \frac{1}{\tilde{\varepsilon}_2} \nu_{12} n_2 (M_{21} - f_2),
\end{split}
\label{BGKconv0}
\end{align}
and we let the reader adapt the micro-macro decomposition \eqref{micro1}-\eqref{macro1}-\eqref{micro2}-\eqref{macro2} to this case.

 \subsection{Decay rate for the BGK model for mixtures in the space-homo\-ge\-neous case}
 We denote by $H(f)= \int f \ln f dv$ the entropy of a function $f$ and by $H(f|g)= \int f \ln \frac{f}{g} dv$ the relative entropy of $f$ and $g$.
 
\begin{theorem}
In the space homogeneous case 
we have the following decay rate of the distribution functions $f_1$ and $f_2$
$$ || f_k - M_k ||_{L^1(dv)} \leq 4 e^{- \frac{1}{2} Ct } [ H(f_1^0|M_1^0) + H(f_2^0 | M_2^0)]^{\frac{1}{2}}, \quad k=1,2,$$
where $C$ is a constant.
\end{theorem}
\begin{proof}
We consider the entropy production of species $1$ defined by
\begin{align*}
D_1(f_1, f_2) = - \int \frac{1}{\varepsilon_1} \nu_{12} n_1 \ln f_1 (M_1 - f_1) dv - \int \frac{1}{\tilde{\varepsilon}_1} \nu_{12} n_2 \ln f_1 (M_{12} - f_1) dv.
\end{align*}
Define $\phi:\mathbb{R}^+ \rightarrow \mathbb{R}, \phi(x) := x \ln x$. Then $\phi'(x)= \ln x + 1$, so we can deduce 
\begin{align*}
D_1(f_1, f_2) = - \int \frac{1}{\varepsilon_1} \nu_{12} n_1 \phi'(f_1) (M_1 - f_1) dv - \int \frac{1}{\tilde{\varepsilon}_1}\nu_{12} n_2 \phi'(f_1) (M_{12} - f_1) dv,
\end{align*}
since $\int (f_1-M_1)dv=\int (f_1-M_{12})dv=0$. Moreover, we have $\phi''(x) = \frac{1}{x}$. So $\phi$ is convex and we obtain
\begin{align}
\begin{split}
D_1(f_1, f_2) & \geq \int \frac{1}{\varepsilon_1} \nu_{12} n_1 (\phi(f_1) - \phi(M_1)) dv + \int \frac{1}{\tilde{\varepsilon}_1} \nu_{12} n_2 ( \phi(f_1) - \phi(M_{12})) dv \\&= 
 \frac{1}{\varepsilon_1} \nu_{12} n_1 (H(f_1) - H(M_1)) + \frac{1}{\tilde{\varepsilon}_1} \nu_{12} n_2 ( H(f_1) - H(M_{12})).
 \end{split}
 \label{eq:D}
\end{align}
In the same way we get a similar expression for $D_2(f_2,f_1)$  just exchanging the indices $1$ and $2$. \\ 
If we use that $\ln M_1$ is a linear combination of $1,v$ and $|v|^2$, we see that $\int (M_1-f_1) \ln M_1 dv=0$ since $f_1$ and $M_1$ have the same moments. With this we can compute that 
\begin{align}
H(f_1 | M_1) = H(f_1)- H(M_1).
\label{eq:relentr}
\end{align}
Moreover in the proof of theorem 2.7 in \cite{Pirner}, we see that
\begin{align}
 \frac{1}{\tilde{\varepsilon}_1} \nu_{12} n_2 H(M_{12}) + \frac{1}{\tilde{\varepsilon}_2}\nu_{12} n_1 H(M_{21}) \leq \frac{1}{\tilde{\varepsilon}_1}\nu_{12} n_2 H(M_1) + \frac{1}{\tilde{\varepsilon}_2}\nu_{12} n_1 H(M_2)
 \label{eq:theo2.7}.
\end{align}
With \eqref{eq:relentr} and \eqref{eq:theo2.7}, we can deduce from \eqref{eq:D} that
\begin{align}
\begin{split}
D_1(f_1,f_2)+ D_2(f_2,f_1) \geq (\frac{1}{\varepsilon_1}\nu_{12} n_1 + \frac{1}{\tilde{\varepsilon}_1}\nu_{12} n_2) H(f_1|M_1)\\ + (\frac{1}{\varepsilon_2}\nu_{12} n_2 + \frac{1}{\tilde{\varepsilon}_2}\nu_{12} n_1) H(f_2 |M_2).
\end{split}
\label{prod}
\end{align}
We want to relate the time derivative of the relative entropies
\begin{align*}
\frac{d}{dt} (H(f_1|M_1) + H(f_2|M_2)) = \frac{d}{dt} [\int f_1 \ln \frac{f_1}{M_1} dv + \int f_2 \ln \frac{f_2}{M_2} dv]
\end{align*}
 to the entropy production in the following. First we use product rule and obtain
\begin{align}
\begin{split}
\frac{d}{dt} (H(f_1|M_1) + H(f_2|M_2)) &= \int \partial_t f_1 (\ln \frac{f_1}{M_1}+1) dv - \int \frac{f_1}{M_1} \partial_t M_1 dv \\ &+ \int \partial_t f_2 (\ln \frac{f_2}{M_2}+1) dv - \int \frac{f_2}{M_2} \partial_t M_2 dv.
\end{split}
\label{eq:derentr}
\end{align}
By using the explicit expression of $\partial_t M_1$, we can compute that $\int f_k \frac{\partial_t M_k}{M_k} dv =\partial_t n_k =0, k=1,2,$ since $n_k$ is constant in the space-homogeneous case. In the first term on the right-hand side of \eqref{eq:derentr}, we insert $\partial_t f_1$ and $\partial_t f_2$ from equation \eqref{BGKconv0} and obtain 
\begin{align*}
\frac{d}{dt} (H(f_1|M_1) + H(f_2|M_2))&=\int ( \frac{1}{\varepsilon_1} \nu_{12} n_1 (M_1 - f_1) + \frac{1}{\tilde{\varepsilon}_1}\nu_{12} n_2 (M_{12} - f_1)) \ln f_1 dv \\&+ \int (\frac{1}{\varepsilon_2} \nu_{12} n_2 (M_2 - f_2) + \frac{1}{\tilde{\varepsilon}_2} \nu_{12} n_1 (M_{21} - f_2)) \ln f_2 dv.
\end{align*}
Indeed, the terms with $\ln M_1$ (resp. $\ln M_2$) vanish since $\ln M_1$ (resp. $\ln M_2$) is a linear combination of $1,v$ and $|v|^2$ and our model satisfies the conservation of the number of particles, total momentum and total energy (see section 2.2 in \cite{Pirner}). All in all, we obtain
\begin{align}
\begin{split}
\frac{d}{dt} (H(f_1|M_1) + H(f_2|M_2)) = - (D_1(f_1,f_2)+ D_2(f_2,f_1)).
\end{split}
\end{align}
 Using \eqref{prod} we obtain
\begin{align*}
\frac{d}{dt} (&H(f_1|M_1) + H(f_2|M_2))\\ &\leq -[(\frac{1}{\varepsilon_1}\nu_{12} n_1 + \frac{1}{\tilde{\varepsilon}_1}\nu_{12} n_2) H(f_1|M_1) + (\frac{1}{\varepsilon_2}\nu_{12} n_2 + \frac{1}{\tilde{\varepsilon}_2}\nu_{12} n_1) H(f_2 |M_2)] \\ &\leq - \text{min} \lbrace \frac{1}{\varepsilon_1} \nu_{12} n_1 + \frac{1}{\tilde{\varepsilon}_1} \nu_{12} n_2 , \frac{1}{\varepsilon_2}\nu_{12} n_2 + \frac{1}{\tilde{\varepsilon}_2}\nu_{12} n_1 \rbrace ( H( f_1| M_1) + H(f_2|M_2)).
\end{align*}
Define $C:= \text{min} \lbrace \frac{1}{\varepsilon_1}\nu_{12} n_1 + \frac{1}{\tilde{\varepsilon}_1}\nu_{12} n_2 , \frac{1}{\varepsilon_2}\nu_{12} n_2 + \frac{1}{\tilde{\varepsilon}_2}\nu_{12} n_1 \rbrace ,$ then we can deduce an exponential decay with Gronwall's identity
\begin{align*}
H(f_k|M_k) &\leq H( f_1| M_1) + H(f_2|M_2)\\& \leq e^{-Ct} [H( f_1^0| M_1^0) + H(f_2^0|M_2^0)], \quad k=1,2.
\end{align*}
With the Ciszar-Kullback inequality (see proposition 1.1 in \cite{Matthes}) we get
\begin{align*}
||f_k - M_k||_{L^1(dv)} & \leq ||f_1 - M_1||_{L^1(dv)} +||f_2 - M_2||_{L^1(dv)} \\ & \leq 4 e^{-\frac{1}{2} Ct} [H( f_1^0| M_1^0) + H(f_2^0|M_2^0)]^{\frac{1}{2}}.
\end{align*}
\end{proof}

\subsection{Decay rate for the velocities and temperatures in the space-homo\-ge\-neous case}\label{app:velocities}

 In this subsection we prove decay rates for the velocities $u_1,u_2$ (resp. temperatures $T_1,T_2$) to a common value in the space-homogeneous case. We start with a decay of $|u_1-u_2|^2$.
  \begin{theorem}\label{th:estimate_vel}
Suppose that $\nu_{12}$ is constant in time. In the space-homogeneous case \eqref{BGKconv0}, we have the following decay rate of the velocities
\begin{equation*}
|u_1(t) - u_2(t)|^2 = e^{- 2 \nu_{12} (1- \delta)\left(\frac{1}{\tilde{\varepsilon}_1}n_2+\frac{\varepsilon}{\tilde{\varepsilon}_2}\frac{m_1}{m_2} n_i\right) t} |u_1(0) - u_2(0)|^2.
\end{equation*}
\end{theorem}
\begin{proof}
If we multiply the  equations \eqref{BGKconv0} by $v$ and integrate with respect to $v$, we obtain by using \eqref{convexveldim}, \eqref{velocdim} and \eqref{eps}
\begin{align*}
\partial_t ( n_1 u_1)  &= \frac{1}{\tilde{\varepsilon}_1}\nu_{12} n_2 n_1  (u_{12}-u_1)= \frac{1}{\tilde{\varepsilon}_1}\nu_{12} n_2 n_1 (1- \delta) (u_2-u_1), \\
\partial_t ( n_2 u_2)&= \frac{1}{\tilde{\varepsilon}_2} \nu_{12} n_2 n_1 (u_{21}-u_2) =\frac{1}{\tilde{\varepsilon}_2} \nu_{12} n_2 n_1 \frac{m_1}{m_2} \varepsilon (1- \delta) (u_1-u_2).
\end{align*}
Since in the space-homogeneous case the densities $n_1$ and $n_2$ are constant, we actually have
\begin{align*}
\partial_t   u_1 = \frac{1}{\tilde{\varepsilon}_1}\nu_{12} n_2  (1- \delta) (u_2-u_1), \quad
\partial_t   u_2 =\frac{1}{\tilde{\varepsilon}_2} \nu_{12} n_1 \frac{m_1}{m_2} \varepsilon (1- \delta) (u_1-u_2).
\end{align*}
With this we get
\begin{align*}
\frac{1}{2} \frac{d}{dt} |u_1-u_2|^2 &= (u_1-u_2) \partial_t (u_1 - u_2) \\&= (u_1-u_2) \nu_{12} (1- \delta)\left(\frac{1}{\tilde{\varepsilon}_i}n_2+\frac{\varepsilon}{\tilde{\varepsilon}_2}\frac{m_1}{m_2} n_1\right) (u_2 - u_1 ) \\&= - \nu_{12} (1- \delta)\left(\frac{1}{\tilde{\varepsilon}_1}n_2+\frac{\varepsilon}{\tilde{\varepsilon}_2}\frac{m_1}{m_2} n_1\right)|u_1 - u_2|^2.
\end{align*}
From this, we deduce
\begin{align*}
|u_1(t) - u_2(t)|^2 = e^{- 2 \nu_{12} (1- \delta)\left(\frac{1}{\tilde{\varepsilon}_1}n_2+\frac{\varepsilon}{\tilde{\varepsilon}_2}\frac{m_1}{m_2} n_1\right) t} |u_1(0) - u_2(0)|^2.
\end{align*}
\end{proof}
We continue with a decay rate of $|T_1(t) - T_2(t)|$.
  \begin{theorem}\label{th:estimate_temp} 
Suppose $\nu_{12}$ is constant in time. In the space-homogeneous case \eqref{BGKconv0}, we have the following decay rate of the temperatures
\begin{equation*}
\begin{split}
\textcolor{black}{T_1(t) - T_2(t) = 
 e^{- C_1t} \left[T_1(0) - T_2(0)+\frac{C_2}{C_1-C_3 } ( e^{(C_1-C_3) t} - 1) |u_1(0) - u_2(0)|^2 \right],}
 \end{split}
\end{equation*}
where the constants are defined by
\begin{align*}
C_1&=(1- \alpha)\nu_{12}\left(\frac{1}{\tilde{\varepsilon}_1}n_2+\frac{\varepsilon}{\tilde{\varepsilon}_2}n_1\right),\\
C_2&=\nu_{12}\left(\frac{1}{\tilde{\varepsilon}_1}n_2\left((1-\delta)^2+\frac{\gamma}{m_1}\right)-\frac{\varepsilon}{\tilde{\varepsilon}_2}n_1\left(1-\delta^2-\frac{\gamma}{m_1}\right)\right),\\
C_3&=2 \nu_{12} (1- \delta)\left(\frac{1}{\tilde{\varepsilon}_1}n_2+\frac{\varepsilon}{\tilde{\varepsilon}_2}\frac{m_1}{m_2} n_1\right).
\end{align*}
\end{theorem}
\begin{proof}
If we multiply the first equation of \eqref{BGKconv0} by $\frac{1}{ n_1}|v-u_1|^2$ and integrate with respect to $v$, we obtain
\begin{align}
\int \frac{1}{n_1}|v-u_1|^2 \partial_t f_1 dv = \frac{1}{\tilde{\varepsilon}_1}\nu_{12} n_2 \frac{1}{n_1} \int  |v-u_1|^2(M_{12} - f_1) dv.
\label{1}
\end{align}
Indeed, the first relaxation term vanishes since $M_1$ and $f_1$ have the same temperature.
We simplify the left-hand side of \eqref{1} to
\begin{align*}
\int \frac{1}{n_1}|v-u_1|^2 \partial_t f_1 dv &= \int \frac{1}{n_1} \partial_t (|v-u_1|^2 f_1) dv + 2\int \frac{1}{n_1} f_1 (v-u_1) \cdot \partial_t u_1 dv \\&= \partial_t (T_1) + 0, 
\end{align*} 
since the density $n_1$ is constant. The right-hand side of \eqref{1} simplifies to
\begin{align*}
\frac{1}{\tilde{\varepsilon}_1}\nu_{12} n_2 \frac{1}{n_1} \int  |v-u_1|^2 (M_{12}-f_1)dv = \frac{1}{\tilde{\varepsilon}_1} \nu_{12} n_2 (T_{12} + |u_{12} - u_{1}|^2 - T_1) \\= \frac{1}{\tilde{\varepsilon}_1}\nu_{12} n_2 \left((1-\alpha) (T_2-T_1) + \left( (1-\delta)^2 +  \frac{\gamma}{m_1}\right)|u_2-u_1|^2 \right).
\end{align*}
For the second species we multiply the second equation of \eqref{BGKconv0} by $\frac{m_2}{m_1} \frac{1}{n_2} |v-u_2|^2$. For the left-hand side, we obtain by using \eqref{momentsdim}
\begin{align*}
\int \frac{m_2}{m_1} \frac{1}{n_2} |v-u_2|^2 \partial_t f_2 dv = \partial_t T_2,
\end{align*}
and for the right-hand side using \eqref{velocdim}, \eqref{tempdim} and  \eqref{eps}
\begin{align*}
\frac{1}{\tilde{\varepsilon}_2} &\nu_{12} n_2 \frac{m_2}{m_1} \frac{1}{n_2} \int |v- u_2|^2  (M_{21} - f_2)dv = \frac{1}{\tilde{\varepsilon}_2} \nu_{12} n_1 (T_{21} + \frac{m_2}{m_1} |u_{21}- u_2|^2 - T_2) \\&= \frac{1}{\tilde{\varepsilon}_2} \nu_{12} n_1 \left[ \varepsilon (1- \alpha) (T_1-T_2) \right.\\ 
&\left.+ \left(\varepsilon (1- \delta)\left( \frac{m_1}{m_2} \varepsilon (\delta -1) + \delta + 1\right) - \varepsilon \frac{\gamma}{m_1} + \varepsilon^2 (1- \delta)^2 \frac{m_1}{m_2} \right) |u_1- u_2|^2 \right] \\ 
&= \frac{1}{\tilde{\varepsilon}_2} \nu_{12} n_1 \left( \varepsilon (1- \alpha) (T_1-T_2)+ \varepsilon (1- \delta^2-\frac{\gamma}{m_1}) |u_1- u_2|^2 \right). 
\end{align*}
  So, we obtain
\begin{align*}
\partial_t  T_1 &= \frac{1}{\tilde{\varepsilon}_1}\nu_{12} n_2 \left((1-\alpha) (T_2-T_1) + \left( (1-\delta)^2 +  \frac{\gamma}{m_1}\right)|u_2-u_1|^2 \right), \\
\partial_t  T_2 &= \frac{1}{\tilde{\varepsilon}_2} \nu_{12} n_1 \left( \varepsilon (1- \alpha) (T_1-T_2)+ \varepsilon \left(1- \delta^2-\frac{\gamma}{m_1}\right) |u_1- u_2|^2 \right).
\end{align*}
We deduce
\begin{align*}
\partial_t ( T_1 - T_2) &=- (1-\alpha)\nu_{12}\left(\frac{1}{\tilde{\varepsilon}_1}n_2+\frac{\varepsilon}{\tilde{\varepsilon}_2}n_1\right) (T_1-T_2)\\
&+ \nu_{12} \left(\frac{1}{\tilde{\varepsilon}_1}n_2\left( (1-\delta)^2 +  \frac{\gamma}{m_1}\right)-\frac{\varepsilon}{\tilde{\varepsilon}_2}n_1\left(1- \delta^2-\frac{\gamma}{m_1}\right) \right) |u_1- u_2|^2, 
\end{align*}
or with the constants defined in this theorem \ref{th:estimate_temp} 
\begin{align*}
\partial_t ( T_1 - T_2) &=- C_1 (T_1-T_2)+ C_2 |u_1- u_2|^2. 
\end{align*}
Duhamel's formula gives
\begin{align*}
T_1(t) - T_2(t)=e^{-C_1t}(T_1(0)-T_2(0))+C_2e^{-C_1t}\int_0^te^{C_1s}|u_1(s)- u_2(s)|^2ds,
\end{align*}
and by using theorem \ref{th:estimate_vel}, we have \textcolor{black}{
\begin{align*}
T_1(t) - T_2(t)&= e^{-C_1t}(T_1(0)-T_2(0))+C_2 e^{-C_1t}\int_0^te^{C_1s}e^{-C_3s}ds|u_1(0)- u_2(0)|^2\\
&= e^{-C_1t}\left(T_1(0)-T_2(0)+\frac{C_2}{C_1-C_3}(e^{(C_1-C_3)t}-1)|u_1(0)- u_2(0)|^2\right).
\end{align*}}
\end{proof}

\section{Numerical approximation}\label{sec:numapp}

This section is devoted to the numerical approximation of the two-species micro-macro system (\ref{micro1})-(\ref{macro1})-(\ref{micro2})-(\ref{macro2}). Following the idea of \cite{MicroMacro2013}, we propose to use a particle method to discretize both microscopic equations (\ref{micro1})-(\ref{micro2}), in order to reduce the cost of the method when approaching the Maxwellian equilibrium. Macroscopic equations (\ref{macro1})-(\ref{macro2}) are solved by a classical Finite Volume method.

In this paper, we only present the big steps of the method and refer to \cite{MicroMacro2013} for the details.

For the microscopic parts, we use a Particle-In-Cell method (see for example \cite{PICmethod}): we approach $g_{11}$ (resp. $g_{22}$) by a set of $N_{p_1}$ (resp. $N_{p_2}$) particles, with position $x_{1_k}(t)$ (resp. $x_{2_k}(t)$), velocity $v_{1_k}(t)$ (resp. $v_{2_k}(t)$) and weight $\omega_{1_k}(t)$ (resp. $\omega_{2_k}(t)$), $k=1,\dots,N_{p_1}$ (resp. $k=1,\dots,N_{p_2}$). Then we assume that the microscopic distribution functions have the following expression:
\begin{align*}
g_{11}(x,v,t)&=\sum_{k=1}^{N_{p_1}}\omega_{1_k}(t)\delta(x-x_{1_k}(t))\delta(v-v_{1_k}(t)),\\
g_{22}(x,v,t)&=\sum_{k=1}^{N_{p_2}}\omega_{2_k}(t)\delta(x-x_{2_k}(t))\delta(v-v_{2_k}(t)),
\end{align*}
with $\delta$ the Dirac mass. Moreover, we have the following relations:
\begin{align*}
\omega_{1_k}(t)&=g_{11}(x_{1_k}(t),v_{1_k}(t),t)\frac{L_xL_v}{N_{p_1}},~k=1,\dots,N_{p_1},\\
\omega_{2_k}(t)&=g_{22}(x_{2_k}(t),v_{2_k}(t),t)\frac{L_xL_v}{N_{p_2}},~k=1,\dots,N_{p_2},
\end{align*}
where $L_x\in\mathbb{R}$ (resp. $L_v\in\mathbb{R}$) denotes the length of the domain in the space (resp. velocity) direction.
 
The method consists now in splitting the transport and the source parts of (\ref{micro1}) (resp.(\ref{micro2})). Let us consider (\ref{micro1}), the steps being the same for (\ref{micro2}). The transport part 
\begin{align}
\partial_t g_{11} + v \partial_x  g_{11} 
= 0,
\end{align}
is solved by pushing the particles, that is evolving the positions \textcolor{black}{(velocities are constants)} thanks to the equations of motion:
$$\mathrm{d}_tx_{1_k}(t)=v_{1_k}(t),~~~\mathrm{d}_tv_{1_k}(t)=0,~~~\forall~k=1,\dots,N_{p_1}.$$
The source part 
\begin{align}
\begin{split}
\partial_t g_{11} =&- (\mathds{1}- \Pi_{M_1})(v\partial_x  M_1) +  \Pi_{M_1}(v\partial_x  g_{11} ) 
\end{split}
\end{align}
is solved by evolving the weights. Let us denote by $S(x,v,t)$ the right-hand side such that $\partial_t g_{11} =S(x,v,t)$. We compute the weight corresponding to S using the relation $s_{1_k}(t)=S(x_{1_k}(t),v_{1_k}(t),t)\frac{L_xL_v}{N_{p_1}},~k=1,\dots,N_{p_1}$ and then solve 
$$\mathrm{d}_t\omega_{1_k}(t)=s_{1_k}(t).$$ 
The strategy is the same as in paragraph 4.1.2 of \cite{MicroMacro2013}, where only one species is considered (and so there is no coupling terms). The supplementary terms coming from the coupling of both species are treated in the source part as the other source terms. They do not add particular difficulty.

A projection step, similar to the matching procedure of \cite{matching}, ensures the preservation of the micro-macro structure (\ref{eq:mi_ma_decomp}) and in particular the property (\ref{eq:moments_of_g}) on the moments of $g_{11}$ (resp. $g_{22}$). Details are given in subsection 4.2 of \cite{MicroMacro2013}.

Finally, macroscopic equations (\ref{macro1})-(\ref{macro2}) are discretized on a grid in space and solved by a classical Finite Volume method. For the one species case, this is detailed in subsection 4.3 of \cite{MicroMacro2013}. 
    
\section{Numerical results}\label{sec:numresults}

We present in this section some numerical experiments obtained by the numerical approximation presented in section \ref{sec:numapp}. A first series of tests aims at verifying numerically the decay rates of velocities and temperatures proved in subsection \ref{app:velocities} in the space-homogeneous case. 
 In a second series of tests, we are interested in the evolution in time of distribution functions, velocities and temperatures  in the general case. In particular, we want to see the influence of the collision frequencies.

In all this section, we consider the phase-space domain $\left(x,v\right)\in\left[0,4\pi\right]\times\left[-10,10\right]$ (assuming that physical particles of velocity $v$ such that $|v|>10$ can be negligible), so that $L_x=4\pi$ and $L_v=20$. \textcolor{black}{Concerning the mixture parameters, we take $\alpha=\delta=0.5$ and $\gamma=0.1$.}

\subsection{Decay rates in the space-homogeneous case}

We first propose to validate our model in the space-homogeneous case, where we have an estimation of the decay rate of $|u_1(t) - u_2(t)|^2$ and of $|T_1(t)-T_2(t)|$ (see section \ref{sec:spacehomcase}). \textcolor{black}{Here, we want to check if the behaviour of a gas mixture in the sense of relaxation to a global equilibrium (Maxwell distributions with the same mean velocity and temperature) is obtained in a reasonable way.} Note that as in section \ref{sec:spacehomcase}, we simplify the notations: $u_1(x,t)=u_1(t)$, $u_2(x,t)=u_2(t)$, $T_1(x,t)=T_1(t)$, $T_2(x,t)=T_2(t)$.

We apply a simplified version of the numerical approximation presented in section \ref{sec:numapp}, adapted to the space-homogeneous system (\ref{BGKconv0}) in its micro-macro form. For different initial conditions, we plot the evolution in time of $|u_1(t) - u_2(t)|^2$ (resp. $|T_1(t)-T_2(t)|$) and compare it to the estimates given in theorem \ref{th:estimate_vel} (resp. theorem \ref{th:estimate_temp}). For all  of these tests, we take $N_{p_1}=N_{p_2}=10^4$ and $\Delta t=10^{-4}$.

The first initial condition we consider corresponds to two Maxwellian functions:
\begin{align}
f_1(v,t=0)&=\frac{n_1}{\sqrt{2\pi T_1(t=0)}}\exp\left(-\frac{|v-u_1(t=0)|^2}{2T_1(t=0)}\right),\\
f_2(v,t=0)&=\frac{n_2}{\sqrt{2\pi T_2(t=0) \frac{m_1}{m_2}}}\exp\left(-\frac{|v-u_2(t=0)|^2}{2T_2(t=0)} \frac{m_2}{m_1}\right),
\end{align}
with the following parameters: $n_1=1$, $u_1(t=0)=0.5$, $T_1(t=0)=1$, $m_1=1$, $n_2=1.2$, $u_2(t=0)=0.1$, $T_2(t=0)=0.1$, $m_2=1.5$, chosen as in subsection 5.1 of \cite{Jin}. Results for $\varepsilon_1=\varepsilon_2=\tilde{\varepsilon}_1=\tilde{\varepsilon}_2=0.05$ are given in figure \ref{fig:vel_maxJin_1}  and results for $\varepsilon_1=\varepsilon_2=\tilde{\varepsilon}_1=\tilde{\varepsilon}_2=0.01$ are given in figure \ref{fig:vel_maxJin_2}. In these two cases, we plot $|u_1(t) - u_2(t)|$ too. As in \cite{Jin}, we remark that when the Knudsen numbers are smaller \textcolor{black}{(or when the collision frequencies are larger)}, the velocities, as well as the temperatures, converge faster to the equilibrium. \textcolor{black}{Moreover, the decay rates obtained with our scheme are in very good agreement with the theoretical ones established in theorems \ref{th:estimate_vel} and \ref{th:estimate_temp}.}

\begin{figure}[!htb]
\includegraphics[angle=-90,width=0.49\textwidth]{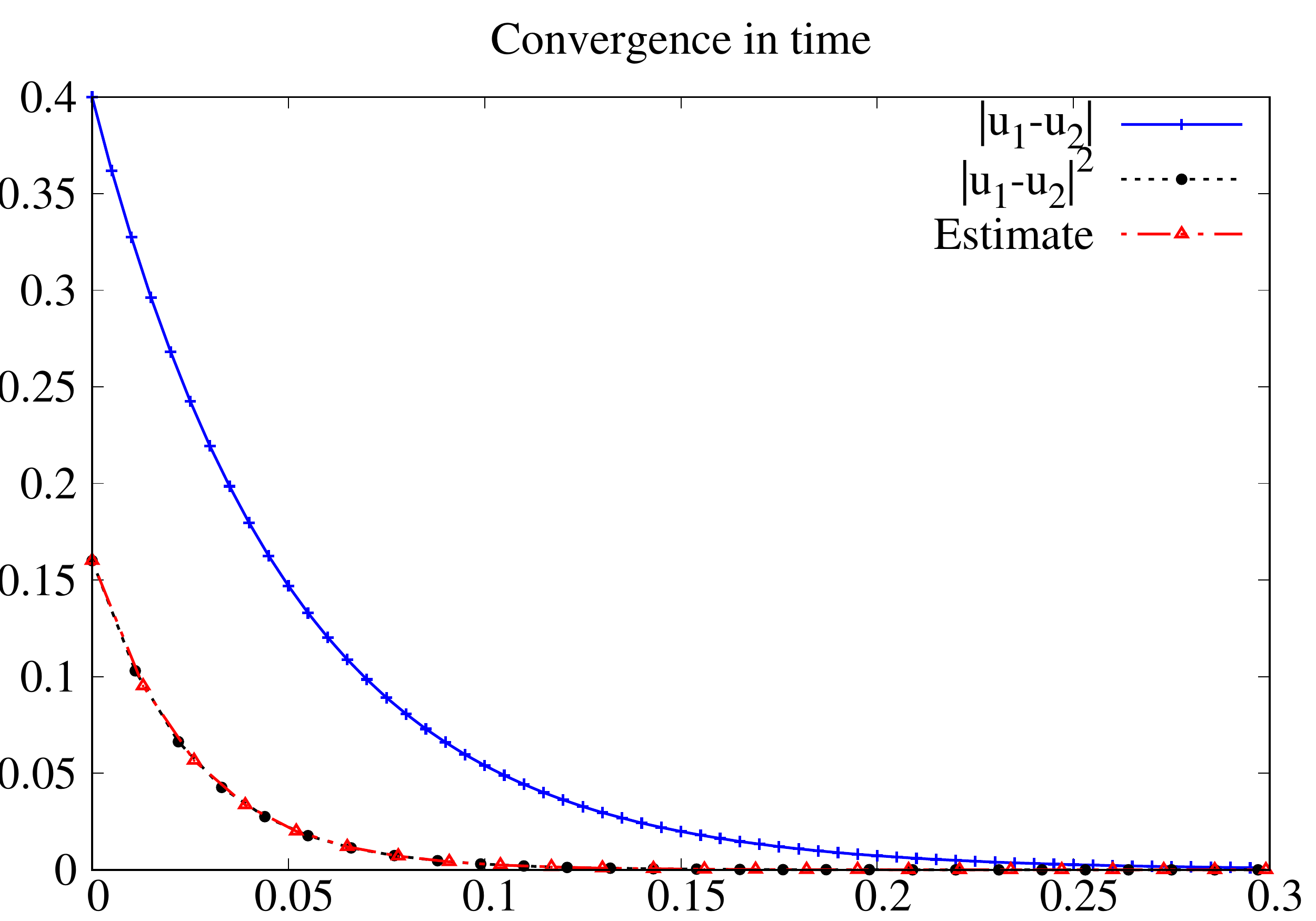}
\includegraphics[angle=-90,width=0.49\textwidth]{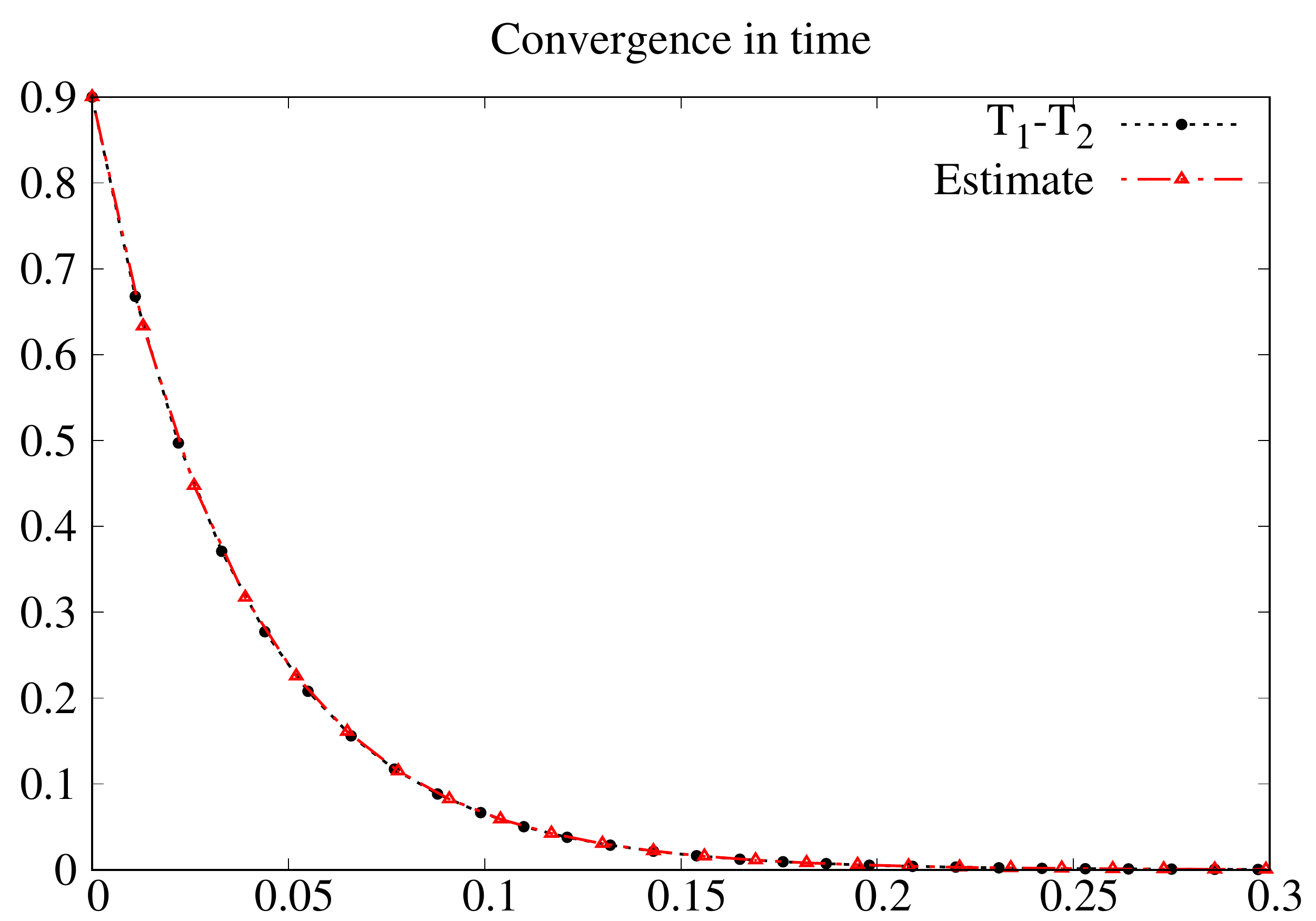}
\caption{Space-homogeneous case. Maxwellians initial conditions. Evolution in time of $|u_1(t) - u_2(t)|$, $|u_1(t) - u_2(t)|^2$ (left) and $|T_1(t) - T_2(t)|$ (right). Comparison to the estimated decay rates. Knudsen numbers: $\varepsilon_1=\varepsilon_2=\tilde{\varepsilon}_1=\tilde{\varepsilon}_2=0.05$.}
\label{fig:vel_maxJin_1}
\end{figure}

\begin{figure}[ht]
\includegraphics[angle=-90,width=0.49\textwidth]{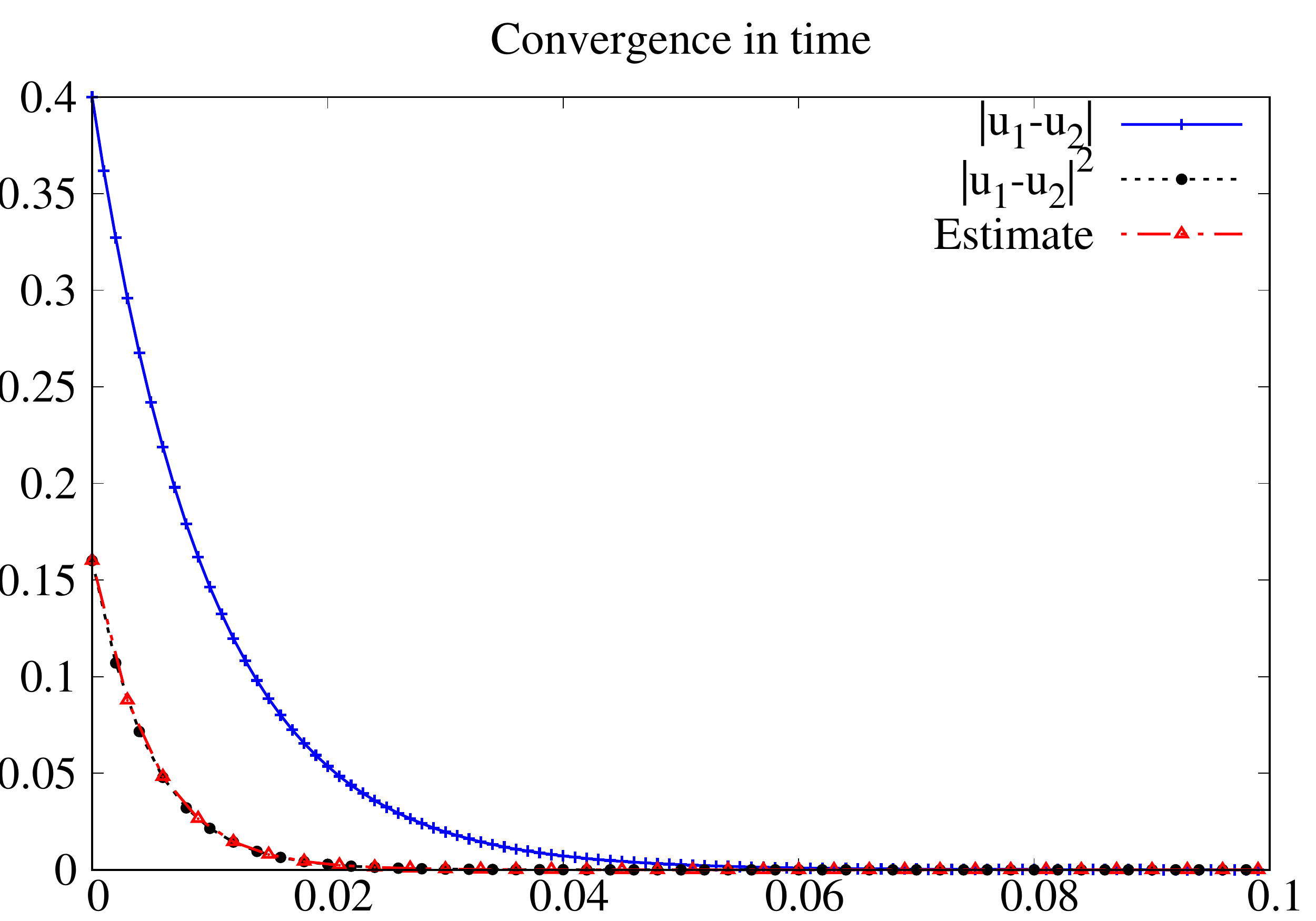}
\includegraphics[angle=-90,width=0.49\textwidth]{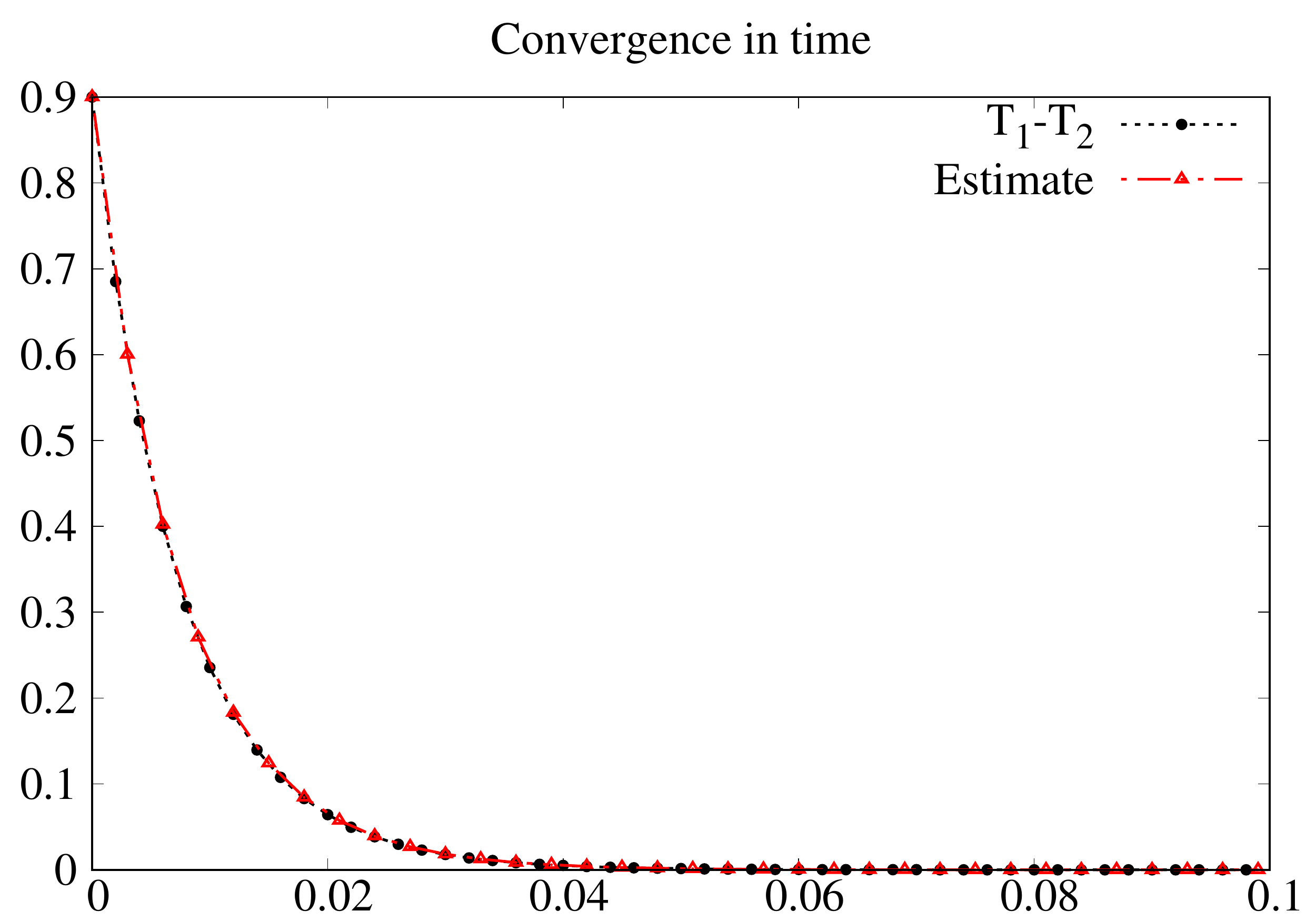}
\caption{Space-homogeneous case. Maxwellians initial conditions. Evolution in time of $|u_1(t) - u_2(t)|$, $|u_1(t) - u_2(t)|^2$ (left) and $|T_1(t) - T_2(t)|$ (right). Comparison to the estimated decay rates. Knudsen numbers: $\varepsilon_1=\varepsilon_2=\tilde{\varepsilon}_1=\tilde{\varepsilon}_2=0.01$.}
\label{fig:vel_maxJin_2}
\end{figure}

We propose now to consider $T_1(t=0)=0.08$ (other parameters are unchanged) and to study two other sets of Knudsen numbers. Results for $\varepsilon_1=\varepsilon_2=\tilde{\varepsilon}_1=\tilde{\varepsilon}_2=1$ are given in figure \ref{fig:vel_maxJin_bis_1} and results for $\varepsilon_1=\varepsilon_2=\tilde{\varepsilon}_1=1$, $\tilde{\varepsilon}_2=0.05$ are given in figure \ref{fig:vel_maxJin_bis_2}. \textcolor{black}{In this case too, we recover the right decay rates.}

\begin{figure}[ht]
\includegraphics[angle=-90,width=0.49\textwidth]{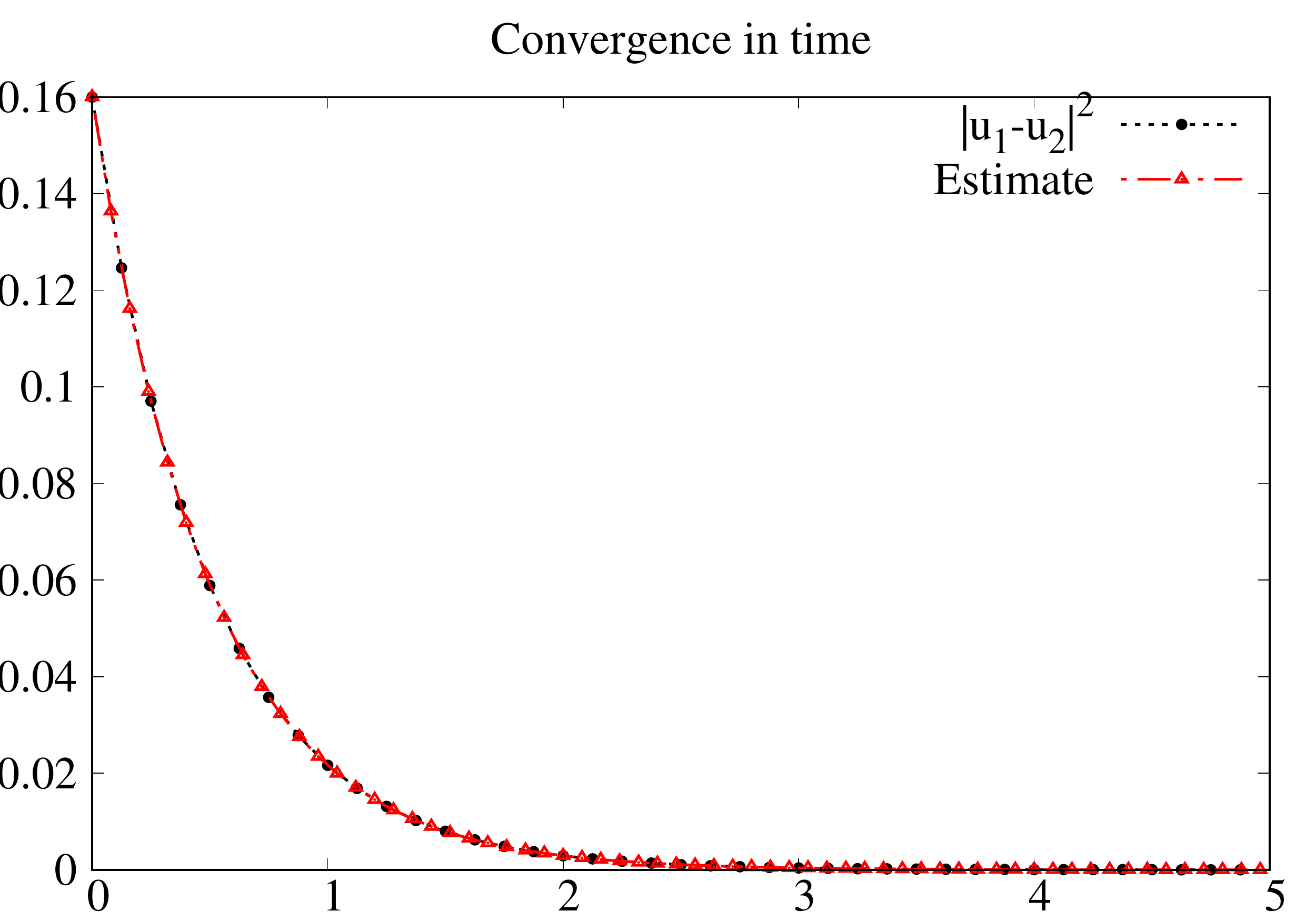}
\includegraphics[angle=-90,width=0.49\textwidth]{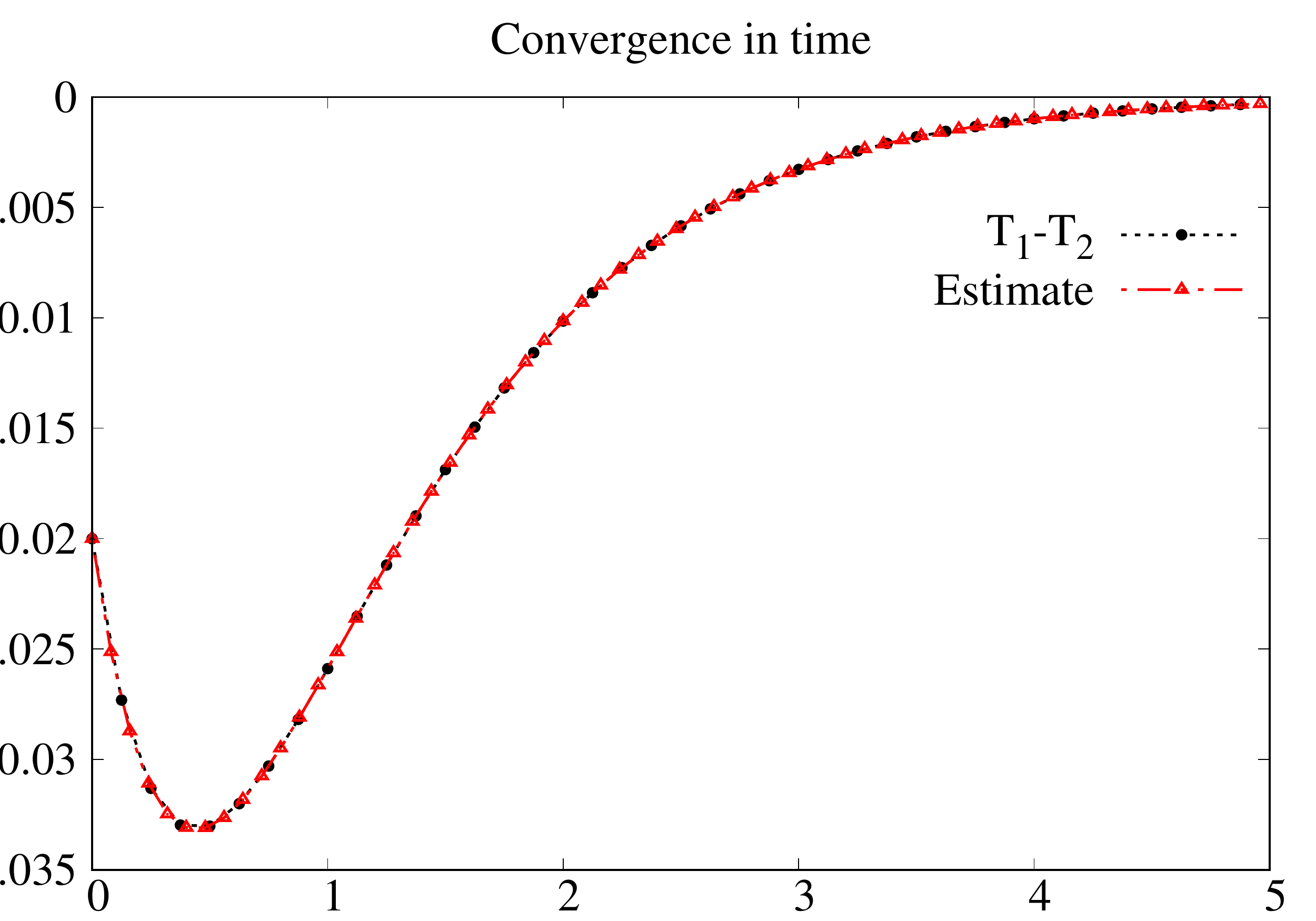}
\caption{Space-homogeneous case. Maxwellians initial conditions. Evolution in time of $|u_1(t) - u_2(t)|^2$ (left) and $|T_1(t) - T_2(t)|$ (right). Comparison to the estimated decay rates. Knudsen numbers: $\varepsilon_1=\varepsilon_2=\tilde{\varepsilon}_1=\tilde{\varepsilon}_2=1$.}
\label{fig:vel_maxJin_bis_1}
\end{figure}

\begin{figure}[ht]
\includegraphics[angle=-90,width=0.49\textwidth]{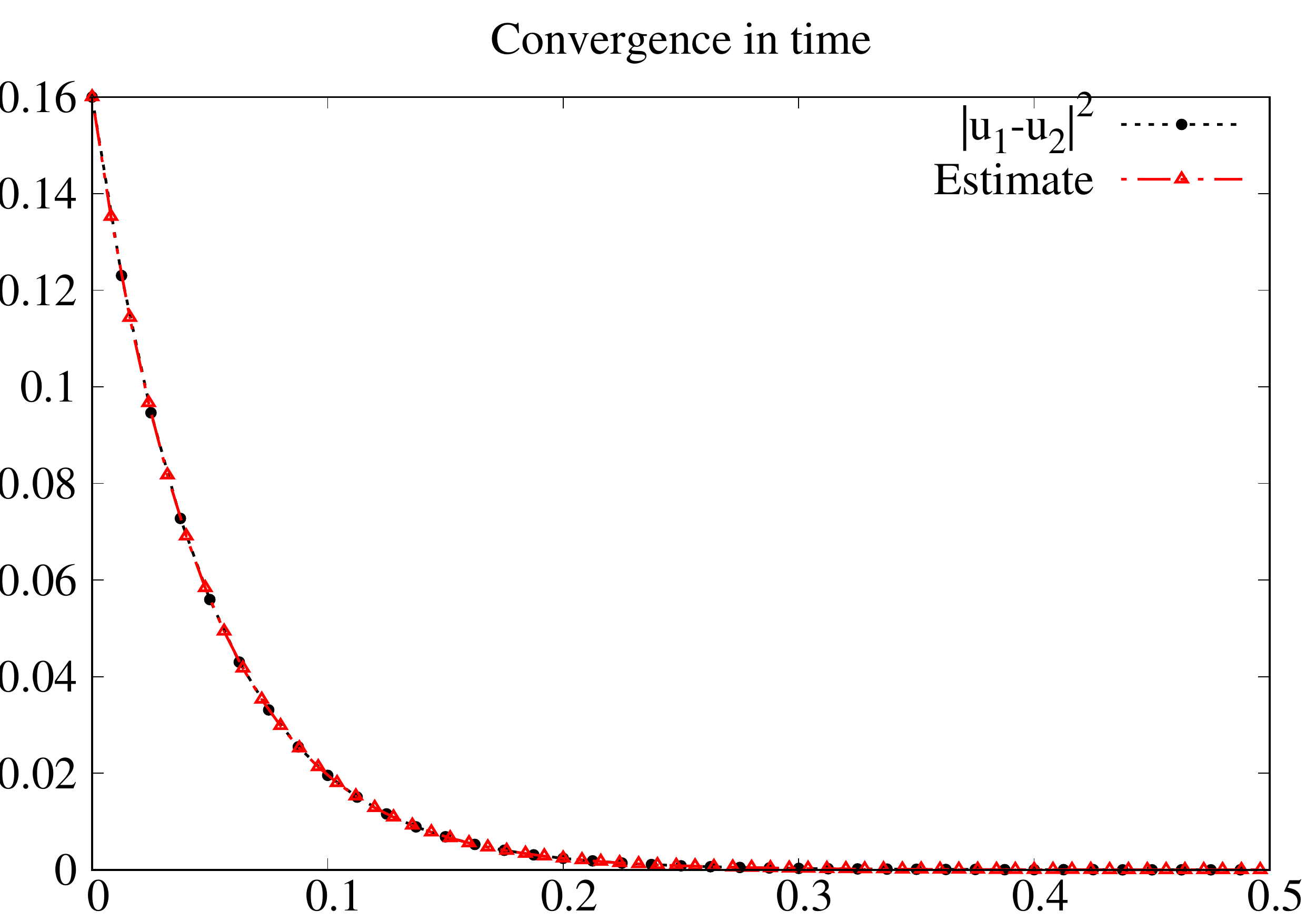}
\includegraphics[angle=-90,width=0.49\textwidth]{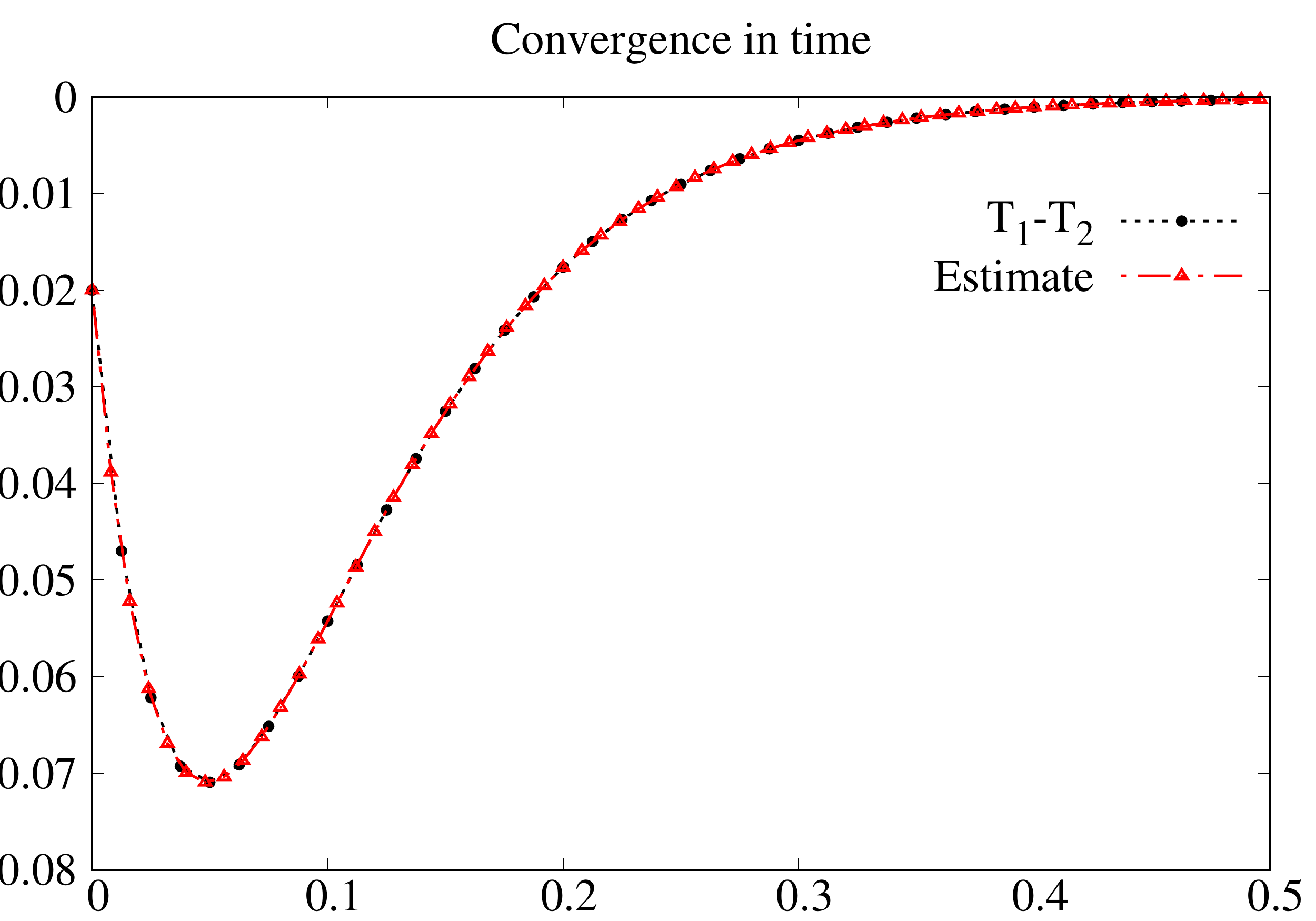}
\caption{Space-homogeneous case. Maxwellians initial conditions. Evolution in time of $|u_1(t) - u_2(t)|^2$ (left) and $|T_1(t) - T_2(t)|$ (right). Comparison to the estimated decay rates. Knudsen numbers: $\varepsilon_1=\varepsilon_2=\tilde{\varepsilon}_1=1$, $\tilde{\varepsilon}_2=0.05$.}
\label{fig:vel_maxJin_bis_2}
\end{figure}

We propose then to study the convergence for an other initial condition, considering
\begin{align}
f_1(v,t=0)&=\frac{v^4}{3\sqrt{2\pi}}\exp\left(-\frac{|v|^2}{2}\right),\\
f_2(v,t=0)&=\frac{n_2}{\sqrt{2\pi T_2(t=0)m_1/m_2}}\exp\left(-\frac{|v-u_2(t=0)|^2}{2T_2(t=0)}\frac{m_2}{m_1}\right),
\end{align}
with the following parameters: $n_2=1.2$, $u_2(t=0)=0.1$, $T_2(t=0)=0.1$, $m_2=1.5$. Here, the initial distribution of species 1 is not a Maxwellian, and then $g_{11}(v,t=0)\neq 0$. The estimates of theorems \ref{th:estimate_vel} and \ref{th:estimate_temp} are still verified, as we can see on figure \ref{fig:vel_maxJin_v4_1} for $\varepsilon_1=\varepsilon_2=\tilde{\varepsilon}_1=\tilde{\varepsilon}_2=1$. By taking now $T_2(t=0)=5$ (the other parameters being unchanged), we obtain results presented on figure \ref{fig:vel_maxJin_v4_2}.

\begin{figure}[ht]
\includegraphics[angle=-90,width=0.49\textwidth]{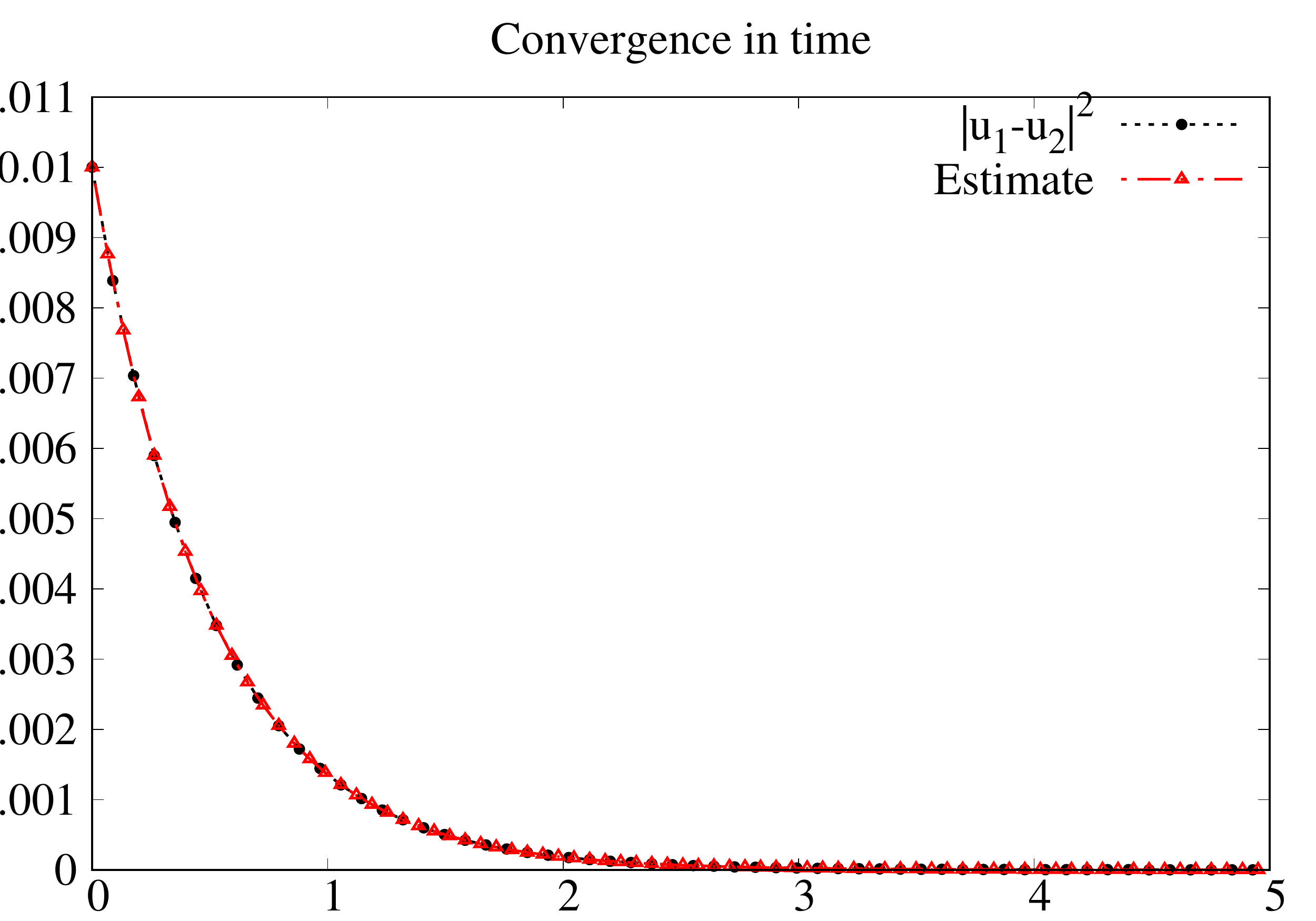}
\includegraphics[angle=-90,width=0.49\textwidth]{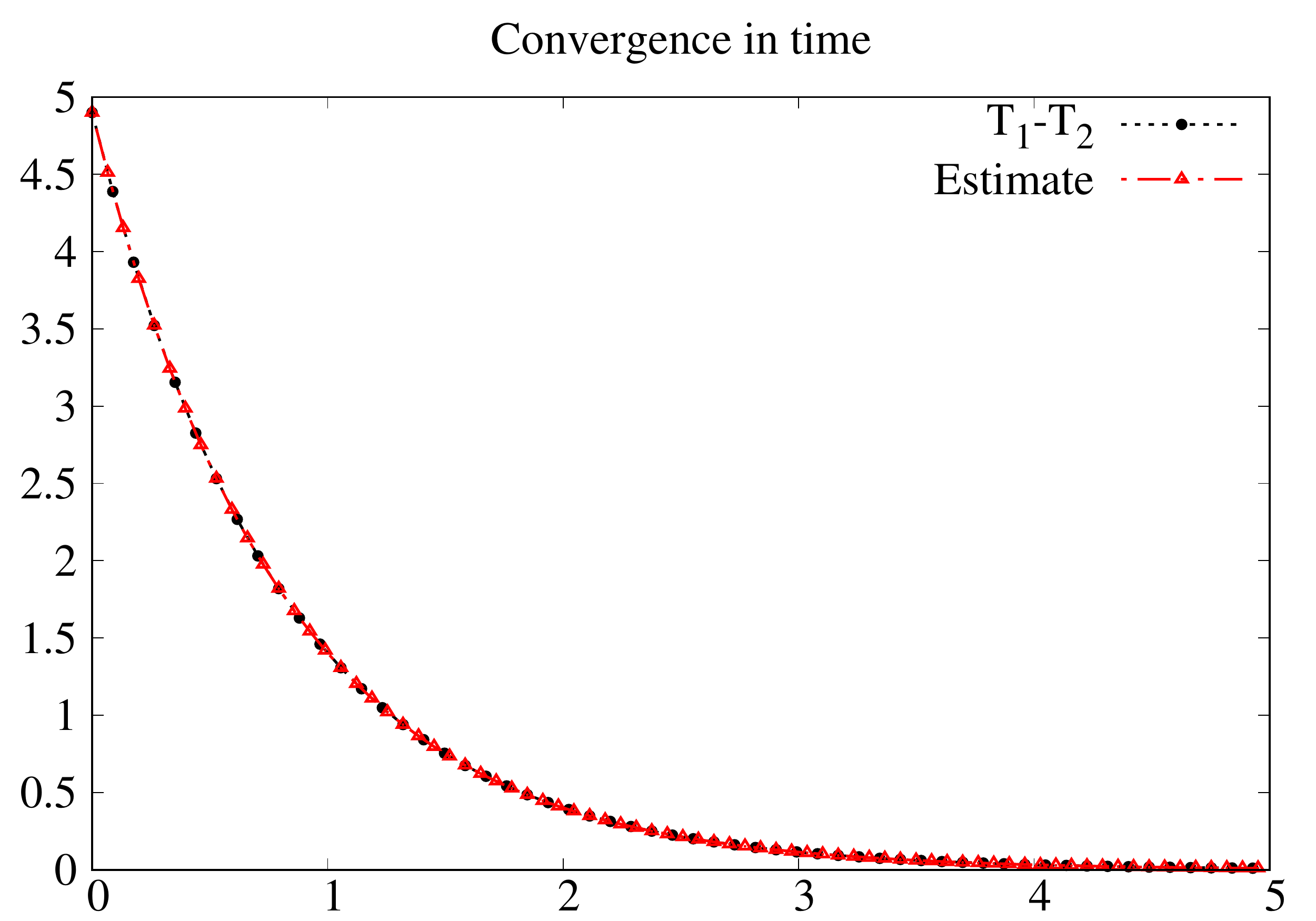}
\caption{Space-homogeneous case. Mixed initial conditions. Evolution in time of $|u_1(t) - u_2(t)|^2$ (left) and $|T_1(t) - T_2(t)|$ (right). Comparison to the estimated decay rates. Knudsen numbers: $\varepsilon_1=\varepsilon_2=\tilde{\varepsilon}_1=\tilde{\varepsilon}_2=1$.}
\label{fig:vel_maxJin_v4_1}
\end{figure}

\begin{figure}[ht]
\includegraphics[angle=-90,width=0.49\textwidth]{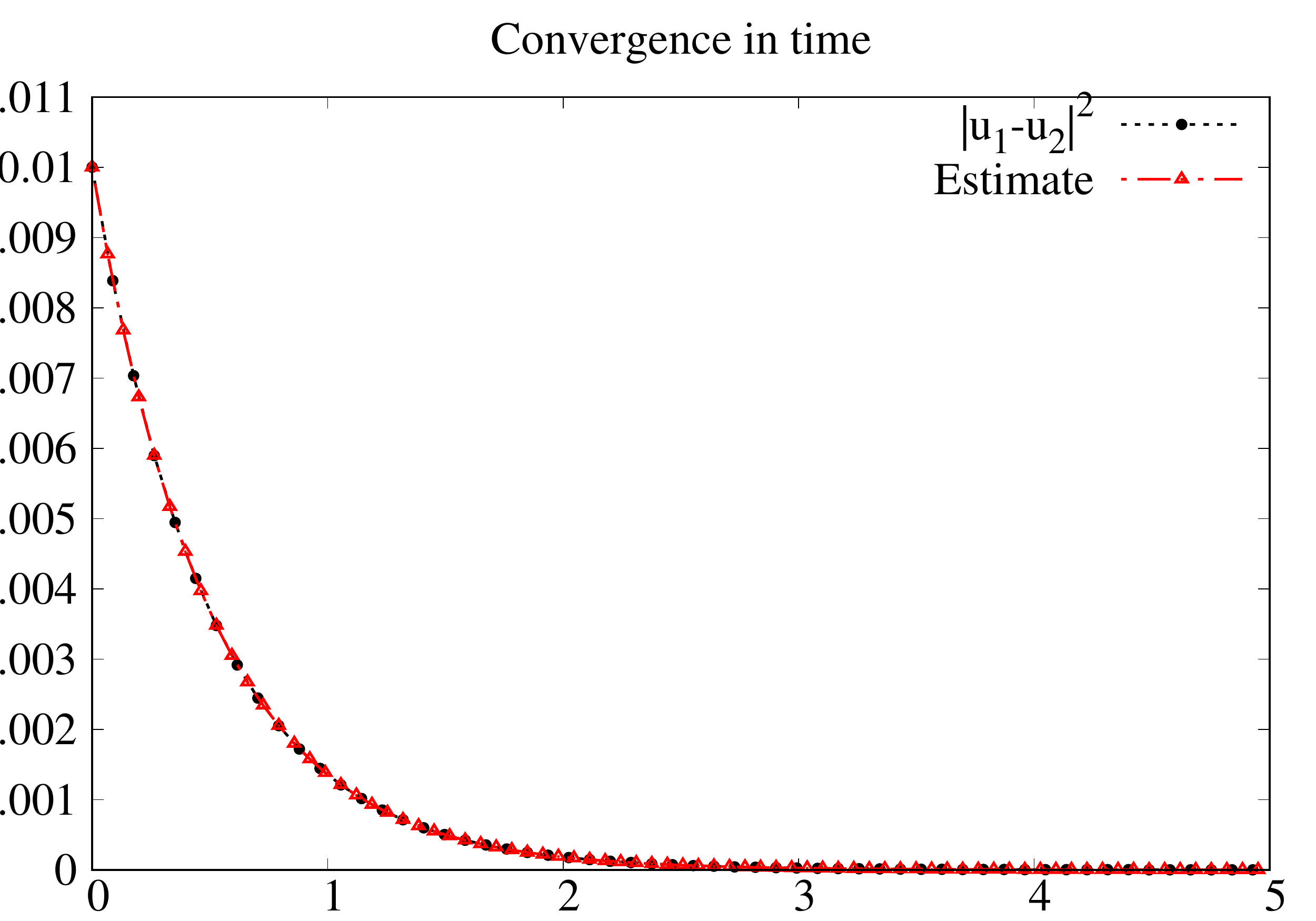}
\includegraphics[angle=-90,width=0.49\textwidth]{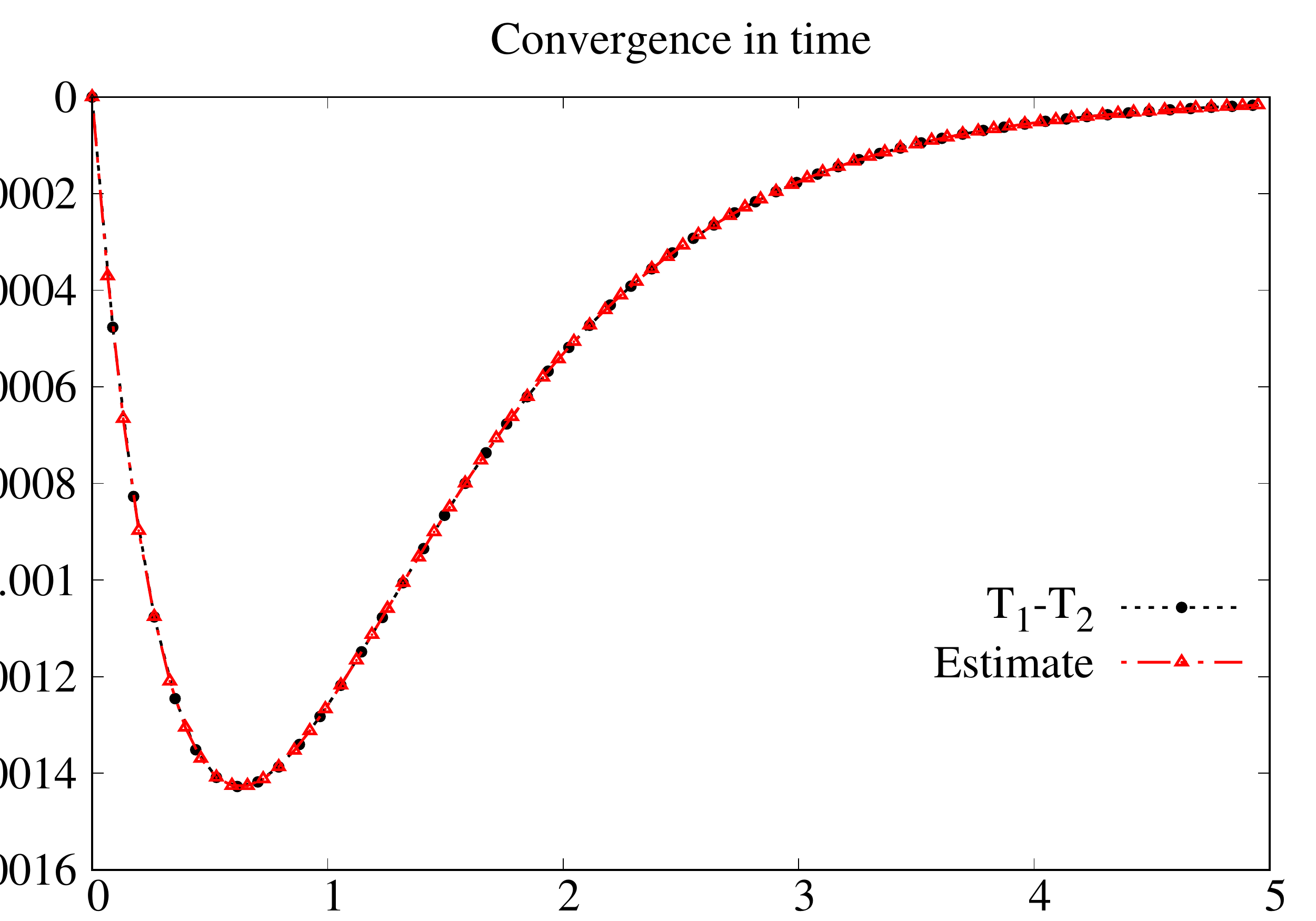}
\caption{Space-homogeneous case. Mixed initial conditions. Evolution in time of $|u_1(t) - u_2(t)|^2$ (left) and $|T_1(t) - T_2(t)|$ (right). Comparison to the estimated decay rates. Knudsen numbers: $\varepsilon_1=\varepsilon_2=\tilde{\varepsilon}_1=\tilde{\varepsilon}_2=1$.}
\label{fig:vel_maxJin_v4_2}
\end{figure}

\subsection{Relaxation towards a global equilibrium}

We present here numerical results in the general (non homogeneous) case. We consider micro-macro equations \eqref{micro1}-\eqref{macro1}-\eqref{micro2}-\eqref{macro2} and discretize them as explained in section \ref{sec:numapp}.

We are interested in the evolution in time of the distribution functions $f_1$, $f_2$ and other quantities such as 
the difference of  the mean velocities of species 1 and species 2 (resp. temperatures) in uniform norm $||u_1(x,t)-u_2(x,t)||_\infty$ (resp. $||T_1(x,t)-T_2(x,t)||_\infty$). Different values of $\varepsilon_1$, $\varepsilon_2$, $\tilde{\varepsilon_1}$ and $\tilde{\varepsilon_2}$ are considered in order to see the influence of the intra and interspecies collision frequencies.

In the following tests, species 1 and 2 are initialized following
\begin{align}
f_2(x,v,t=0)&=\left(1+\textcolor{black}{\beta} \cos(x/2)\right)\frac{v^4}{3\sqrt{2\pi}}\exp\left(-\frac{|v|^2}{2}\right),\\
f_1(x,v,t=0)&=\frac{1}{\sqrt{2\pi}}\exp\left(-\frac{|v-1/2|^2}{2}\right).
\end{align}
So, for $\beta\neq 0$, species 2 has initially a space dependent distribution. From the computation of $\langle m(v)f_2\rangle$, we obtain $n_2(x,0)=1+\textcolor{black}{\beta}\cos(kx)$, $u_2(x,0)=0$ and $T_2(x,0)=5\left(1+\textcolor{black}{\beta}\cos(kx)\right)$. Species 1 has initially a Maxwellian distribution with $n_1(x,0)=1$, $u_1(x,0)=1/2$ and $T_1(x,0)=1$. Here, we have taken $m_2=m_1=1$. 

For $\beta=0.1$, we illustrate the initial distribution functions on figure \ref{fig:init_0p1}, $f_2(x,v,t=0)$ is presented on the left, $f_1(x,v,t=0)$ on the middle and a side view of them on the right. 

\begin{figure}[ht]
\includegraphics[angle=-90,trim=3cm 4cm 4cm 3.5cm,width=0.33\textwidth]{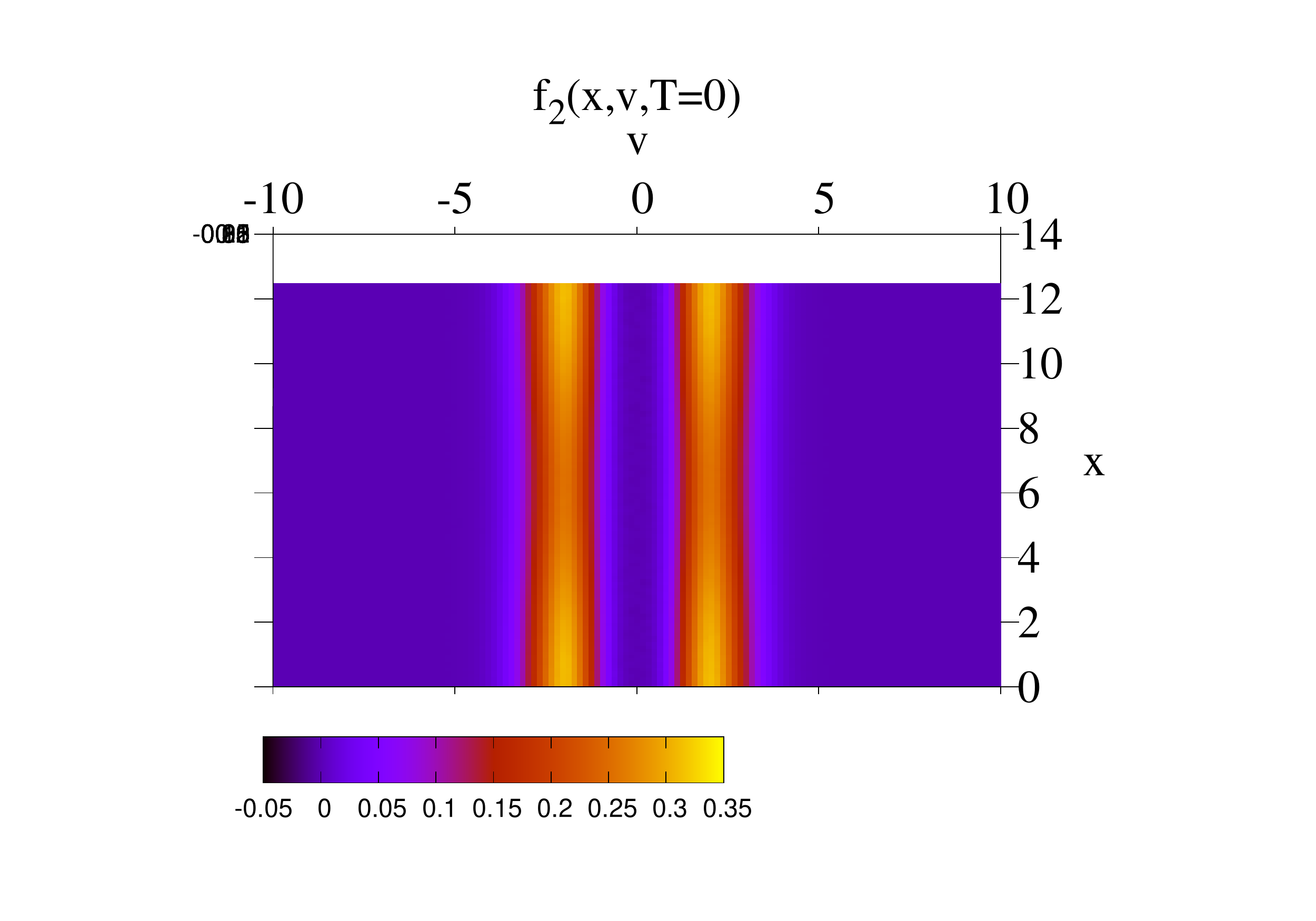}
\includegraphics[angle=-90,trim=3cm 4cm 4cm 3.5cm,width=0.33\textwidth]{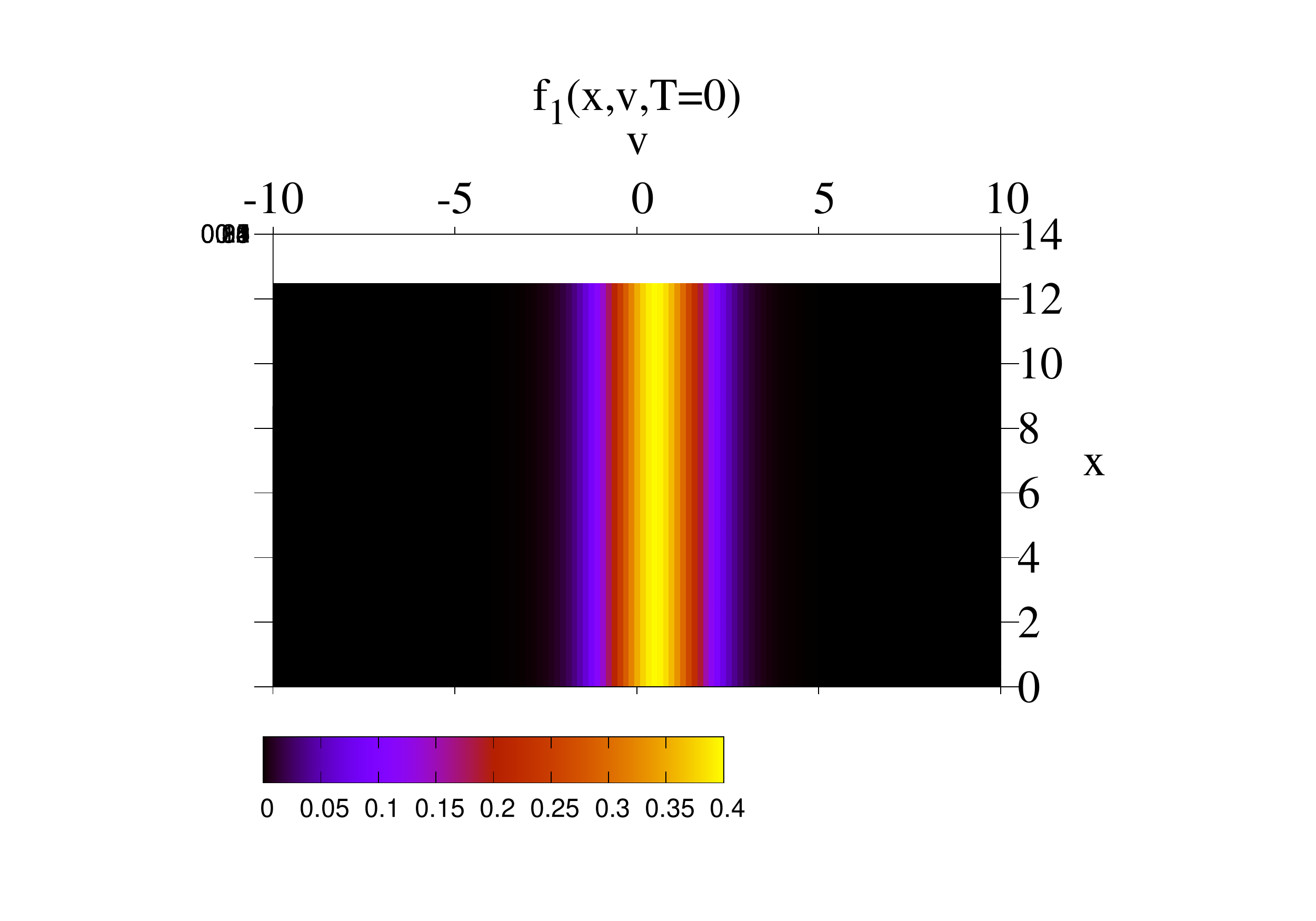}
\includegraphics[angle=-90,trim=2cm 0cm 0cm 1.cm,width=0.32\textwidth]{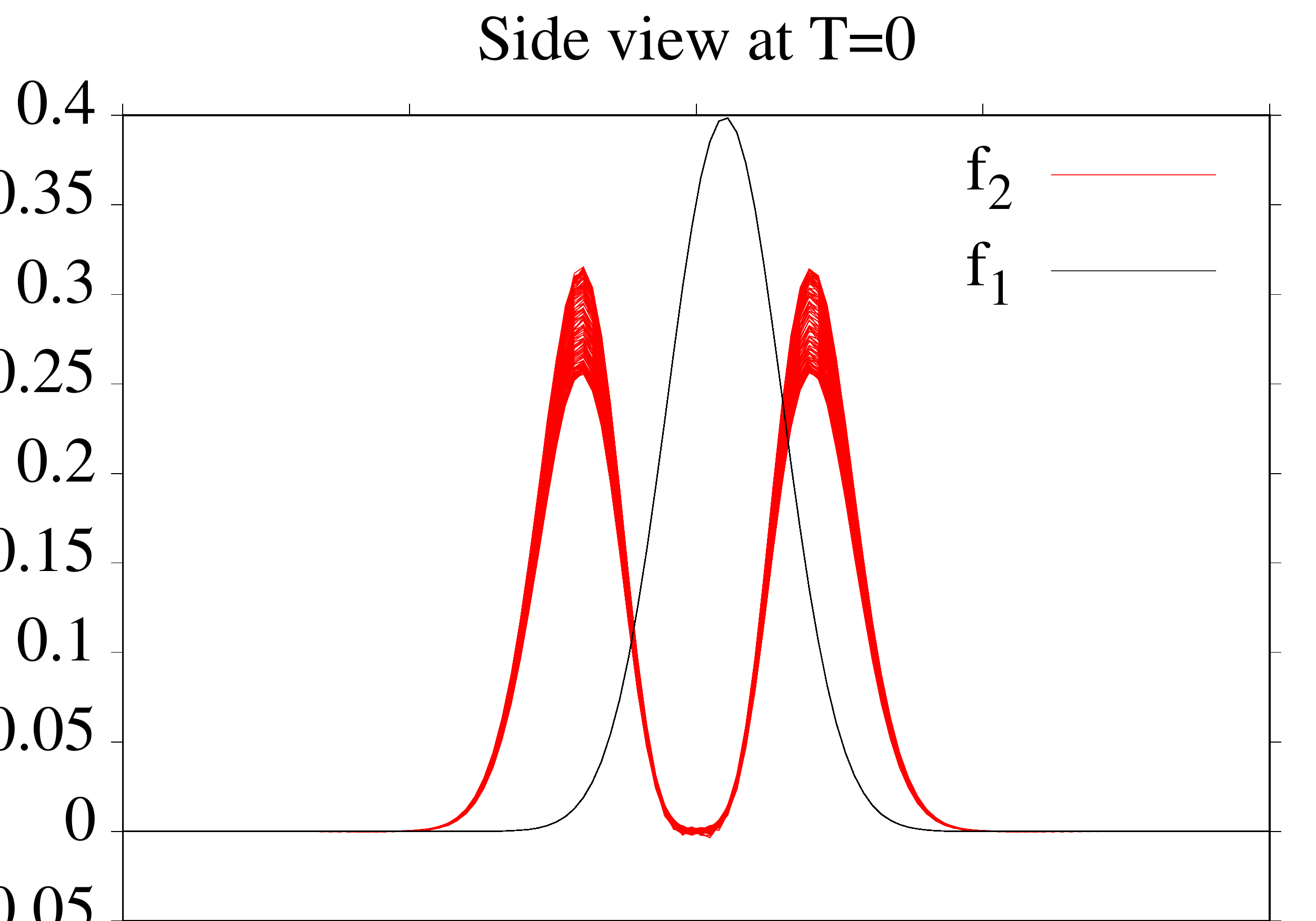}\vspace{0.1cm}
\caption{General case. Initial distribution functions for $\beta=0.1$: $f_2(x,v,t=0)$ in phase-space (left), $f_1(x,v,t=0)$ in phase-space (middle), side view of $f_2(x,v,t=0)$ and $f_1(x,v,t=0)$ (right).}
\label{fig:init_0p1}
\end{figure}

First, we propose two testcases with the following parameters: $\beta=0.1$, $N_{p_2}=N_{p_1}=5\cdot 10^5$, $N_x=128$ and $\Delta t=10^{-2}$. The first one consists in the kinetic regime $\varepsilon_1=\varepsilon_2=\tilde{\varepsilon_1}=\tilde{\varepsilon_2}=1000$, collision frequencies are small and particles do not interact a lot with each other. Distribution functions are plotted at time $T=6$ on figure \ref{fig:1000_T6} and at time $T=60$ on figure \ref{fig:1000_T60}.

\begin{figure}[ht]
\includegraphics[angle=-90,trim=3cm 4cm 4cm 3.5cm,width=0.33\textwidth]{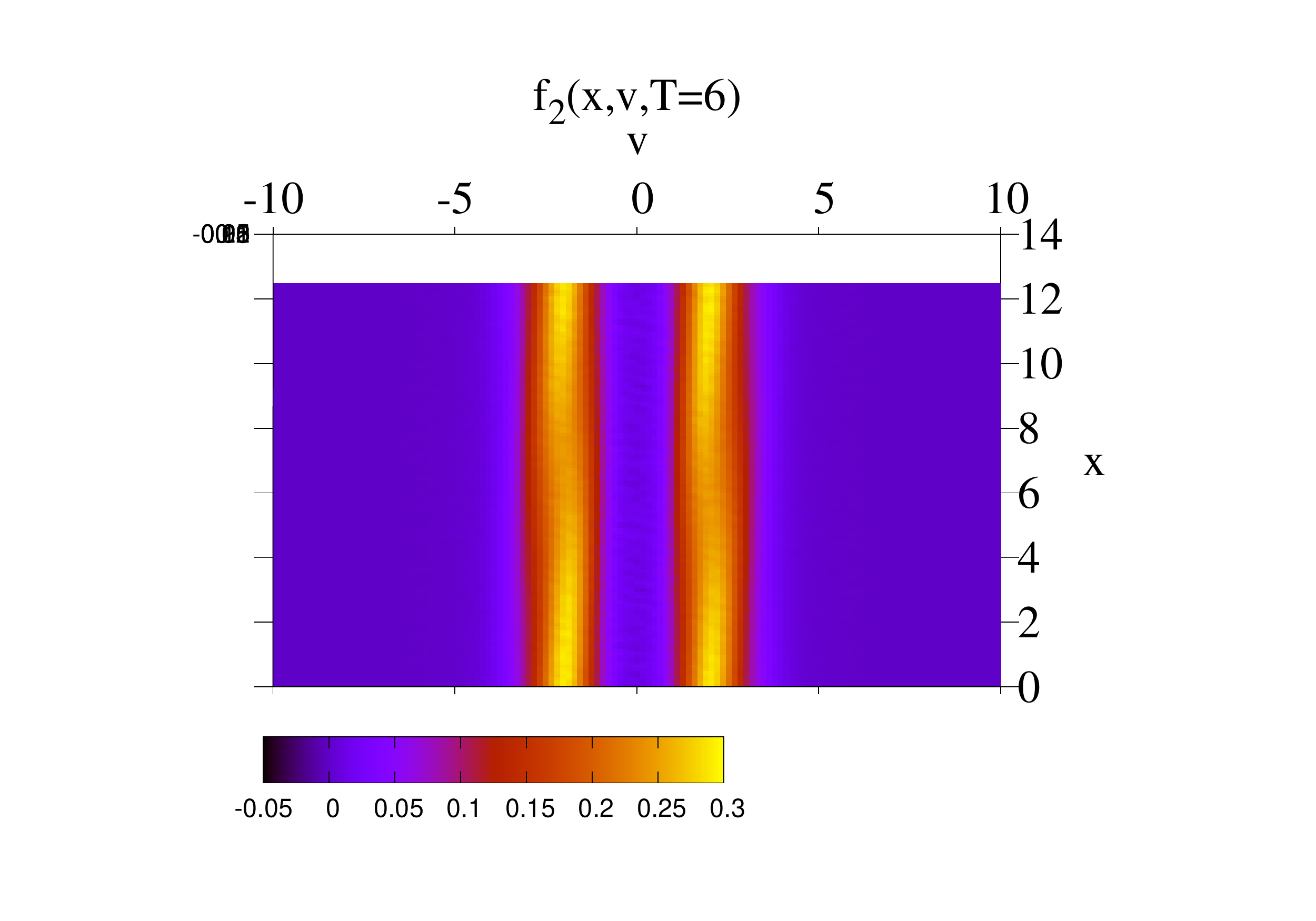}
\includegraphics[angle=-90,trim=3cm 4cm 4cm 3.5cm,width=0.33\textwidth]{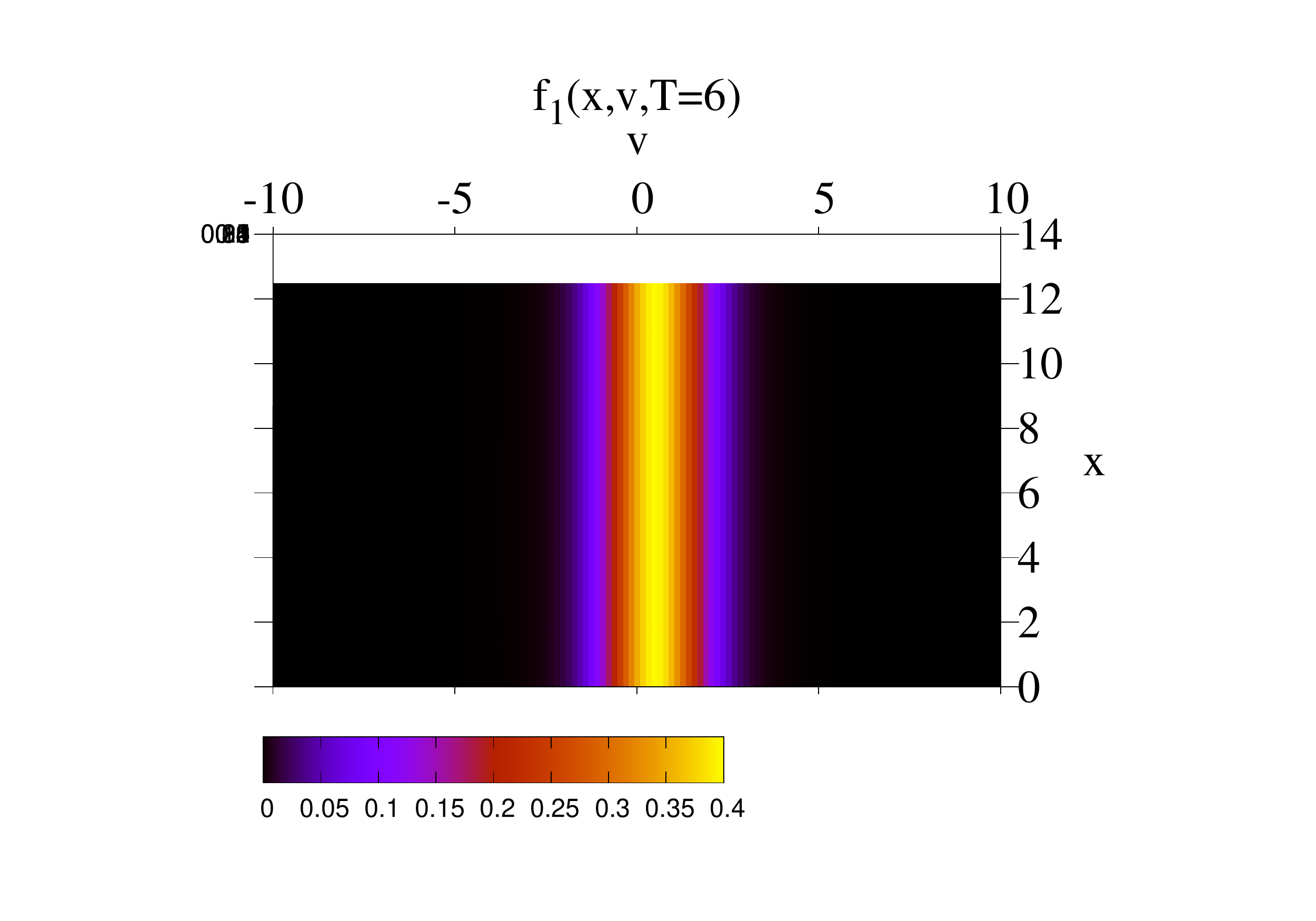}
\includegraphics[angle=-90,trim=2cm 0cm 0cm 1cm,width=0.32\textwidth]{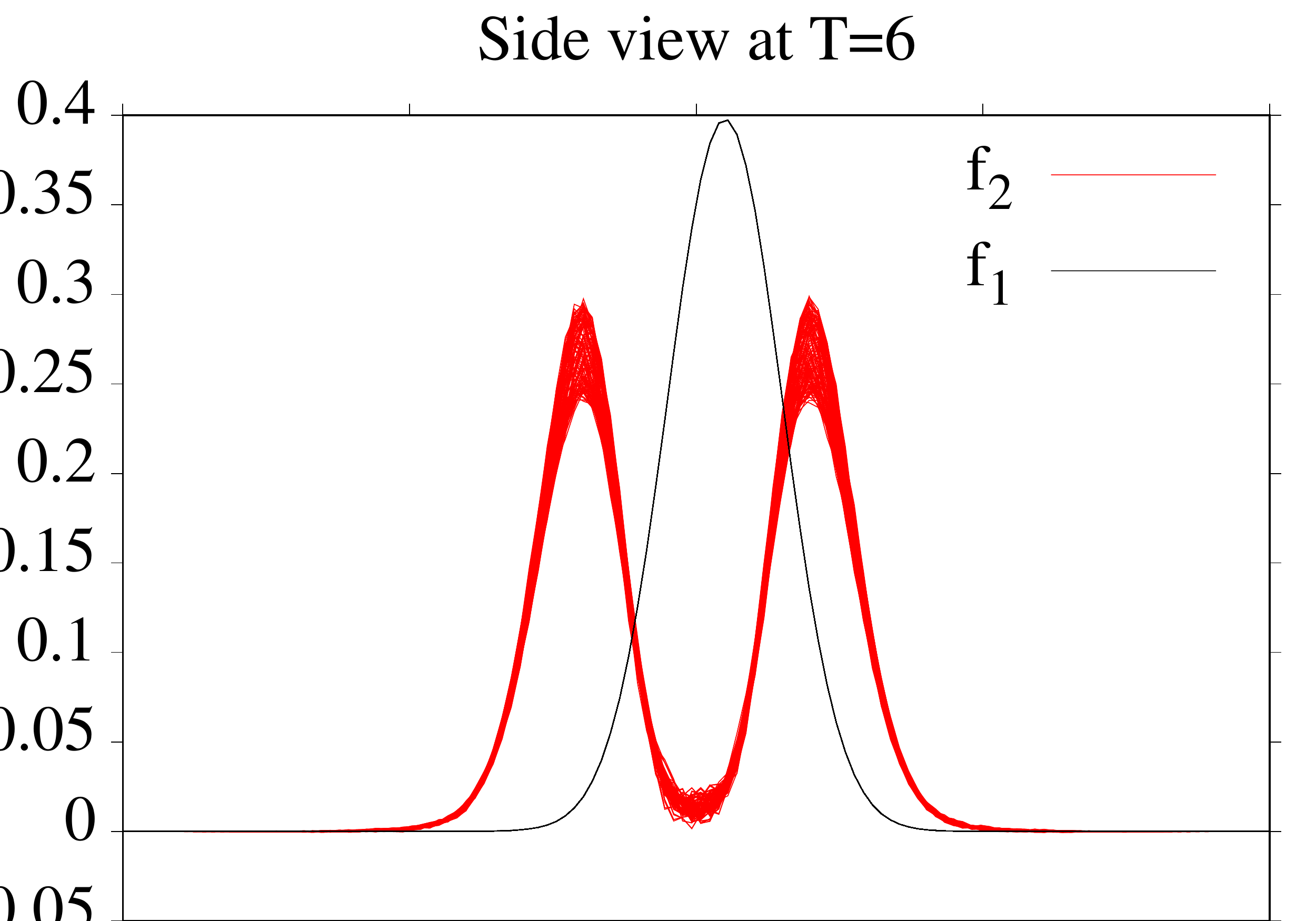}\vspace{0.1cm}
\caption{General case, $\beta=0.1$, $\varepsilon_1=\varepsilon_2=\tilde{\varepsilon_1}=\tilde{\varepsilon_2}=1000$. Distribution functions at time $T=6$: $f_2(x,v,T)$ in phase-space (left), $f_1(x,v,T)$ in phase-space (middle), side view of $f_2(x,v,T)$ and $f_1(x,v,T)$ (right).}
\label{fig:1000_T6}
\end{figure}

\begin{figure}[ht]
\includegraphics[angle=-90,trim=3cm 4cm 4cm 3.5cm,width=0.33\textwidth]{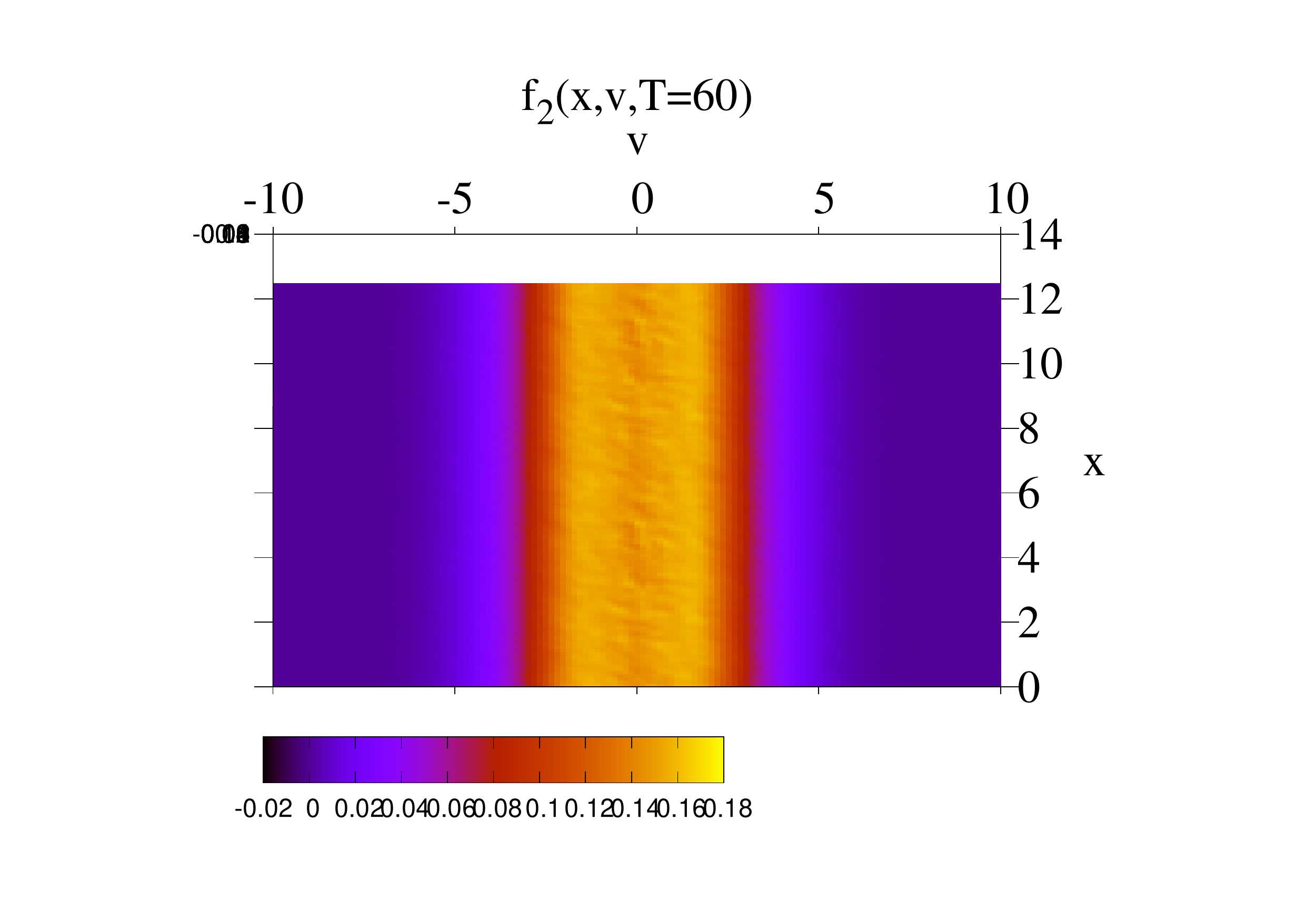}
\includegraphics[angle=-90,trim=3cm 4cm 4cm 3.5cm,width=0.33\textwidth]{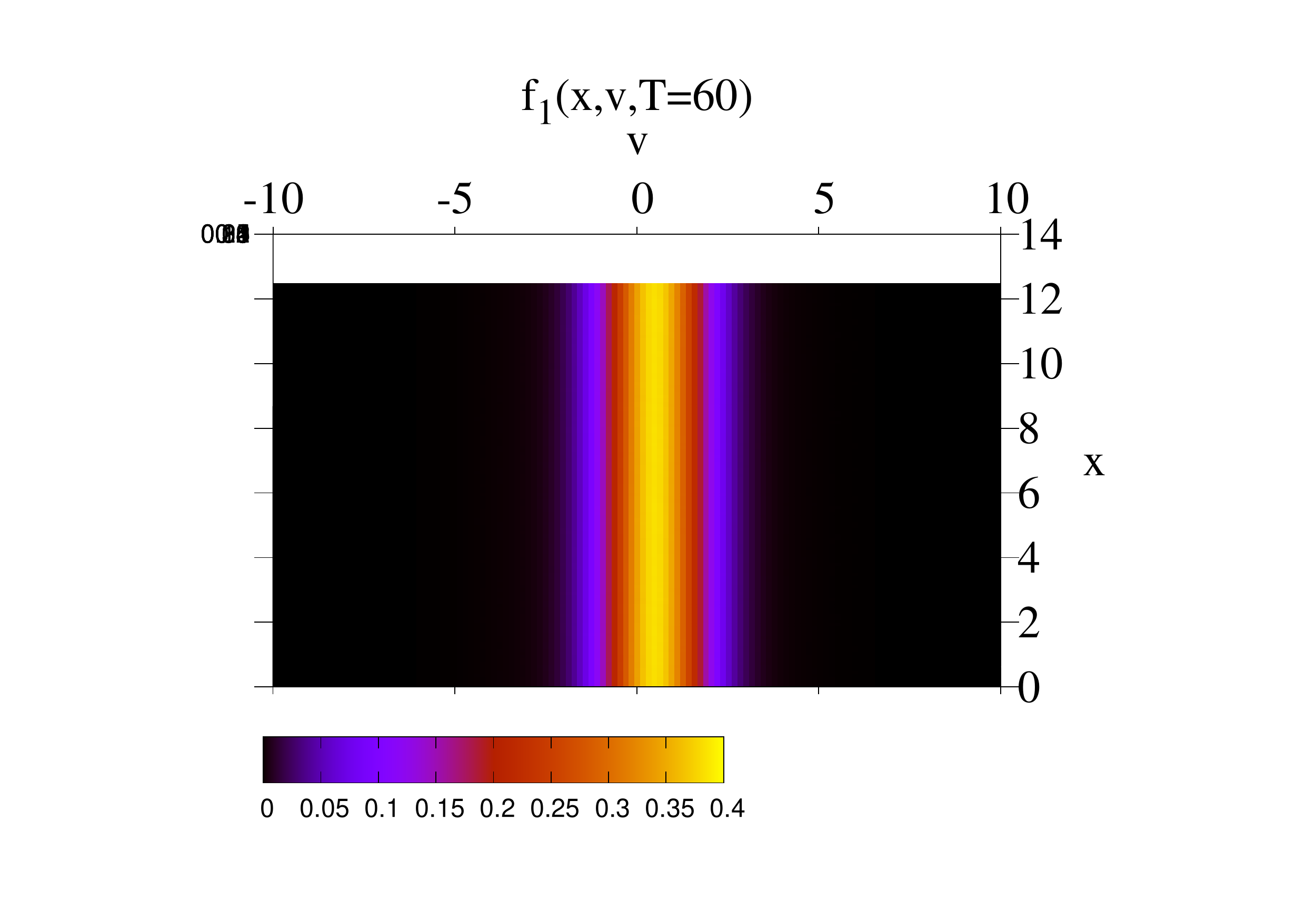}
\includegraphics[angle=-90,trim=2cm 0cm 0cm 1cm,width=0.32\textwidth]{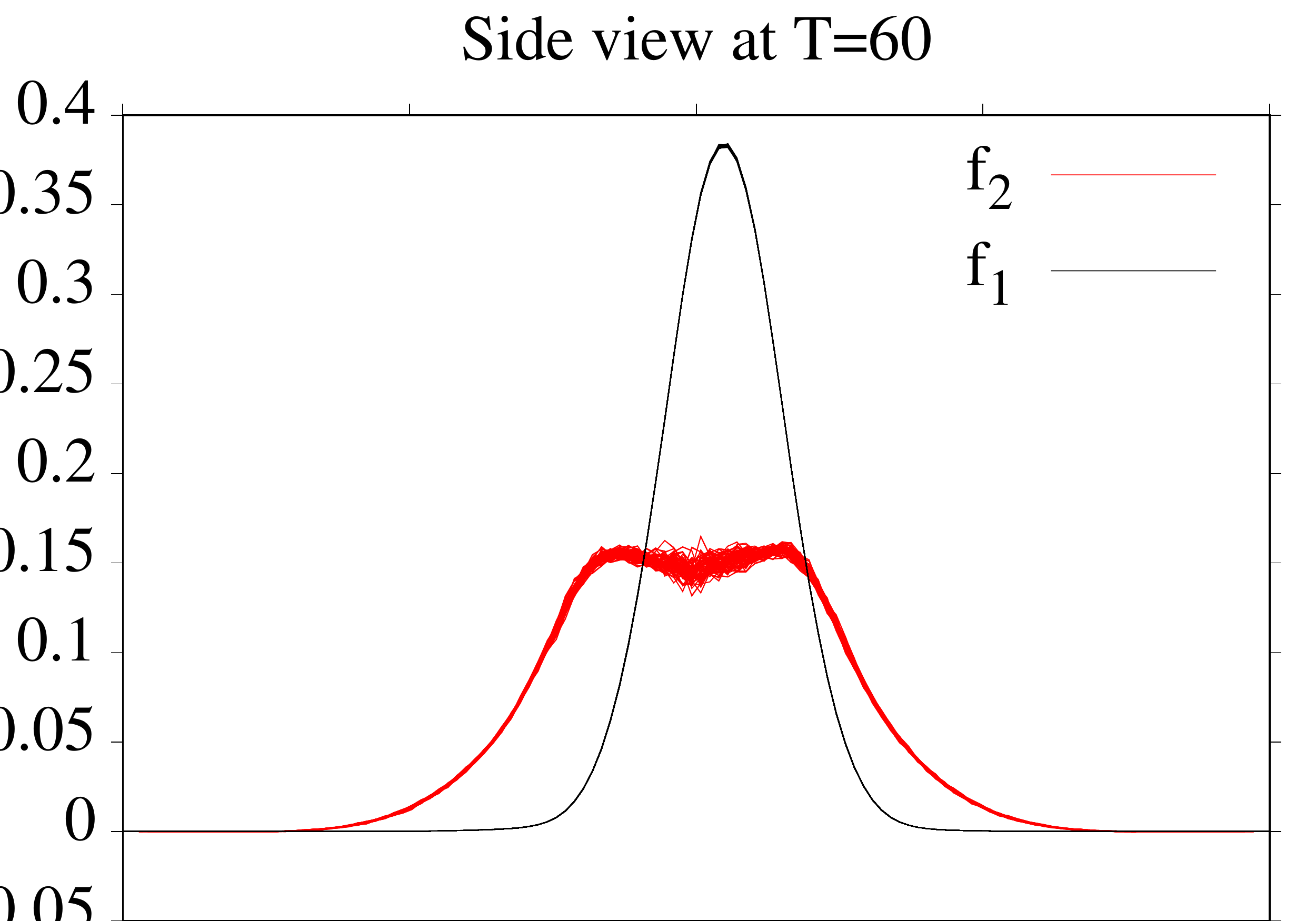}\vspace{0.1cm}
\caption{General case, $\beta=0.1$, $\varepsilon_1=\varepsilon_2=\tilde{\varepsilon_1}=\tilde{\varepsilon_2}=1000$. Distribution functions at time $T=60$: $f_2(x,v,T)$ in phase-space (left), $f_1(x,v,T)$ in phase-space (middle), side view of $f_2(x,v,T)$ and $f_1(x,v,T)$ (right).}
\label{fig:1000_T60}
\end{figure}

For these values of \textcolor{black}{Knudsen numbers}, the convergence of $f_2$ towards its equilibrium $M_2$ is slow. Moreover, even at time $T=60$, the convergence towards a global equilibrium $f_2=M_2=M_1=f_1$ can not be seen. To see the difference on macroscopic quantities, we present on figure \ref{fig:1000_vte} the evolution in time of $||u_1(x,t)-u_2(x,t)||_\infty$ and $||T_1(x,t)-T_2(x,t)||_\infty$. 
 
 \begin{figure}[!htb]
 \begin{center}
\includegraphics[angle=-90,width=0.49\textwidth]{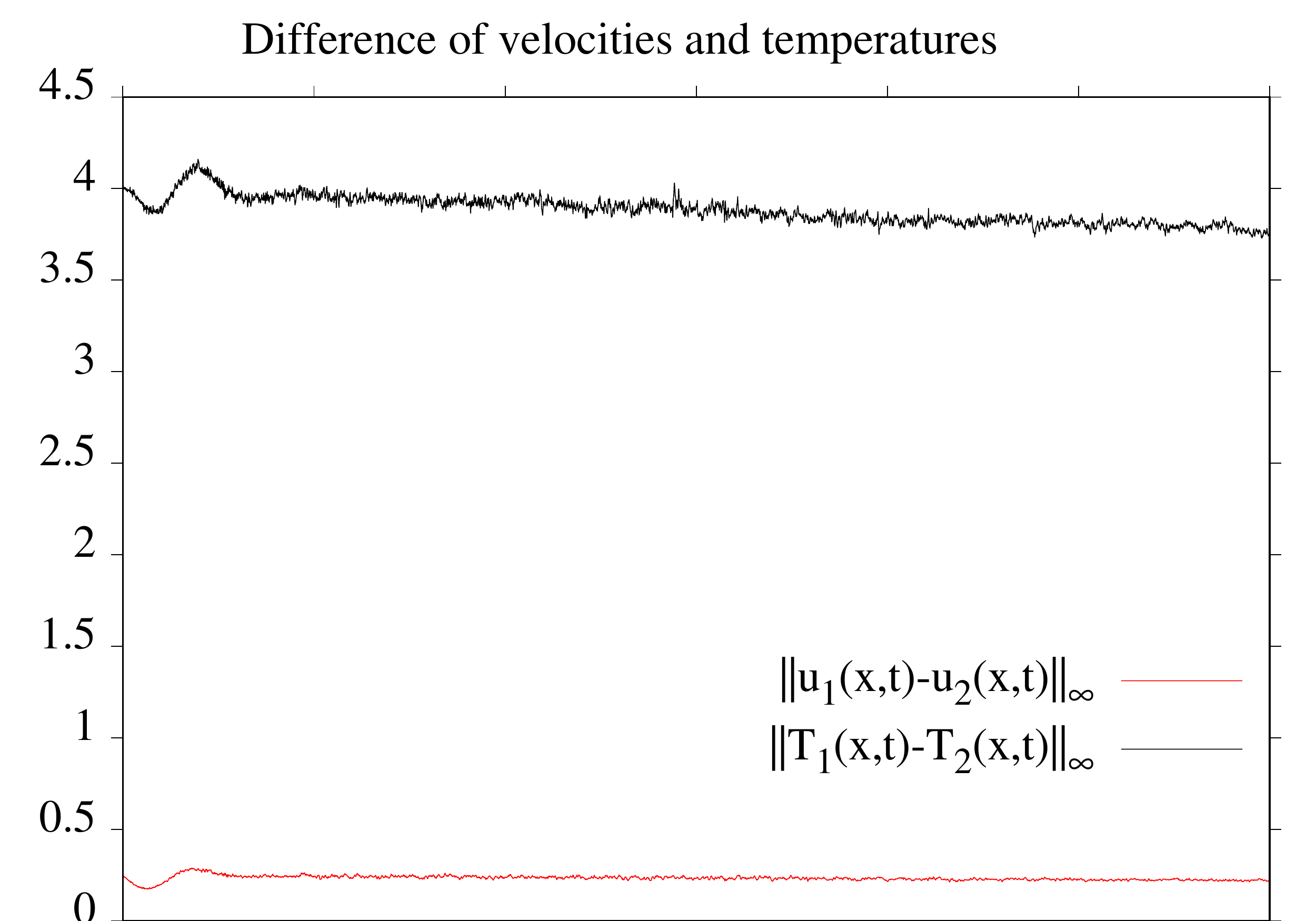}
\caption{General case, $\beta=0.1$, $\varepsilon_1=\varepsilon_2=\tilde{\varepsilon_1}=\tilde{\varepsilon_2}=1000$. Evolution in time of $||u_1(x,t)-u_2(x,t)||_\infty$ and $||T_1(x,t)-T_2(x,t)||_\infty$.
}
\label{fig:1000_vte}
\end{center}
\end{figure}

Even at time $T=60$, the velocities (resp. temperatures) of species 1 and species 2 are very different. There is no global equilibrium.  

Otherwise, these figures show that the results are affected by some numerical noise. This is a classical effect of particle methods, due to the probabilistic character of the initialisation. This noise affects macroscopic quantities because of the coupling between micro and macro equations. At fixed parameters ($\beta$, collision frequencies, $N_x$, \textit{etc.}), the noise can be reduced by increasing the number of particles. In fact, the noise means that we have not enough particles per cell to represent the distribution function ($g_{22}$ or $g_{11}$ here). But thanks to the micro-macro decomposition, we only represent the perturbations $g_{22}$ and $g_{11}$ with particles, and not the whole functions $f_2$ and $f_1$. So when $g_{22}$ (resp. $g_{11}$) becomes smaller, fewer particles are necessary. It means that if $f_2$ (resp. $f_1$) goes towards its equilibrium $M_2$ (resp. $M_1$), the required number of particles diminishes. This is the main reason for using a micro-macro scheme with a particle method for the micro part.

The second testcase consists in an intermediate regime with $\varepsilon_1=\varepsilon_2=\tilde{\varepsilon_1}=\tilde{\varepsilon_2}=1$. Collisions are enough frequent to bring the system towards a global equilibrium, as we can see on figure \ref{fig:1_T0p5} at time $T=0.5$ and then on figure \ref{fig:1_T6} at time $T=6$.

\begin{figure}[ht]
\includegraphics[angle=-90,trim=3cm 4cm 4cm 3.5cm,width=0.33\textwidth]{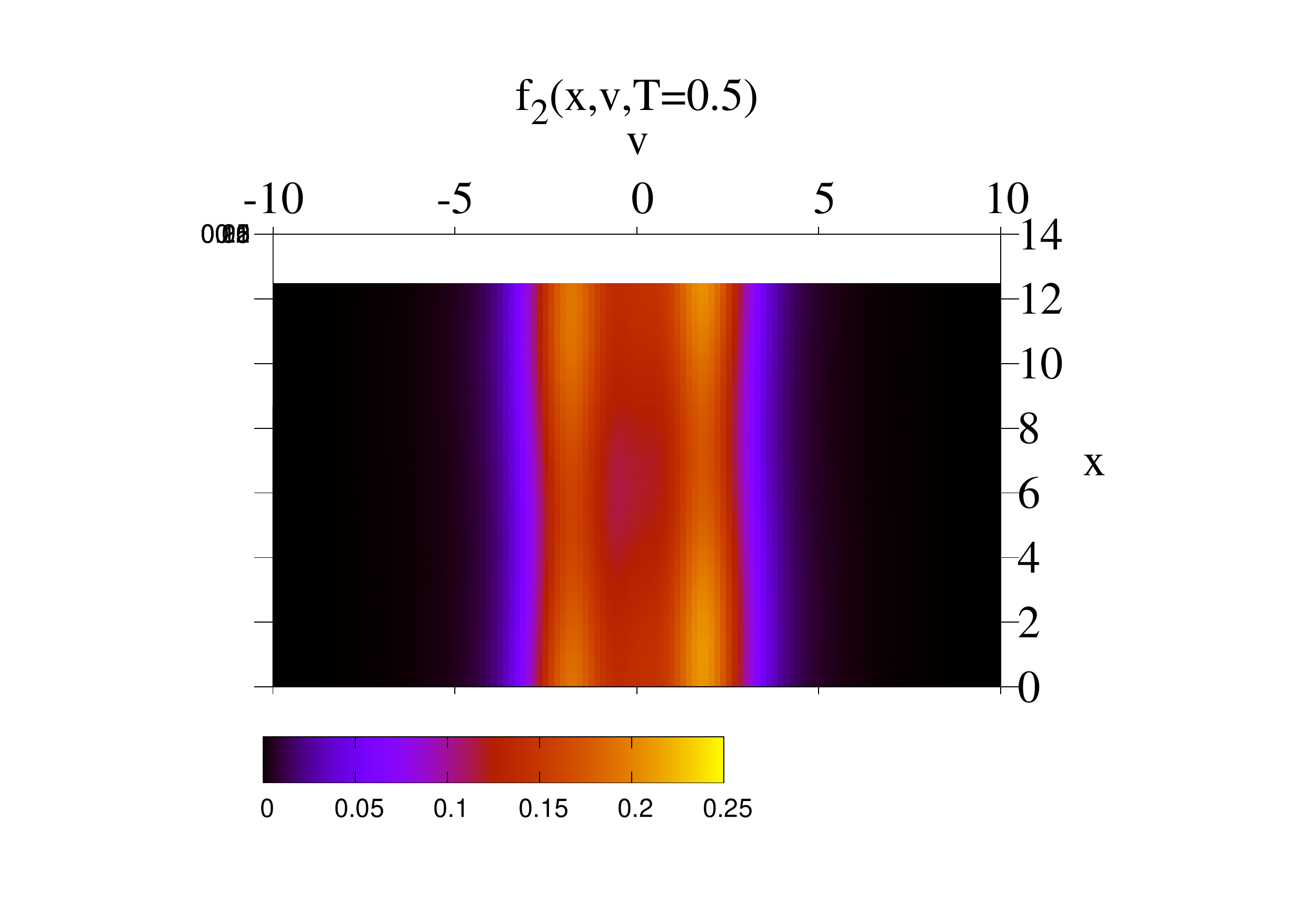}
\includegraphics[angle=-90,trim=3cm 4cm 4cm 3.5cm,width=0.33\textwidth]{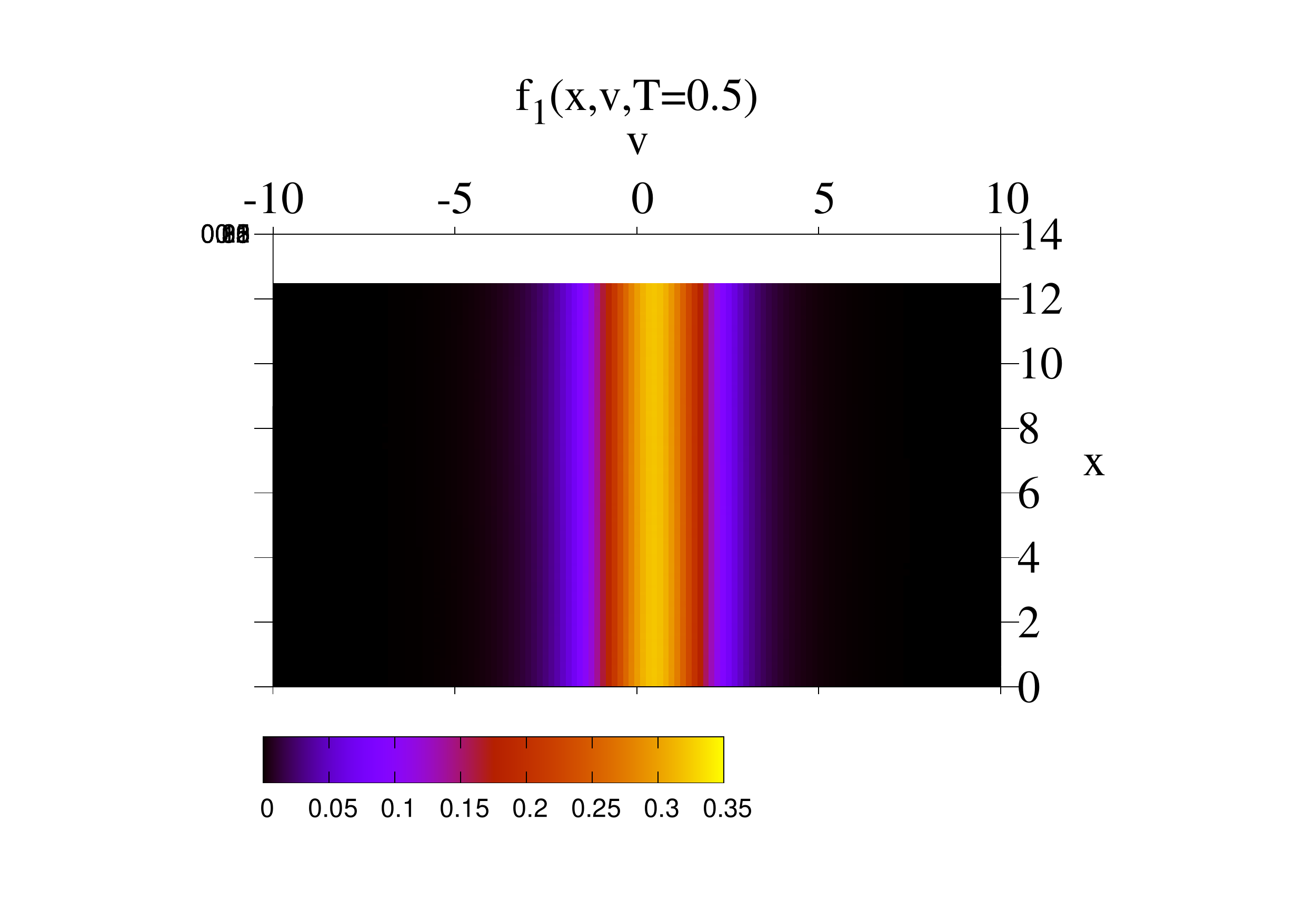}
\includegraphics[angle=-90,trim=2cm 0cm 0cm 1cm,width=0.32\textwidth]{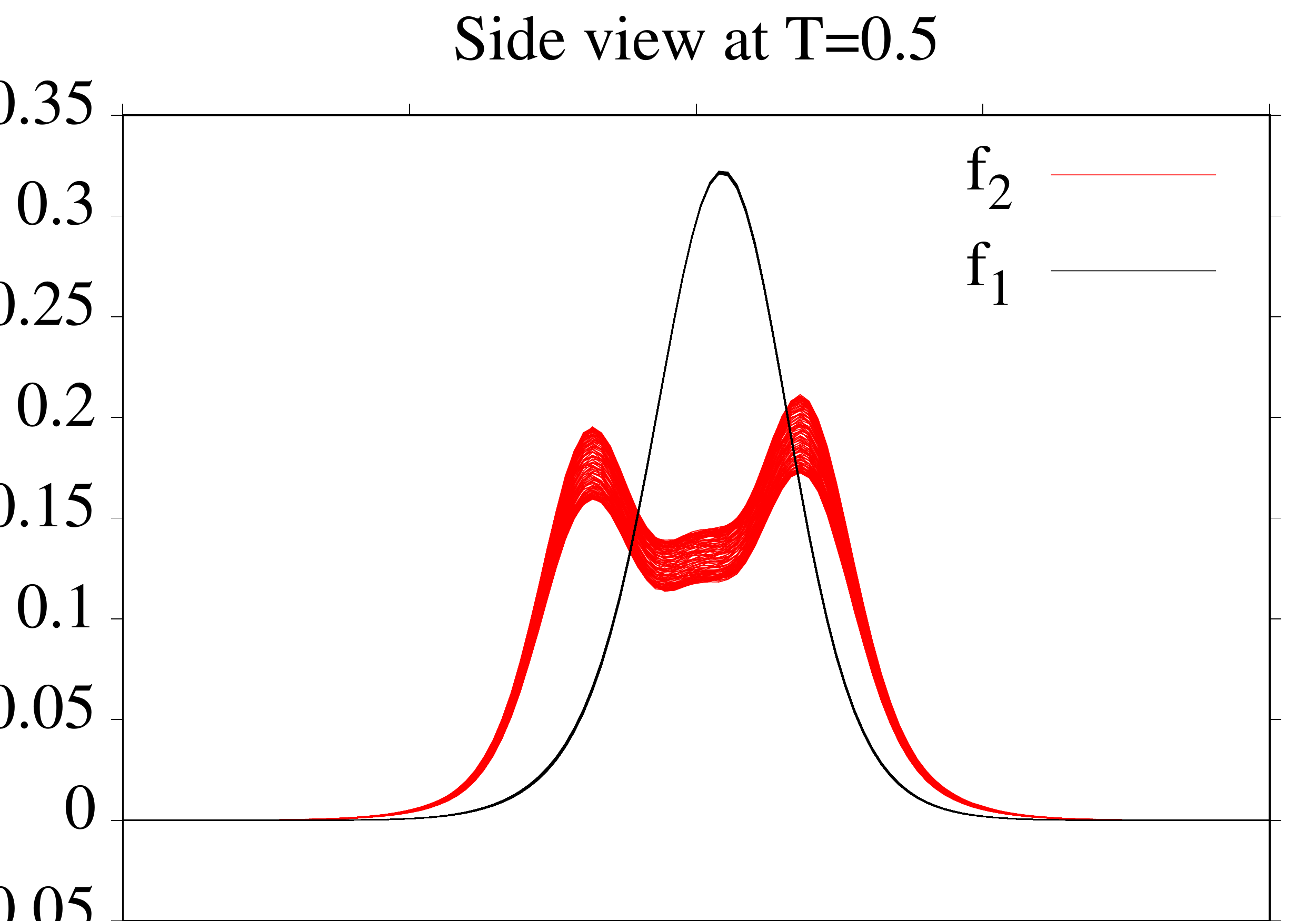}\vspace{0.1cm}
\caption{General case, $\beta=0.1$, $\varepsilon_1=\varepsilon_2=\tilde{\varepsilon_1}=\tilde{\varepsilon_2}=1$. Distribution functions at time $T=0.5$: $f_2(x,v,T)$ in phase-space (left), $f_1(x,v,T)$ in phase-space (middle), side view of $f_2(x,v,T)$ and $f_1(x,v,T)$ (right).}
\label{fig:1_T0p5}
\end{figure}

\begin{figure}[ht]
\includegraphics[angle=-90,trim=3cm 4cm 4cm 3.5cm,width=0.33\textwidth]{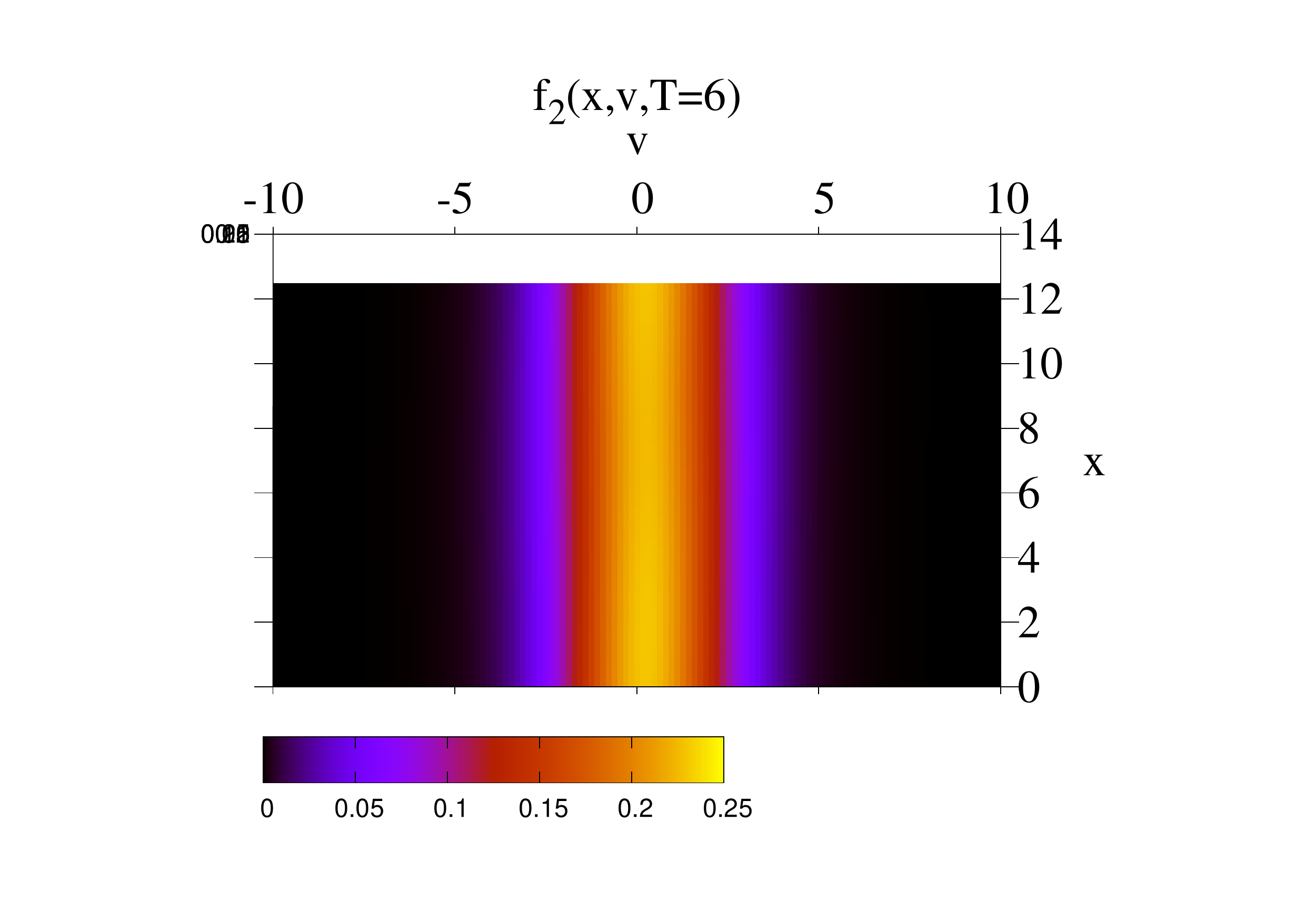}
\includegraphics[angle=-90,trim=3cm 4cm 4cm 3.5cm,width=0.33\textwidth]{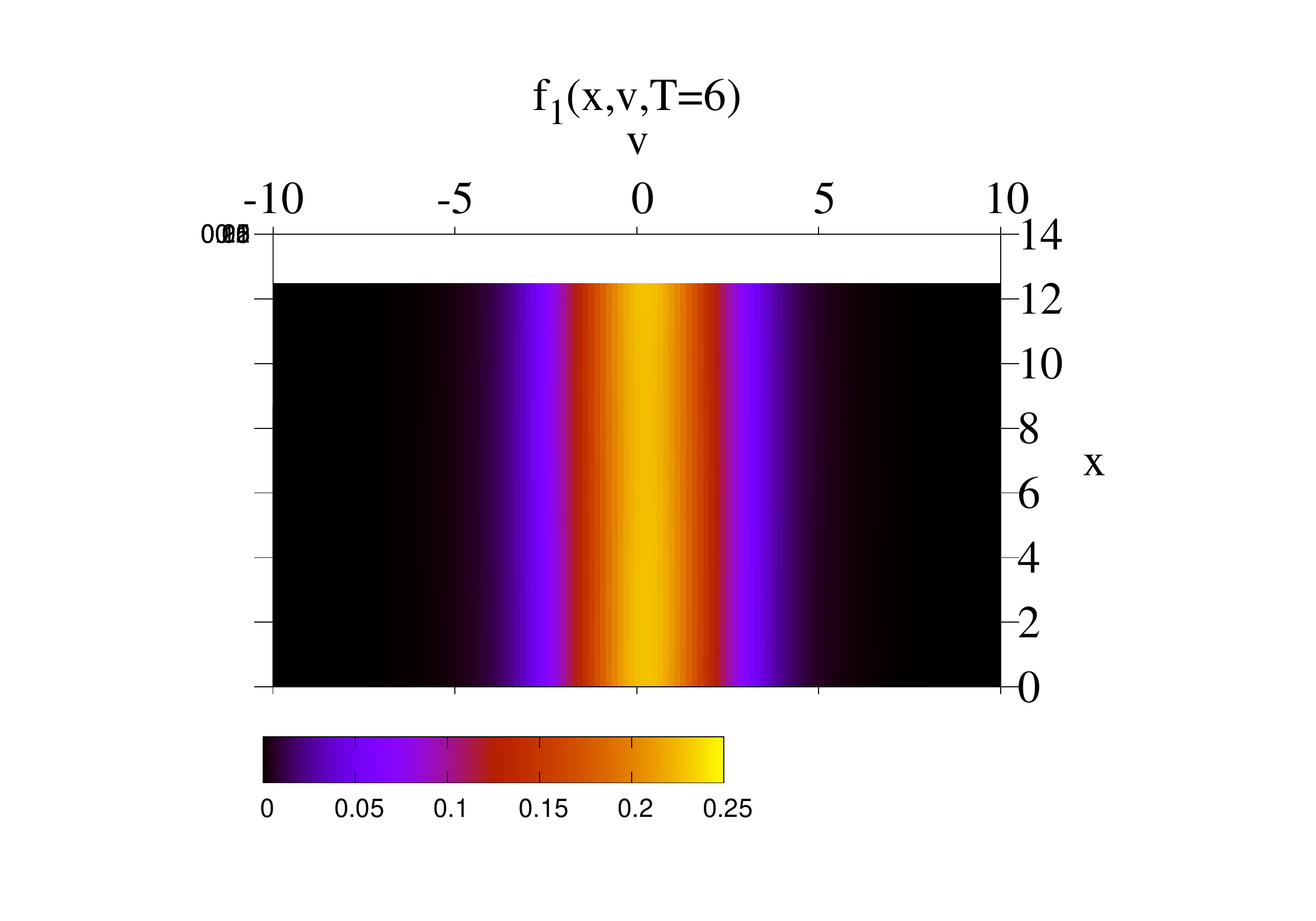}
\includegraphics[angle=-90,trim=2cm 0cm 0cm 1cm,width=0.32\textwidth]{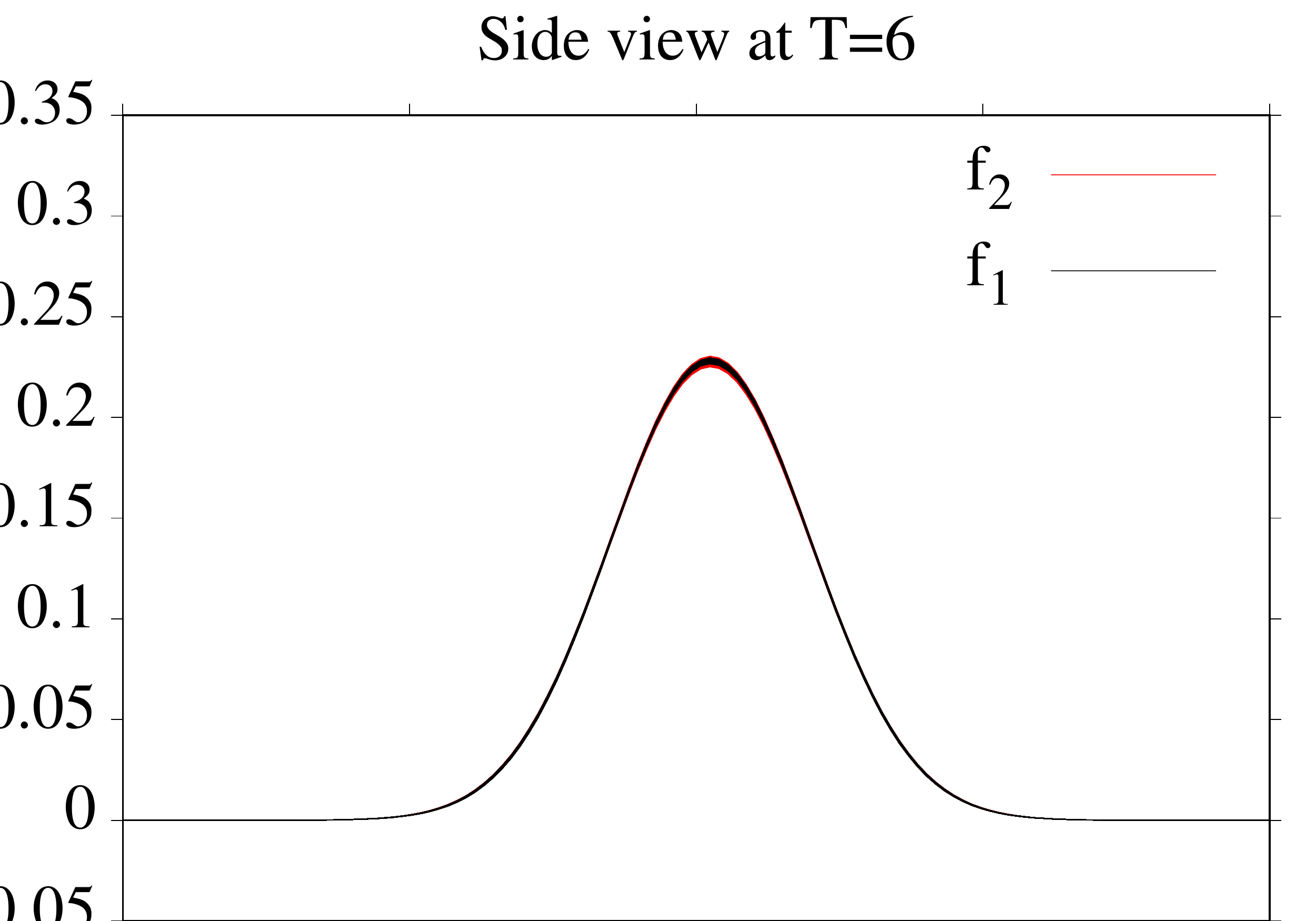}\vspace{0.1cm}
\caption{General case, $\beta=0.1$, $\varepsilon_1=\varepsilon_2=\tilde{\varepsilon_1}=\tilde{\varepsilon_2}=1$. Distribution functions at time $T=6$: $f_2(x,v,T)$ in phase-space (left), $f_1(x,v,T)$ in phase-space (middle), side view of $f_2(x,v,T)$ and $f_1(x,v,T)$ (right).}
\label{fig:1_T6}
\end{figure}

The evolution in time of $||u_1(x,t)-u_2(x,t)||_\infty$ and $||T_1(x,t)-T_2(x,t)||_\infty$, presented on figure \ref{fig:1_vte}, confirms the convergence towards a global equilibrium.

 \begin{figure}[ht]
 \begin{center}
\includegraphics[angle=-90,width=0.49\textwidth]{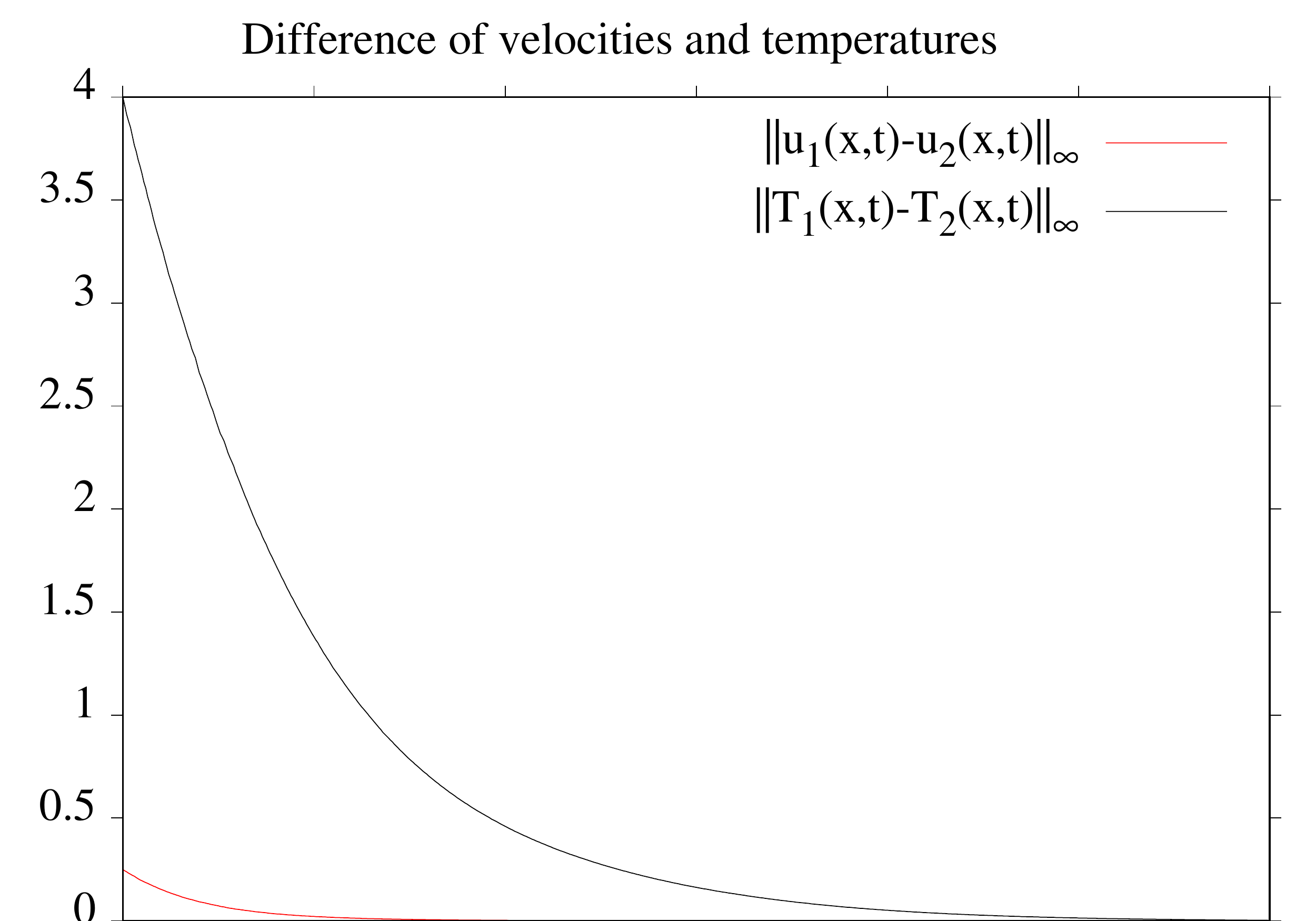}
\caption{General case, $\beta=0.1$, $\varepsilon_1=\varepsilon_2=\tilde{\varepsilon_1}=\tilde{\varepsilon_2}=1$. Evolution in time of $||u_1(x,t)-u_2(x,t)||_\infty$ and $||T_1(x,t)-T_2(x,t)||_\infty$.
}
\label{fig:1_vte}
\end{center}
\end{figure}

We expect that the convergence towards a global equilibrium is faster when collisions are more frequent. We will highlight this in the following test. For a convergence of the densities in short time, we now take $\beta=10^{-2}$. Other parameters are unchanged and particularly we still have \textcolor{black}{$N_{p_2}=N_{p_1}=5\cdot 10^5$, $N_x=128$, $\Delta t=10^{-2}$ and} $n_2(x,0)=1+\beta\cos(kx)$, $u_2(x,0)=0$, $T_2(x,0)=5\left(1+\beta\cos(kx)\right)$, $n_1(x,0)=1$, $u_1(x,0)=1/2$ and $T_1(x,0)=1$. For $\varepsilon_1=\varepsilon_2=\tilde{\varepsilon_1}=\tilde{\varepsilon_2}=10^{-2}$, distribution functions are plotted on figure \ref{fig:1e-2_T0p01} at time $T=0.01$ and then on figure \ref{fig:1e-2_T0p1} at time $T=0.1$. 

\begin{figure}[ht]
\includegraphics[angle=-90,trim=3cm 4cm 4cm 3.5cm,width=0.33\textwidth]{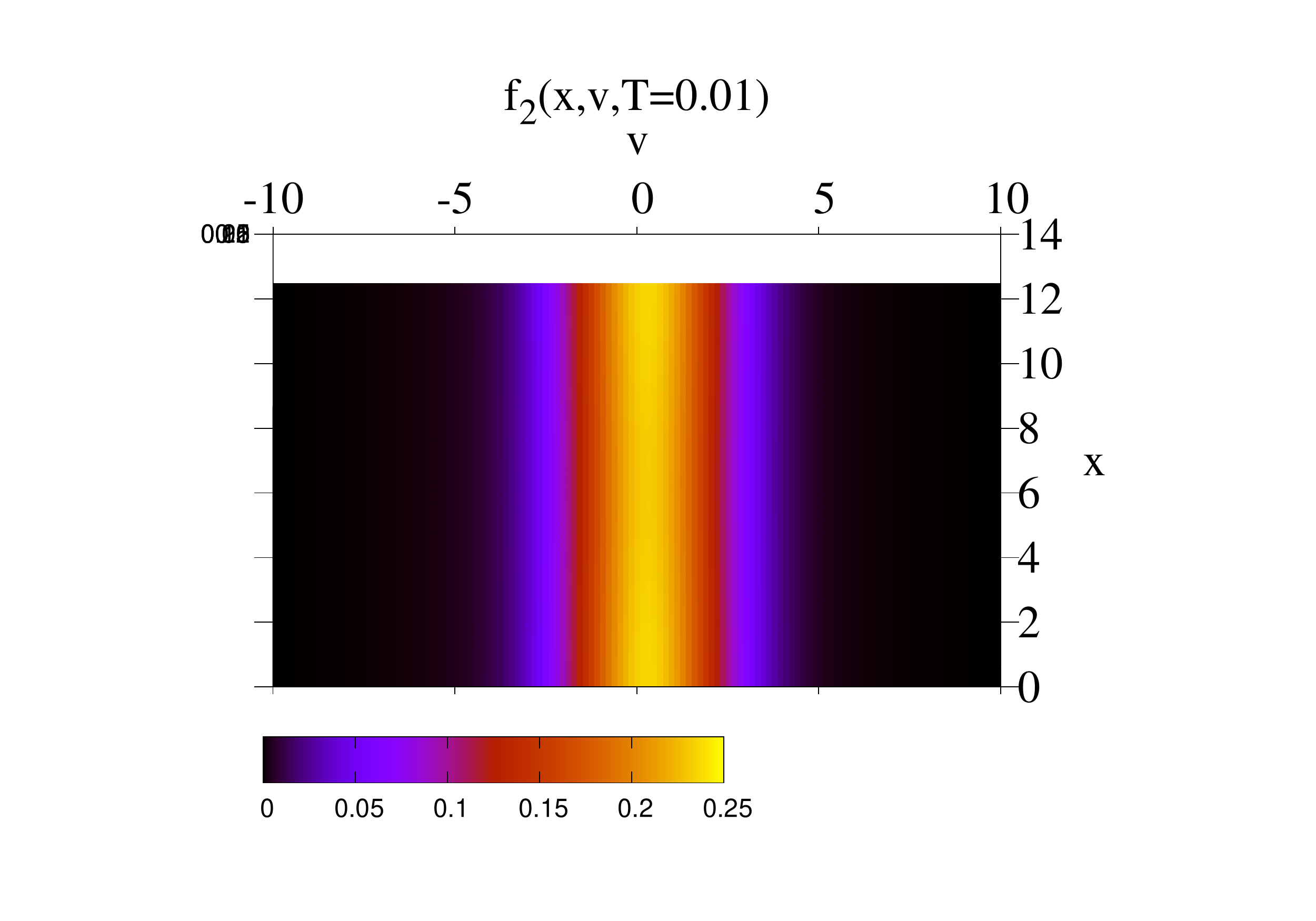}
\includegraphics[angle=-90,trim=3cm 4cm 4cm 3.5cm,width=0.33\textwidth]{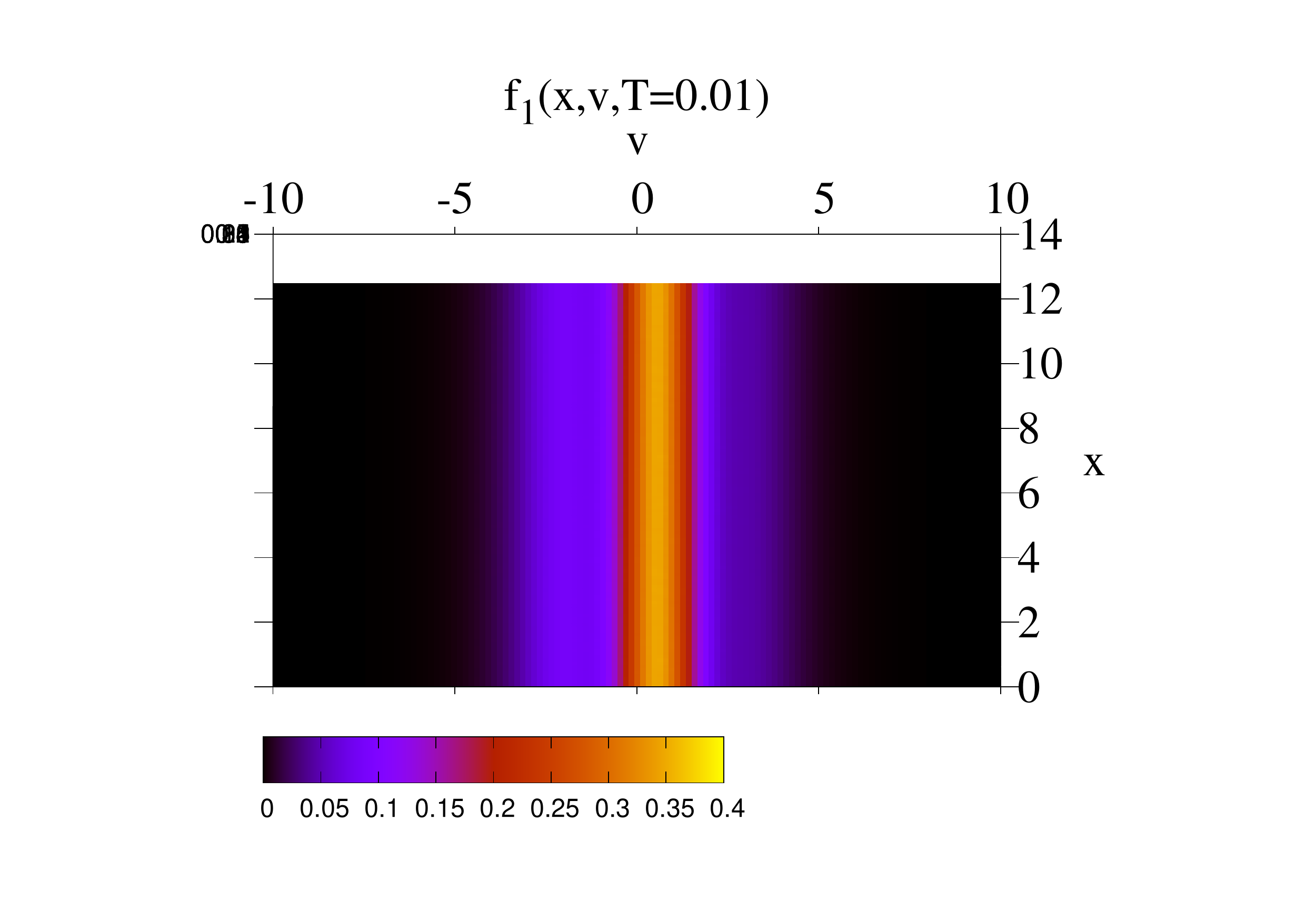}
\includegraphics[angle=-90,trim=2cm 0cm 0cm 1cm,width=0.32\textwidth]{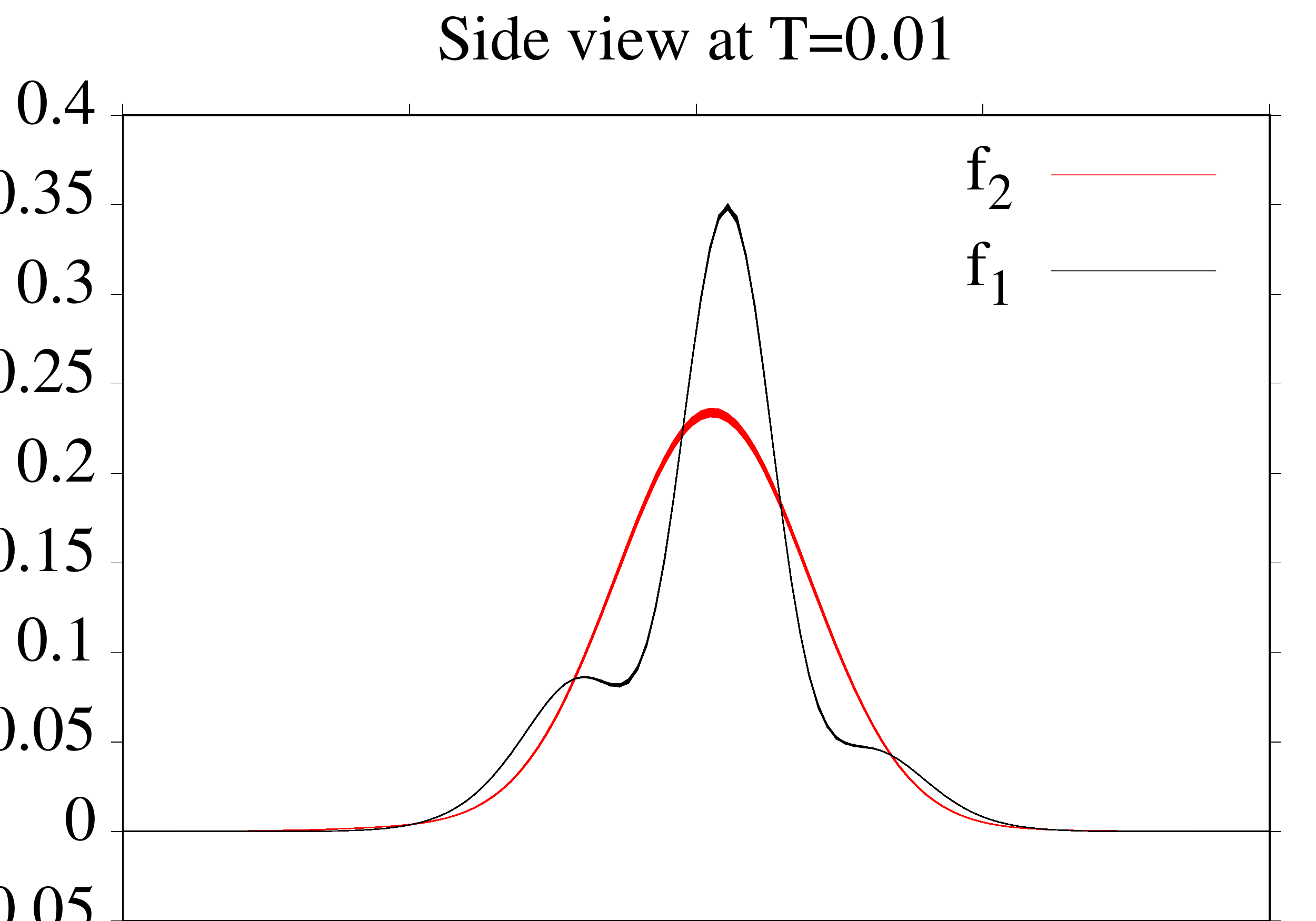}\vspace{0.1cm}
\caption{General case, $\beta=10^{-2}$, $\varepsilon_1=\varepsilon_2=\tilde{\varepsilon_1}=\tilde{\varepsilon_2}=10^{-2}$. Distribution functions at time $T=0.01$: $f_2(x,v,T)$ in phase-space (left), $f_1(x,v,T)$ in phase-space (middle), side view of $f_2(x,v,T)$ and $f_1(x,v,T)$ (right).}
\label{fig:1e-2_T0p01}
\end{figure}

\begin{figure}[ht]
\includegraphics[angle=-90,trim=3cm 4cm 4cm 3.5cm,width=0.33\textwidth]{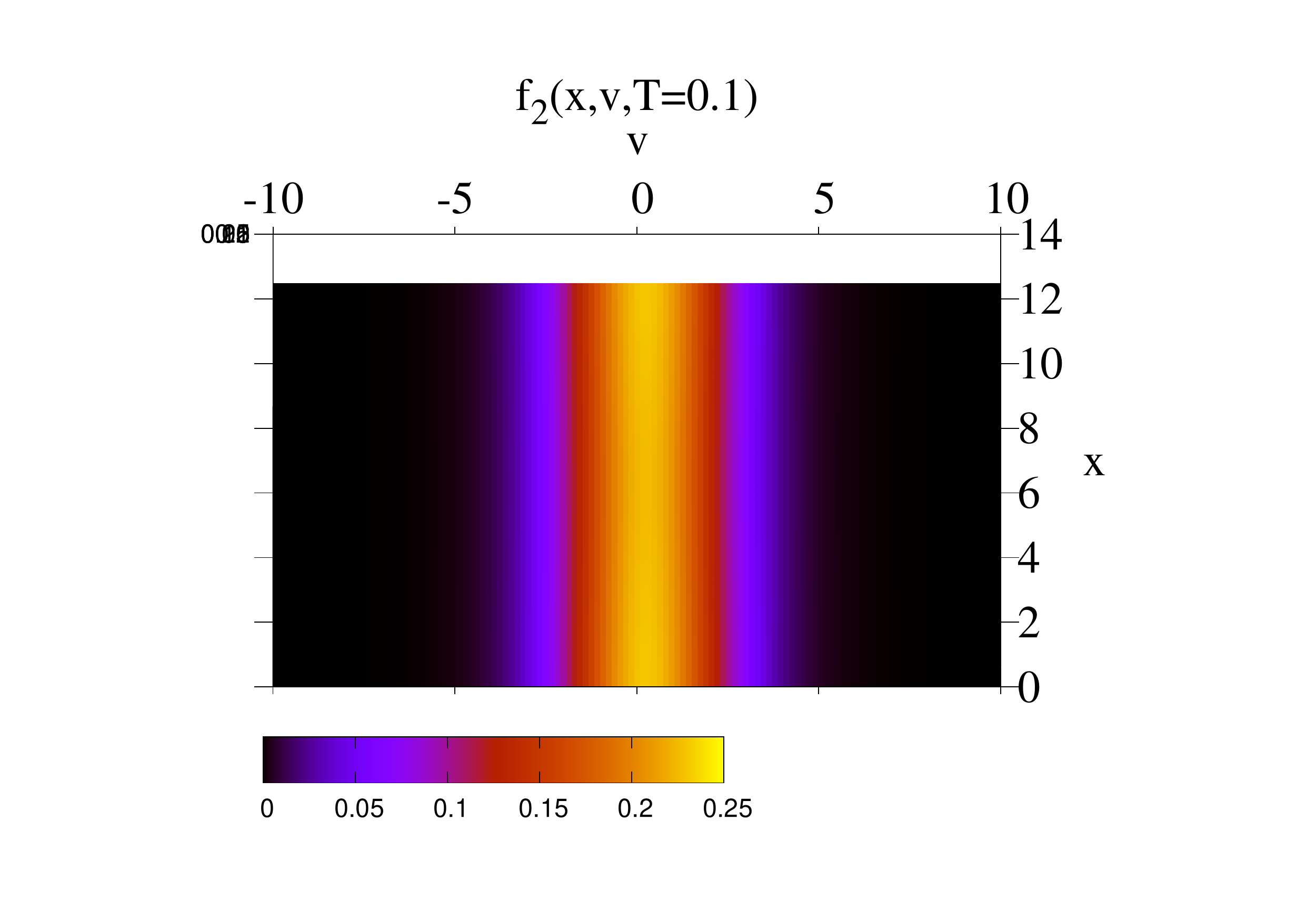}
\includegraphics[angle=-90,trim=3cm 4cm 4cm 3.5cm,width=0.33\textwidth]{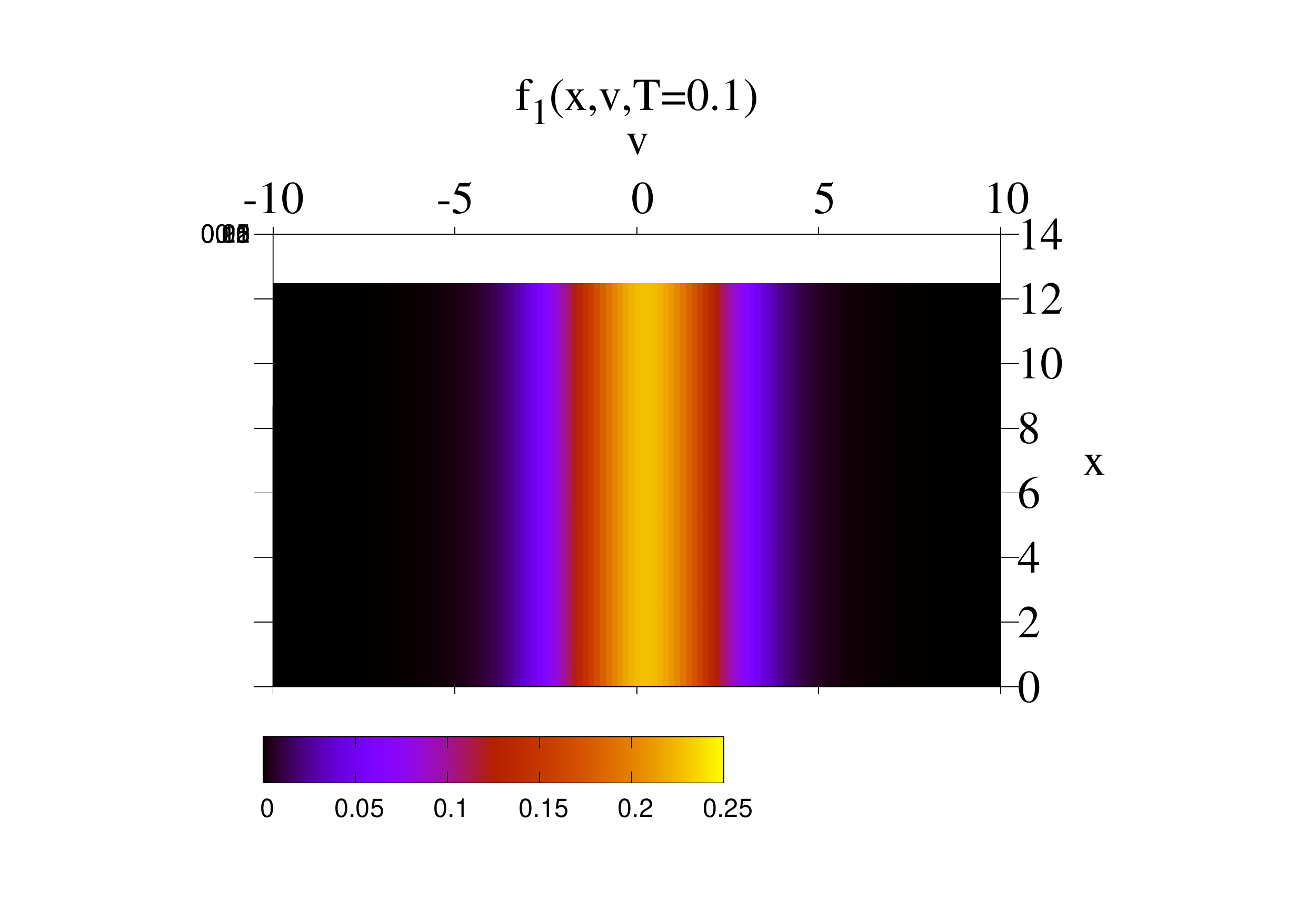}
\includegraphics[angle=-90,trim=2cm 0cm 0cm 1cm,width=0.32\textwidth]{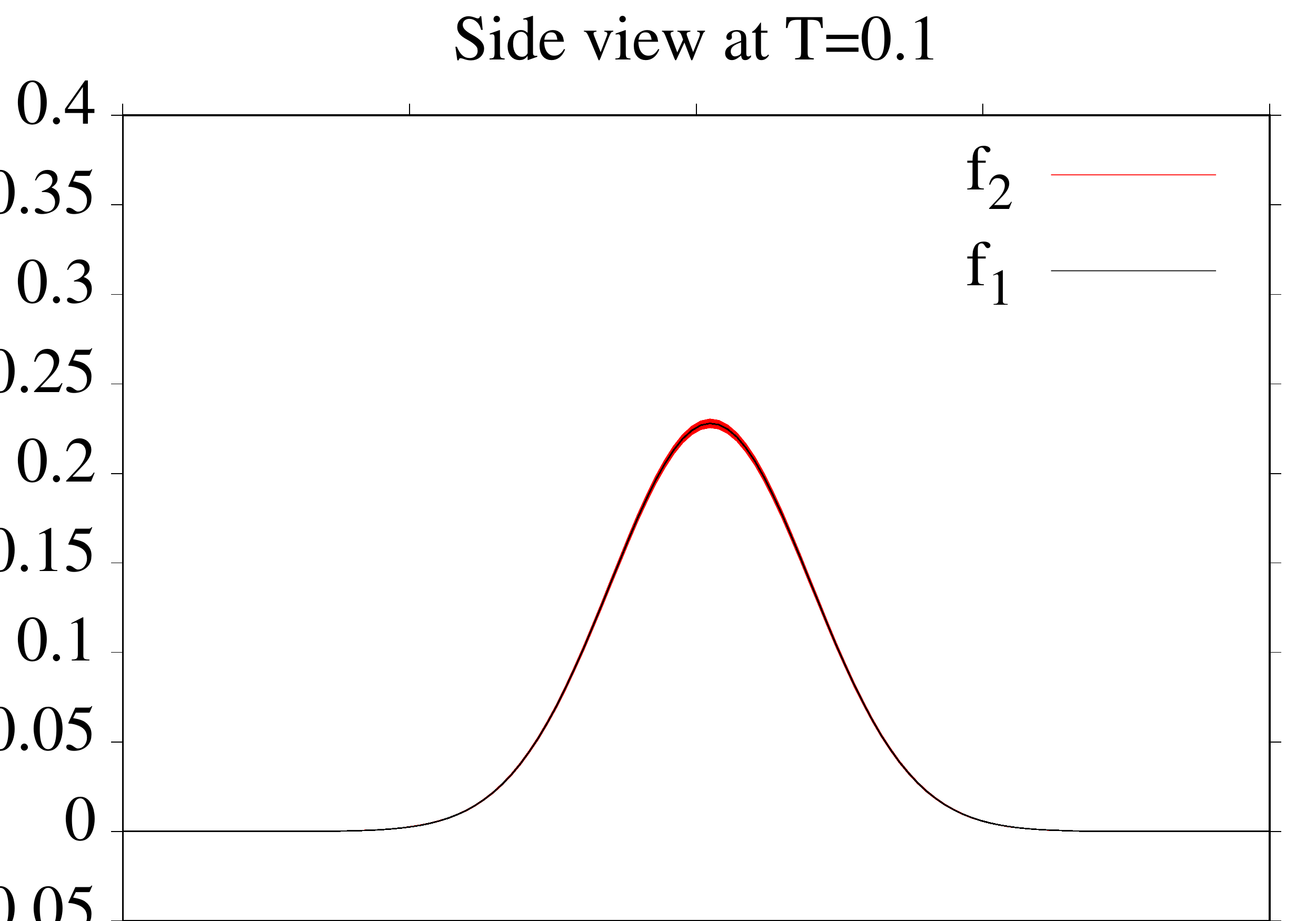}\vspace{0.1cm}
\caption{General case, $\beta=10^{-2}$, $\varepsilon_1=\varepsilon_2=\tilde{\varepsilon_1}=\tilde{\varepsilon_2}=10^{-2}$. Distribution functions at time $T=0.1$: $f_2(x,v,T)$ in phase-space (left), $f_1(x,v,T)$ in phase-space (middle), side view of $f_2(x,v,T)$ and $f_1(x,v,T)$ (right).}
\label{fig:1e-2_T0p1}
\end{figure}

We can see that the distribution functions are very close from each other at $T=0.1$. The evolution in time of $||u_1(x,t)-u_2(x,t)||_\infty$ and $||T_1(x,t)-T_2(x,t)||_\infty$, presented on figure \ref{fig:1e-2_vte}, confirms the convergence of velocities and temperatures.

 \begin{figure}[ht]
 \begin{center}
\includegraphics[angle=-90,width=0.49\textwidth]{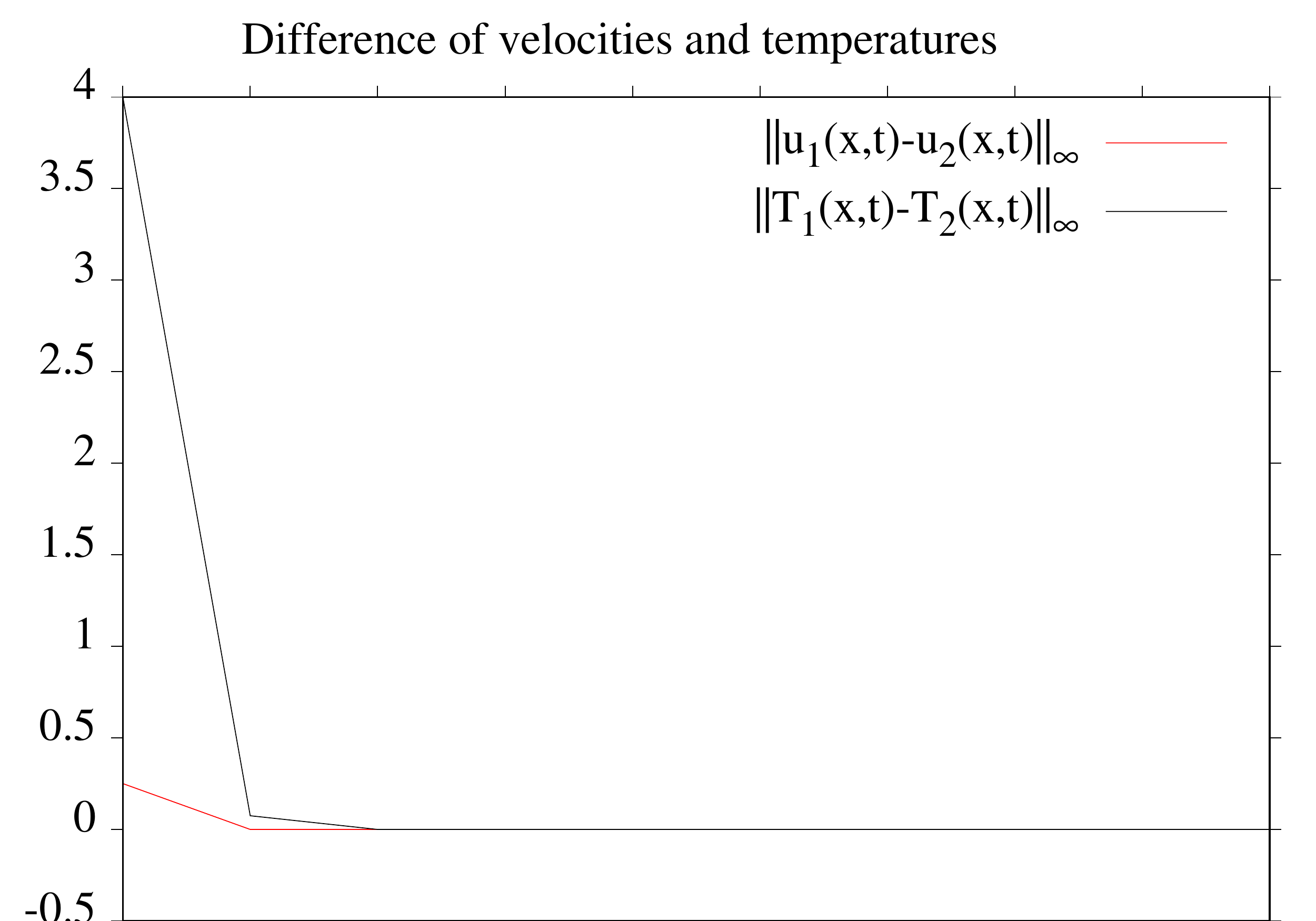}
\caption{General case, $\beta=10^{-2}$, $\varepsilon_1=\varepsilon_2=\tilde{\varepsilon_1}=\tilde{\varepsilon_2}=10^{-2}$. Evolution in time of $||u_1(x,t)-u_2(x,t)||_\infty$ and $||T_1(x,t)-T_2(x,t)||_\infty$.
}
\label{fig:1e-2_vte}
\end{center}
\end{figure}

Now, we propose a testcase in which the collisions between particles of the same species are frequent, whereas collisions between species 1 and species 2 are infrequent. More precisely, we take $\textcolor{black}{\beta}=10^{-2}$, \textcolor{black}{$N_{p_2}=N_{p_1}=5\cdot 10^5$}, $N_x=128$, $\Delta t=10^{-2}$, $\varepsilon_1=\varepsilon_2=10^{-2}$ and $\tilde{\varepsilon_1}=\tilde{\varepsilon_2}=1000$. Distribution functions are presented on figure \ref{fig:1000_1e-2_T0p01} at time $T=0.01$ and then on figure \ref{fig:1000_1e-2_T6} at time $T=6$.

\begin{figure}[ht]
\includegraphics[angle=-90,trim=3cm 4cm 4cm 3.5cm,width=0.33\textwidth]{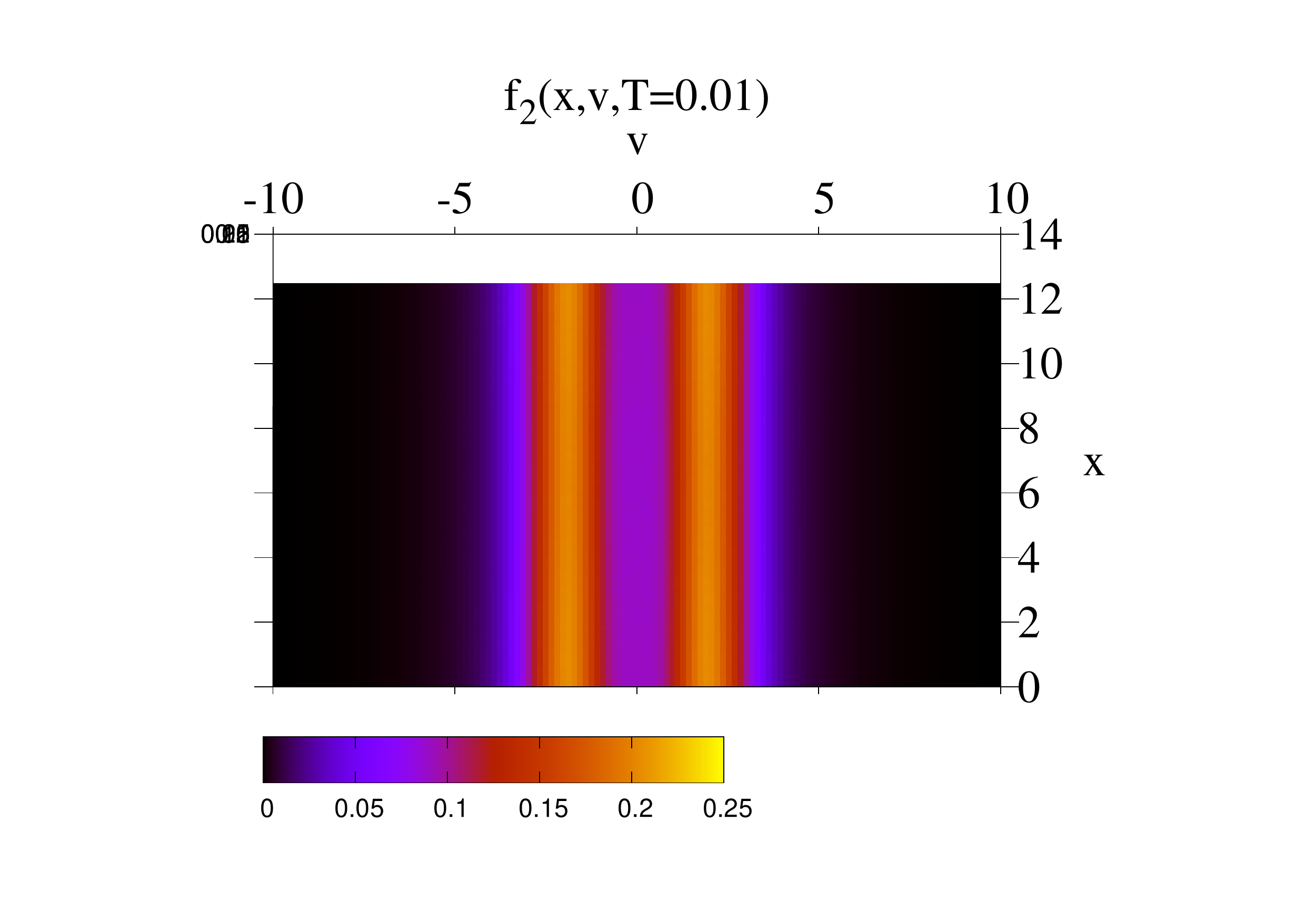}
\includegraphics[angle=-90,trim=3cm 4cm 4cm 3.5cm,width=0.33\textwidth]{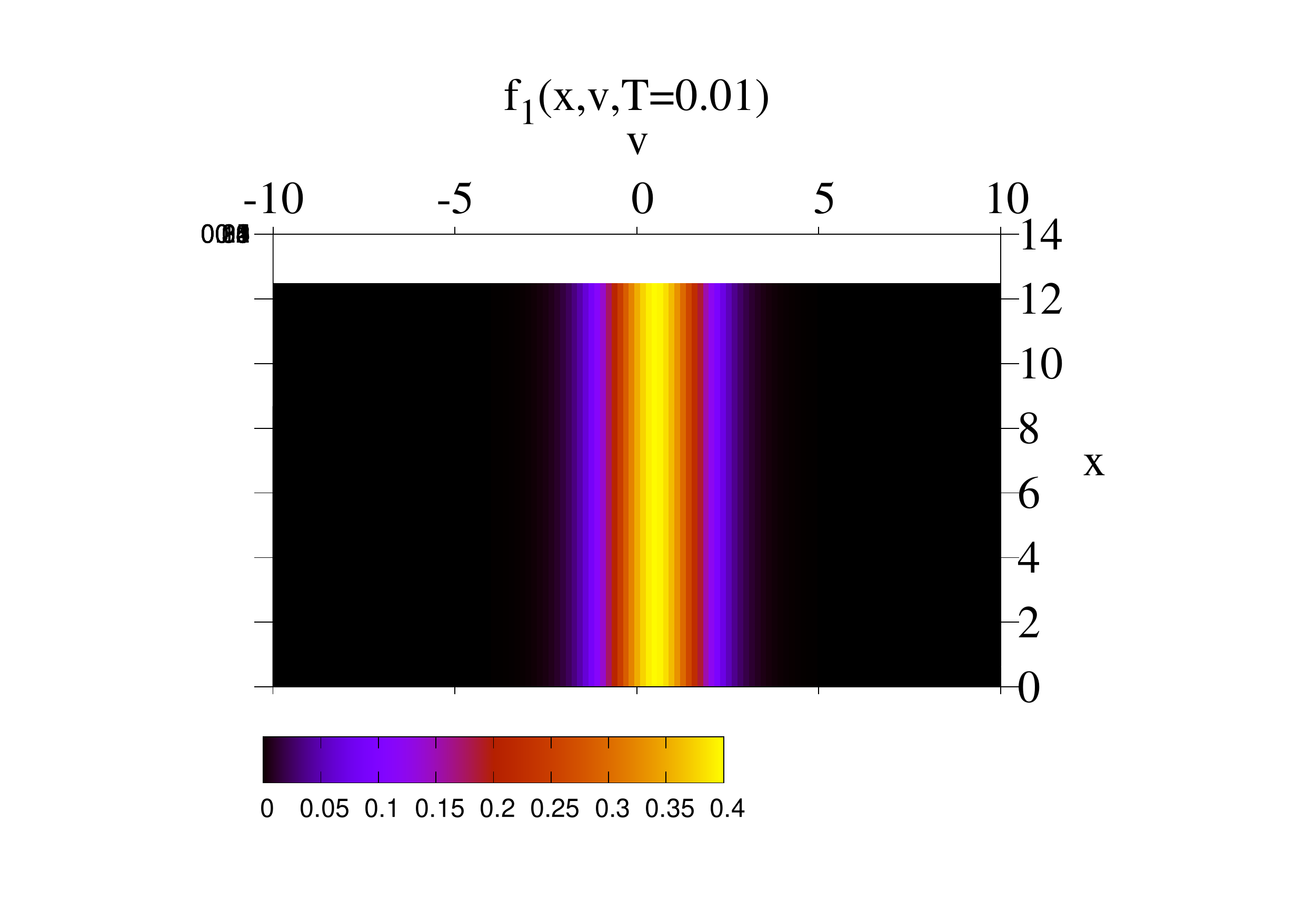}
\includegraphics[angle=-90,trim=2cm 0cm 0cm 1cm,width=0.32\textwidth]{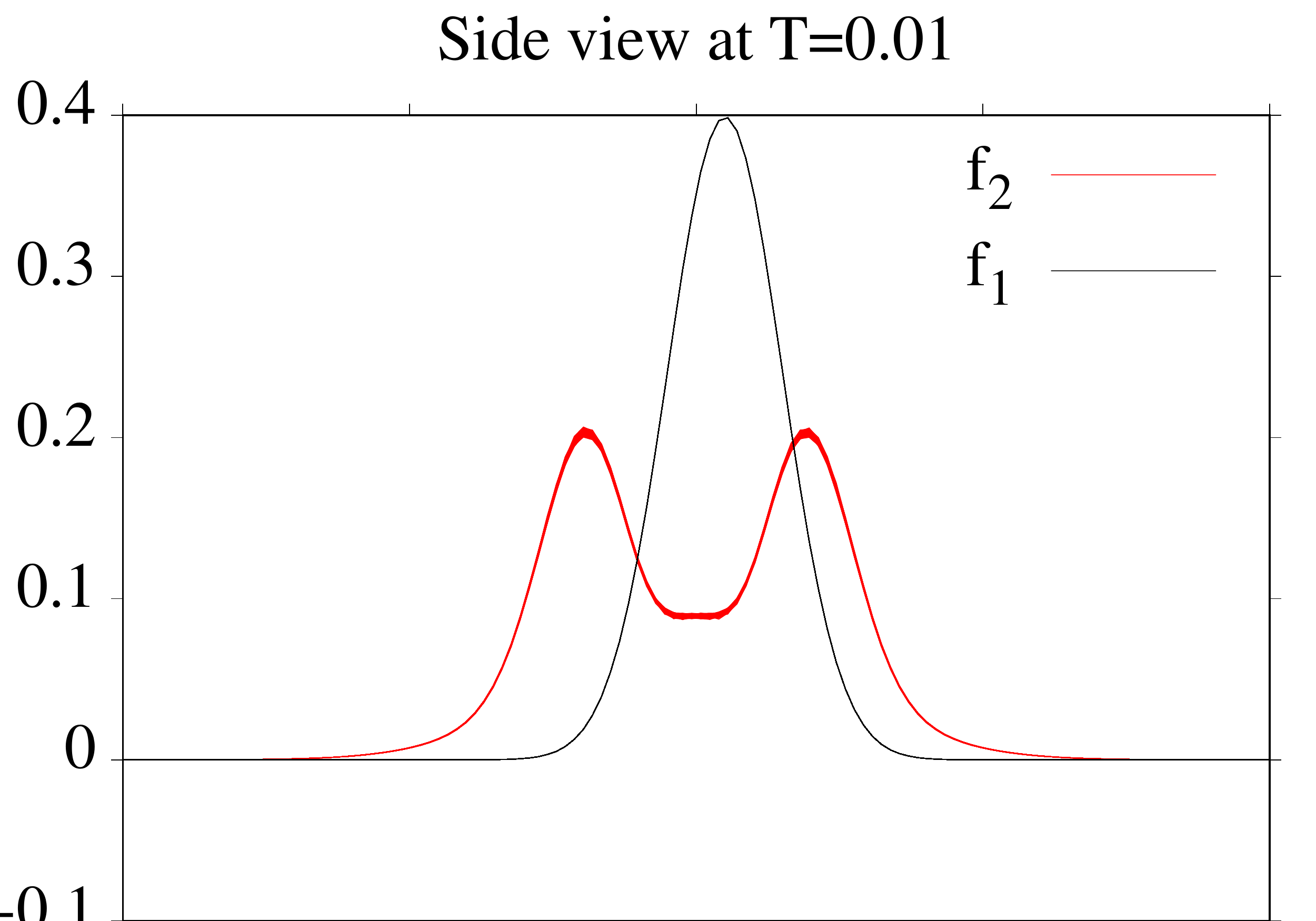}\vspace{0.1cm}
\caption{General case, $\beta=10^{-2}$, $\varepsilon_1=\varepsilon_2=10^{-2}$, $\tilde{\varepsilon_1}=\tilde{\varepsilon_2}=1000$. Distribution functions at time $T=0.01$: $f_2(x,v,T)$ in phase-space (left), $f_1(x,v,T)$ in phase-space (middle), side view of $f_2(x,v,T)$ and $f_1(x,v,T)$ (right).}
\label{fig:1000_1e-2_T0p01}
\end{figure}

\begin{figure}[ht]
\includegraphics[angle=-90,trim=3cm 4cm 4cm 3.5cm,width=0.33\textwidth]{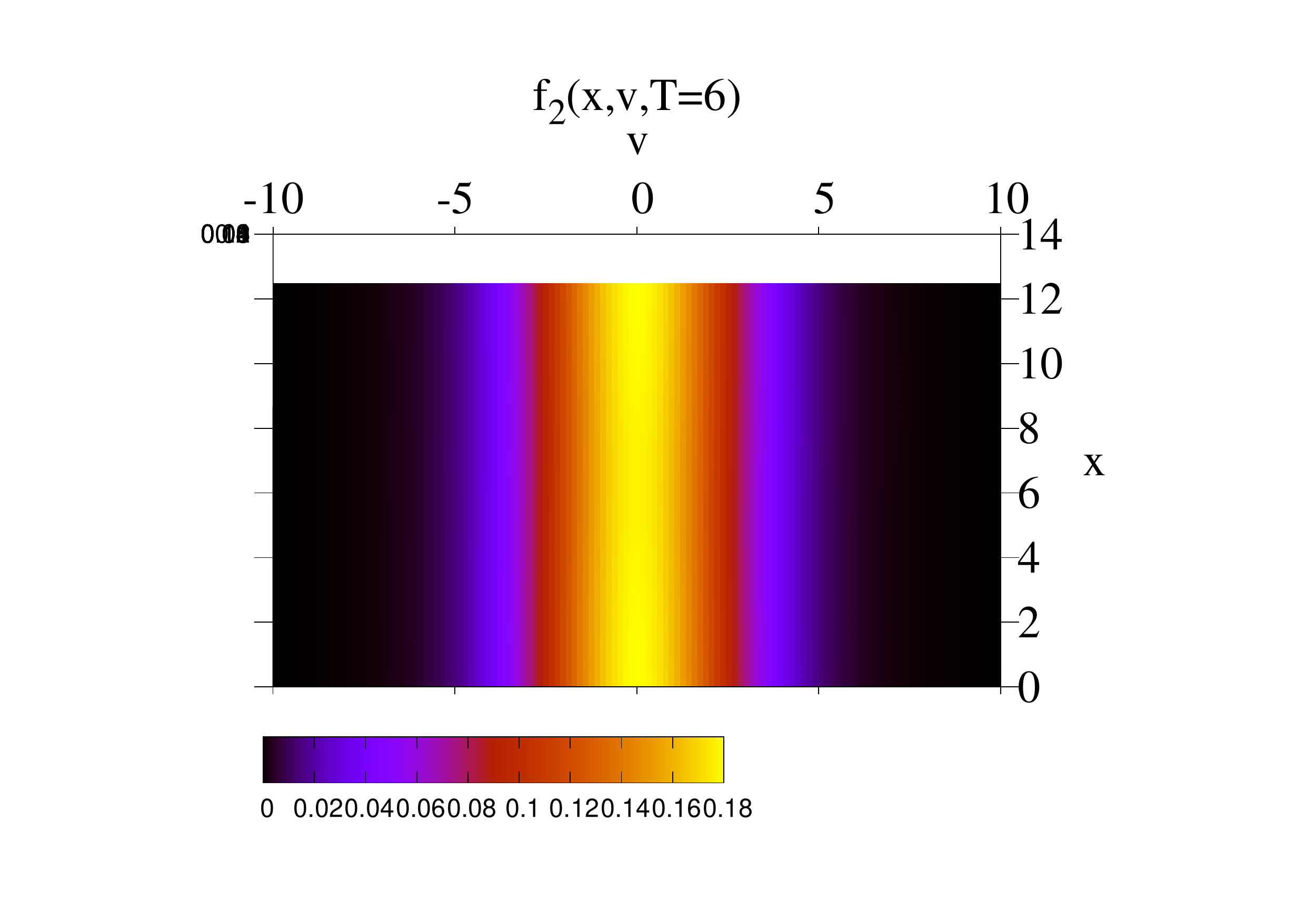}
\includegraphics[angle=-90,trim=3cm 4cm 4cm 3.5cm,width=0.33\textwidth]{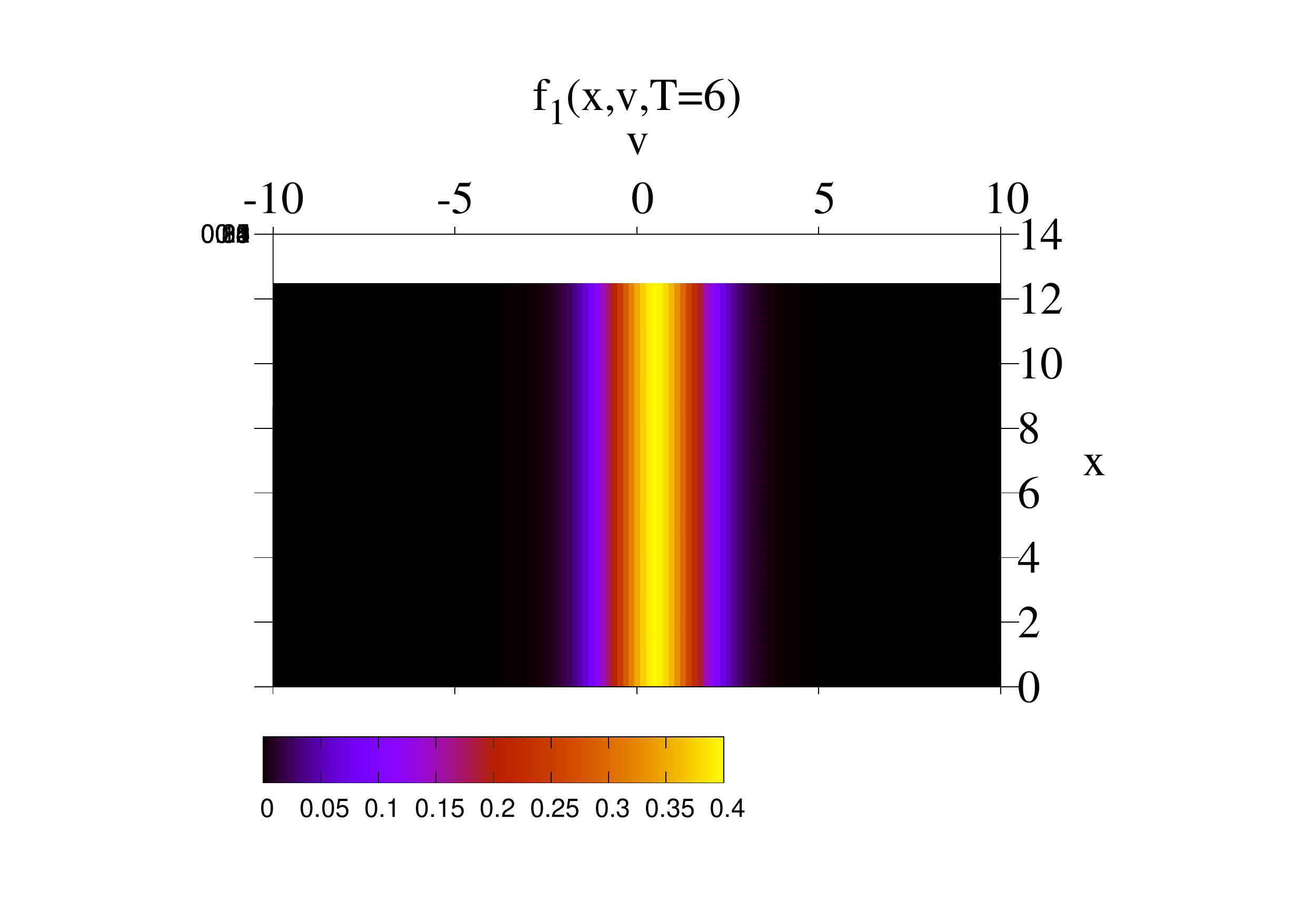}
\includegraphics[angle=-90,trim=2cm 0cm 0cm 1cm,width=0.32\textwidth]{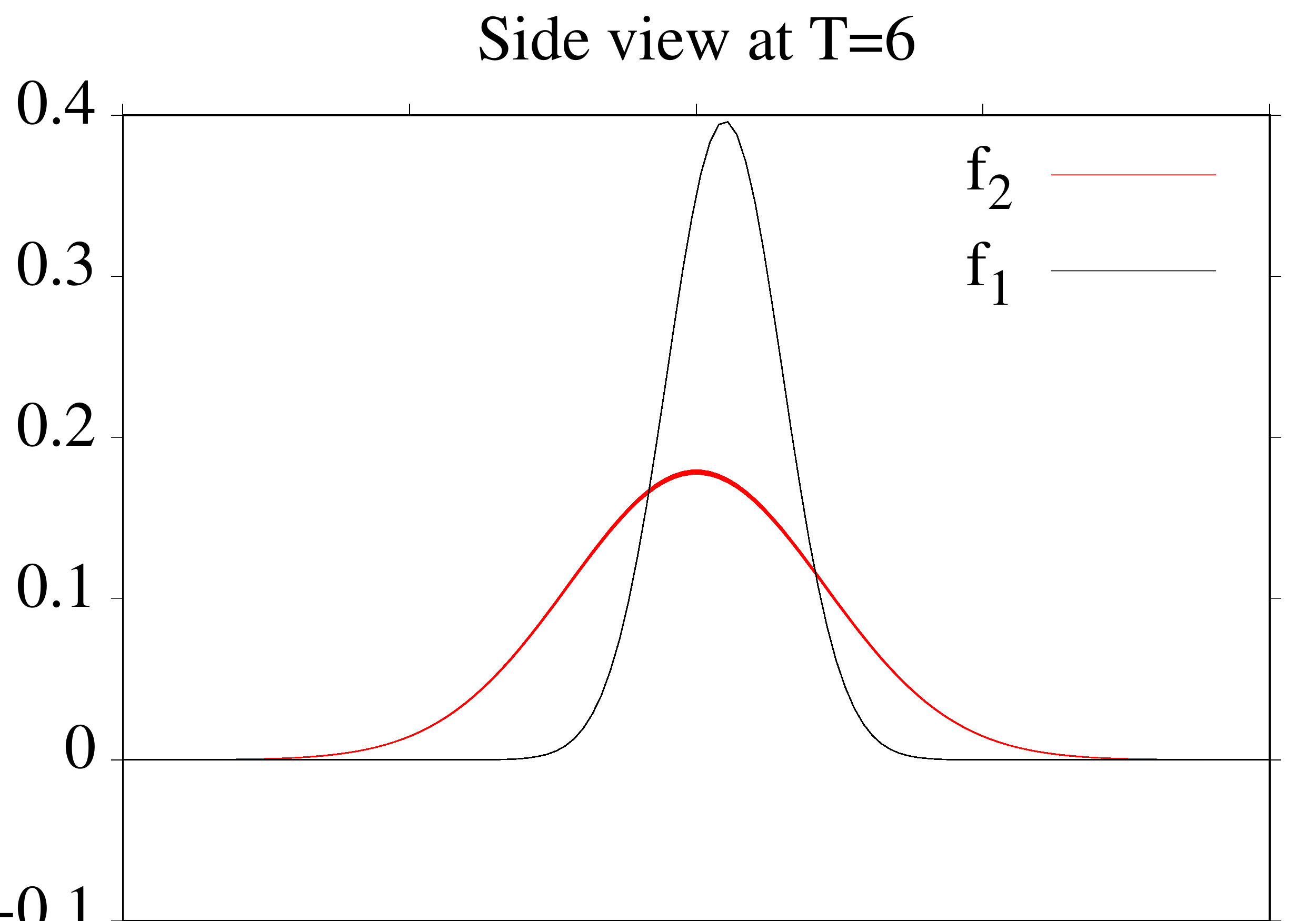}\vspace{0.1cm}
\caption{General case, $\beta=10^{-2}$, $\varepsilon_1=\varepsilon_2=10^{-2}$, $\tilde{\varepsilon_1}=\tilde{\varepsilon_2}=1000$. Distribution functions at time $T=6$: $f_2(x,v,T)$ in phase-space (left), $f_1(x,v,T)$ in phase-space (middle), side view of $f_2(x,v,T)$ and $f_1(x,v,T)$ (right).}
\label{fig:1000_1e-2_T6}
\end{figure}

Species 2 tend to have a Maxwellian distribution function, but collisions between them and species 1 are to infrequent to bring the system to a global equilibrium, at least at time $T=6$. The evolution of $||u_1(x,t)-u_2(x,t)||_\infty$ and $||T_1(x,t)-T_2(x,t)||_\infty$ is presented on figure \ref{fig:1000_1e-2_vte}. 

\begin{figure}[!htb]
\begin{center}
\includegraphics[angle=-90,width=0.49\textwidth]{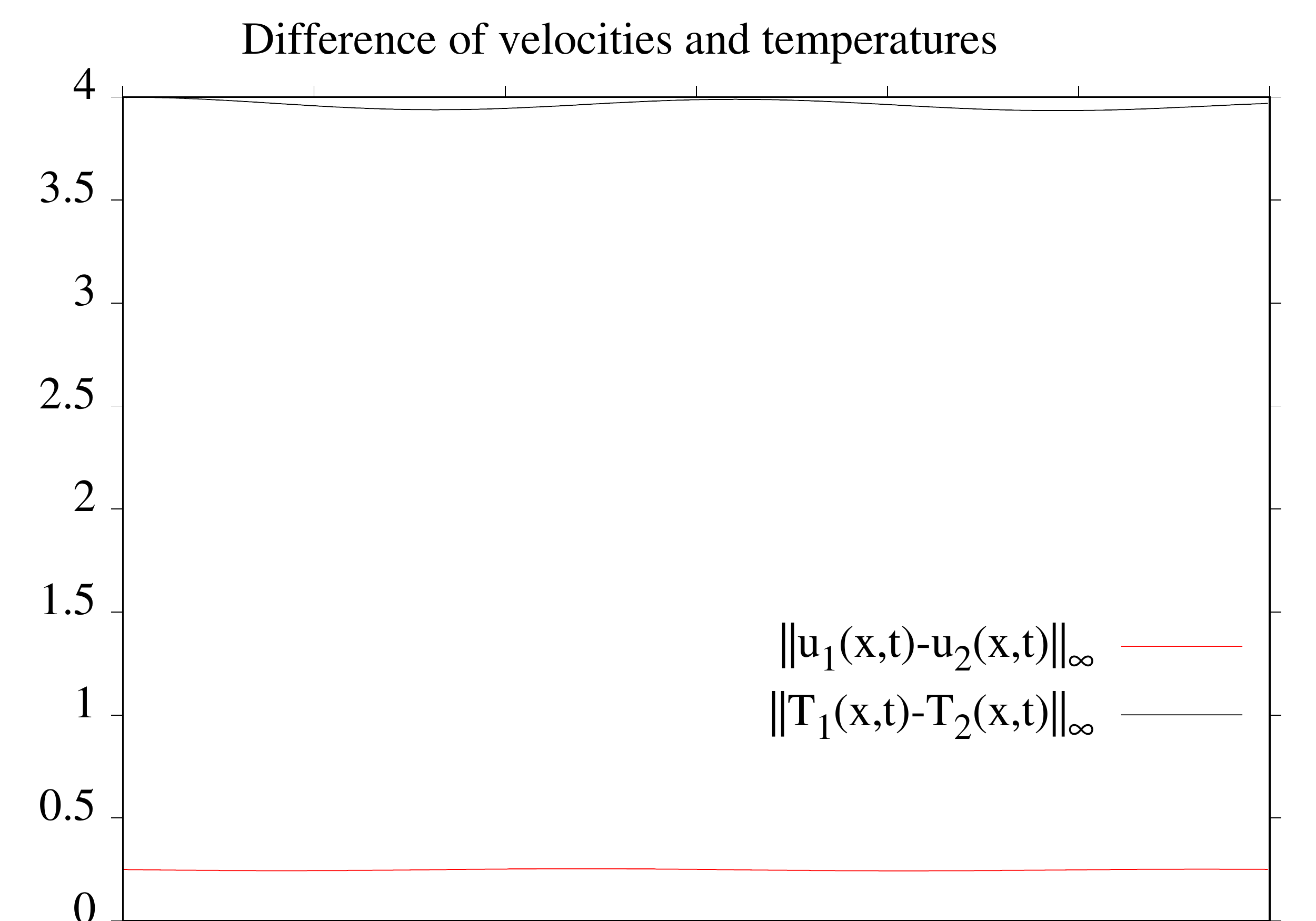}
\caption{General case, $\beta=10^{-2}$, $\varepsilon_1=\varepsilon_2=10^{-2}$, $\tilde{\varepsilon_1}=\tilde{\varepsilon_2}=1000$. Evolution in time of $||u_1(x,t)-u_2(x,t)||_\infty$ and $||T_1(x,t)-T_2(x,t)||_\infty$. 
}
\label{fig:1000_1e-2_vte}
\end{center}
\end{figure}

\textcolor{black}{Finally, we would like to highlight the main advantage of the micro-macro approach considered here: it requires a lower number of particles when approaching the equilibrium. We propose to reproduce the last experiment with only $N_{p_2}=N_{p_1}=5\cdot 10^3$ particles. Side views of the reconstructed distribution functions are presented at time $T=0.01$ on figure \ref{fig:1000_1e-2_Np} left and at time $T=6$ on figure \ref{fig:1000_1e-2_Np} right.}

\begin{figure}[!htb]
\begin{center}
\includegraphics[angle=-90,width=0.49\textwidth]{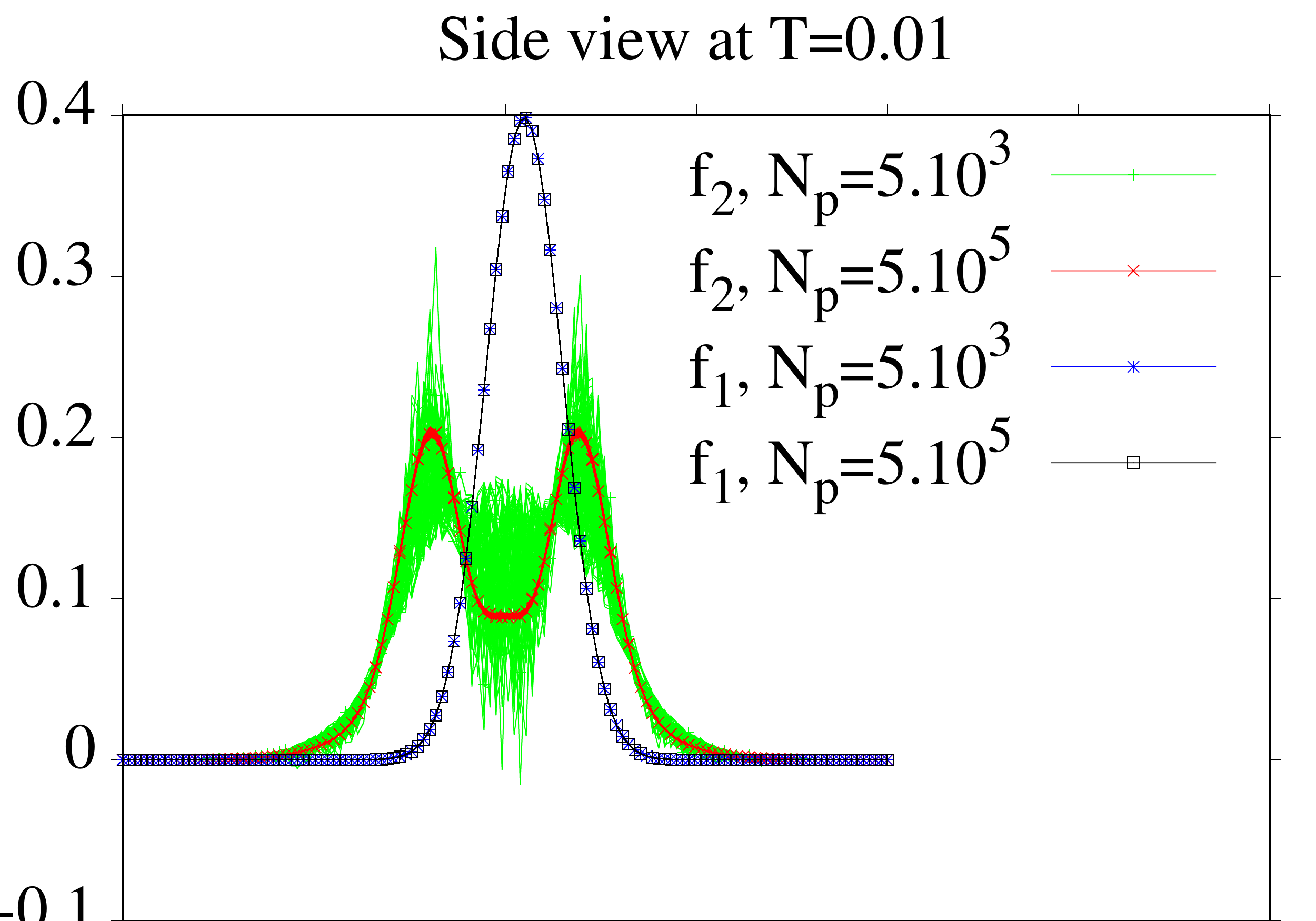}
\includegraphics[angle=-90,width=0.49\textwidth]{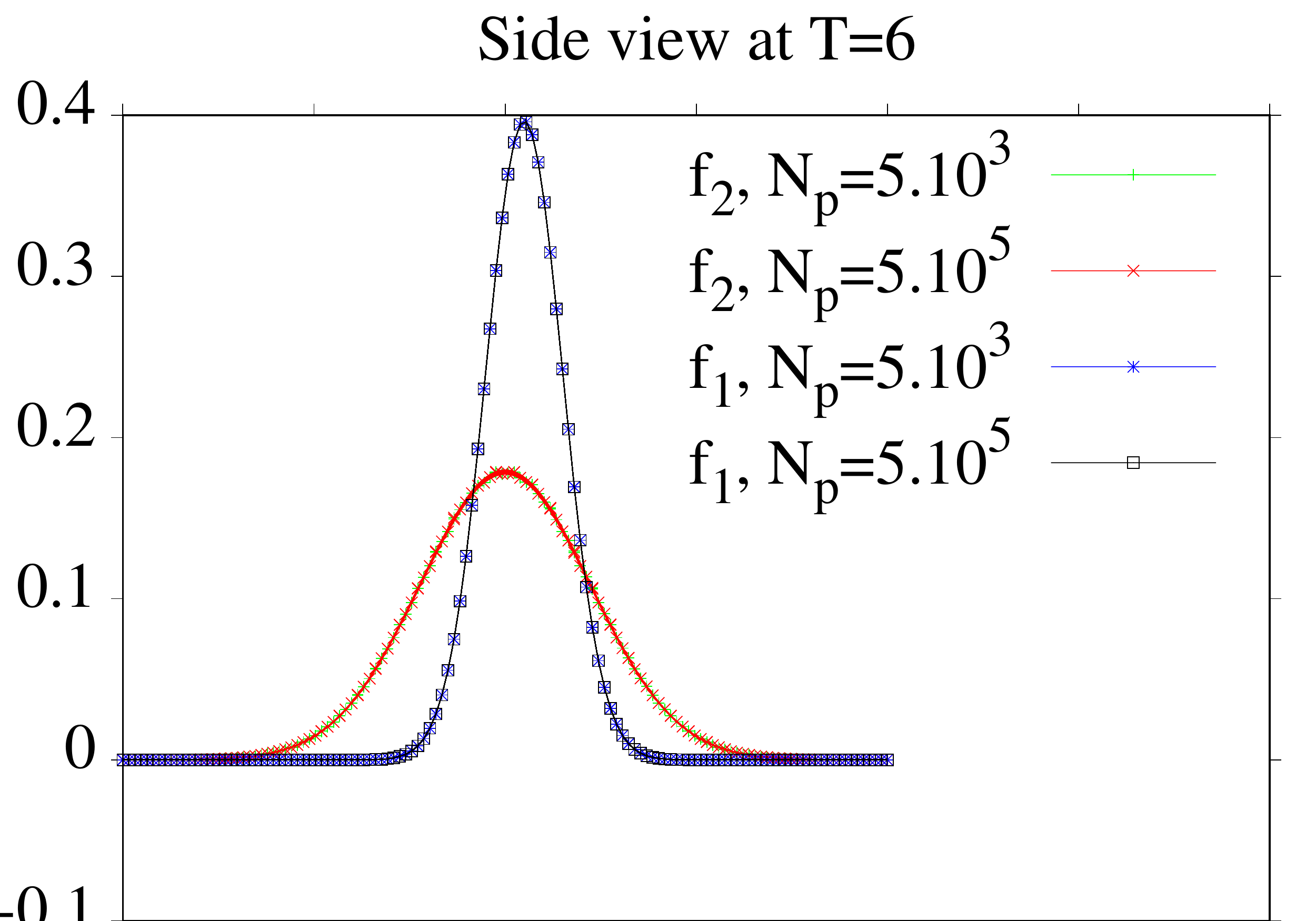}\vspace{0.2cm}
\caption{General case, $\beta=10^{-2}$, $\varepsilon_1=\varepsilon_2=10^{-2}$, $\tilde{\varepsilon_1}=\tilde{\varepsilon_2}=1000$. Side views of $f_2(x,v,T)$ and $f_1(x,v,T)$ at time $0.01$ (left) and $6$ (right). Influence of the number of particles.
}
\label{fig:1000_1e-2_Np}
\end{center}
\end{figure}

\textcolor{black}{The numerical noise that we see on the distribution $f_2$ on figure \ref{fig:1000_1e-2_Np} left means that there is not enough particles initially to represent in a good way $g_{22}$. Indeed, this quantity is big at $T=0$ since $f_2$ is far from an equilibrium. But $f_2$ goes fast towards a Maxwellian, so that $g_{22}$ becomes small and $N_{p_2}=5\times 10^{3}$ particles is then sufficient. This explains why this noise is no longer perceptible as time goes by, for instance at time $T=6$ as we can see on figure \ref{fig:1000_1e-2_Np} right. Moreover, the experiment with $5\times 10^{3}$ particles gives very good results at time $T=6$, similar to the simulation with $5\times 10^{5}$ particles. Of course, this property leads to a reduction of the numerical cost of the method when we are close to equilibrium states.}

Let us remark that in a full particle method on $f_2$ and $f_1$ (in a model without micro-macro decomposition), many more particles are necessary, since the distribution functions $f_2$ and $f_1$ keep the same order of magnitude as time goes by. So the cost of a full particle method is constant with respect to the collision frequencies. On the contrary, the cost of our micro-macro model is reduced when $\varepsilon_2$ and $\varepsilon_1$ decrease.

\section{Conclusion}\label{sec:conclusion}
In this paper, we first present a new model for a two species 1D Vlasov-BGK system based on a micro-macro decomposition. This one, derived from \cite{Pirner}, separates the intra and interspecies collision frequencies. Thus, the convergence of the system towards a global equilibrium can, depending on the values of the collision frequencies, be separated into two steps: the convergence towards the own equilibrium of each species and then towards the global one. Moreover, in the space-homogeneous case, we estimate the convergence rate of the distribution functions towards the equilibrium, as well as the convergence rate of the velocities (resp. temperatures) towards the same value.

Then, we derive a scheme using a particle method for the kinetic micro part and a standard finite volume method for the fluid macro part. In the space-homogeneous case, we illustrate numerically the convergence rates of velocities and temperatures and verify that it is in accordance with the estimations.  Finally, in the general case, we propose testcases to see the evolution in time of the distribution functions and their convergence towards equilibrium. The main advantage of this particle micro-macro approach is the reduction of the numerical cost, especially in the fuid limit, where few particles are sufficient.

\textcolor{black}{Finally, let us remark that the here presented model can be enriched by considering a transport in the velocity direction, induced for example by an electric field. The numerical method can easily be extended to this case, and no major issue would appear. For the sake of simplicity, we have not considered this case in this paper, but we have obtained encouraging results for our testcases.}

\section*{Acknowledgments}
The authors would like to thanks Eric Sonnendr\"ucker for useful discussions and suggestions about this paper.

This work has been supported by the PHC Procope DAAD Program and by the SCIAS Fellowship Program.
Moreover, Ana\"is Crestetto is supported by the French ANR project ACHYLLES ANR-14-CE25-0001 and Marlies Pirner is supported by the German Priority Program 1648,  the Austrian Science Fund (FWF) project F 65 and the Humboldt foundation.

\end{document}